\documentclass[11pt,reqno]{amsart}

\usepackage[latin1]{inputenc}
\usepackage{amsmath,amsthm,amssymb,amscd,mathrsfs}
\usepackage{mathtools,color,esint,bm}
\usepackage{booktabs}
\usepackage[shortlabels]{enumitem}
\usepackage[perpage]{footmisc}
\usepackage{subfig,caption}

\definecolor{darkgreen}{rgb}{0.0, 0.6, 0.13}

\usepackage{hyperref}

\newtheorem{thm}{Theorem}[section]
 
 \newtheorem{lem}[thm]{Lemma}
 \newtheorem{prop}[thm]{Proposition}

 \theoremstyle{definition}
 \newtheorem{df}[thm]{Definition}
 \theoremstyle{remark}
 \newtheorem{rem}[thm]{Remark}
 
 \numberwithin{equation}{section}

\newcommand{\Ab}{\mathbb A}
\newcommand{\Bb}{\mathbb B}
\newcommand{\Cb}{\mathbb C}
\newcommand{\Db}{\mathbb D}
\newcommand{\Eb}{\mathbb E}

\newcommand{\Mb}{\mathbb M}

\newcommand{\Pb}{\mathbb P}
\newcommand{\Qb}{\mathbb Q}
\newcommand{\Rb}{\mathbb R}

\newcommand{\Tb}{\mathbb T}

\newcommand{\Zb}{\mathbb Z}
\newcommand{\Ac}{\mathcal A}
\newcommand{\Bc}{\mathcal B}
\newcommand{\Cc}{\mathcal C}
\newcommand{\Dc}{\mathcal D}
\newcommand{\Ec}{\mathcal E}

\newcommand{\Gc}{\mathcal{G}}
\newcommand{\Hc}{\mathcal H}
\newcommand{\Ic}{\mathcal I}
\newcommand {\Jc}{\mathcal{J}}
\newcommand{\Kc}{\mathcal K}
\newcommand{\Lc}{\mathcal L}
\renewcommand{\Mc}{\mathcal M}
\newcommand{\Nc}{\mathcal N}
\newcommand{\Oc}{\mathcal O}
\newcommand{\Pc}{\mathcal P}
\newcommand{\Qc}{\mathcal Q}
\newcommand{\Rc}{\mathcal R}

\newcommand{\Tc}{\mathcal T}
\newcommand{\Uc}{\mathcal U}
\newcommand{\Vc}{\mathcal V}

\newcommand{\Xc}{\mathcal X}

\newcommand{\Bs}{\mathscr B}
\newcommand{\Cs}{\mathscr C}
\newcommand{\Ds}{\mathscr D}
\newcommand{\Es}{\mathscr E}
\newcommand{\Fs}{\mathscr F}

\newcommand{\Is}{\mathscr I}

\newcommand{\Ls}{\mathscr L}

\newcommand{\Os}{\mathscr O}
\newcommand{\Ps}{\mathscr P}

\newcommand{\Rs}{\mathscr R}
\newcommand{\Ss}{\mathscr S}

\newcommand{\Df}{\mathfrak D}

\newcommand{\Lf}{\mathfrak L}

\newcommand{\Rf}{\mathfrak R}
\newcommand{\Sf}{\mathfrak S}

\newcommand{\Xf}{\mathfrak X}

\newcommand{\Zf}{\mathfrak Z}

\newcommand{\bff}{\mathfrak b}
\newcommand{\cf}{\mathfrak c}

\newcommand{\ef}{\mathfrak e}
\newcommand{\ff}{\mathfrak f}

\newcommand{\lf}{\mathfrak l}
\newcommand{\mf}{\mathfrak m}
\newcommand{\nf}{\mathfrak n}

\newcommand{\pf}{\mathfrak p}
\newcommand{\qf}{\mathfrak q}
\newcommand{\rf}{\mathfrak r}

\newcommand{\dirac}{\boldsymbol{\delta}}

\newcommand{\vsigma}{\boldsymbol{\sigma}}
\newcommand{\vlambda}{\boldsymbol{\lambda}}

\begin{document}
\title{Propagation of chaos and higher order statistics in wave kinetic theory}
\author{Yu Deng}
\address{\textsc{Department of Mathematics, University of Southern California, Los Angeles, CA, USA}}
\email{\texttt{yudeng@usc.edu}}
\author{Zaher Hani}
\address{\textsc{Department of Mathematics, University of Michigan, Ann Arbor, MI, USA}}
\email{\texttt{zhani@umich.edu}}
\maketitle
\begin{abstract}
This manuscript continues and extends in various directions the result in \cite{DH}, which gave a full derivation of the wave kinetic equation (WKE) from the nonlinear Schr\"odinger (NLS) equation in dimensions $d\geq 3$. The wave kinetic equation describes the effective dynamics of the second moments of the Fourier modes of the NLS solution at the \emph{kinetic timescale}, and in the \emph{kinetic limit} in which the size of the system diverges to infinity and the strength of the nonlinearity vanishes to zero, according to a specified \emph{scaling law}.

Here, we investigate the behavior of the joint distribution of these Fourier modes and derive their effective limit dynamics at the kinetic timescale. In particular, we prove \emph{propagation of chaos} in the wave setting: initially independent Fourier modes retain this independence in the kinetic limit. Such statements are central to the formal derivations of all kinetic theories, dating back to the work of Boltzmann (Stosszahlansatz). We obtain this by deriving the asymptotics of the higher Fourier moments, which are given by solutions of the wave kinetic heirarchy (WKH) with factorized initial data. {As a byproduct, we also provide a rigorous justification of this hierarchy for general (not necessarily factorized) initial data.}

We treat both Gaussian and non-Gaussian initial distributions. In the Gaussian setting, we prove \emph{propagation of Gaussianity} as we show that the asymptotic distribution retains the Gaussianity of the initial data in the limit. In the non-Gaussian setting, we derive the limiting equations for the higher order moments, as well as for the density function (PDF) of the solution. Some of the results we prove were conjectured in the physics literature, others appear to be new. This gives a complete description of the statistics of the solutions in the kinetic limit.

\end{abstract}
\maketitle
\section{Introduction} 
 
Propagation of chaos is a central theme in all kinetic theories in statistical physics. Roughly speaking, it states that for a microscopic system with many interacting objects (particles or waves), two distinct objects should be statistically independent in the kinetic limit. Of course, this independence is not true before taking the limit, even if it is true at initial time, because naturally the dynamics produces correlations between the objects. Nonetheless, the fact that this independence is resurrected in the limit is a cornerstone of the whole kinetic description, in both particle and wave kinetic theories. In fact, almost all formal derivations of kinetic models, dating back to founding work of Boltzmann, \emph{assume} propagation of chaos to hold in order to get a closed kinetic equation for the lowest nontrivial marginal or moment of the solution. 

\smallskip

Mathematically speaking, propagation of chaos can be phrased in terms of the asymptotics of appropriate correlations or joint distributions of the solution. In wave kinetic theory, also called wave turbulence theory, these are given by the (second and higher order) moments of the Fourier modes of the solution to the dispersive equation that describes the microscopic system. If $u(t)$ is this solution, the second moment $\mathbb E|\widehat u(t, k)|^2$ is the central quantity whose asymptotics, in the kinetic limit, is given by the \emph{wave kinetic equation} (WKE), which acts as the wave analog of Boltzmann's equation. The formal derivations of this equation in the physics literature, dating back to the pioneering works of Peierls, Hasselman, and others \cite{Peierls,Hass1, Hass2, Nazarenko,NNB,ZLFBook}, are based on the unjustified assumption of propagation of chaos, which effectively allows to represent higher order mixed moments by products of second order ones, thus yielding a closed equation for the second moments. 

A rigorous derivation of the WKE at the kinetic timescale, starting from the nonlinear Schr\"odinger (NLS) equation with random initial data, has been given in our recent work \cite{DH}. This is the first result of its kind for any dispersive system (we will review some of the literature below). The derivation is done via a delicate analysis of the iterates of the NLS equation and their second order correlations, which are represented by ternary trees (and couples of such trees) often called Feynman diagrams. The analysis of such diagrams involves (a) identifying the leading order diagrams called \emph{regular couples}, (b) proving that all remaining diagrams lead to negligible contributions, and (c) controlling the remainder term in the iteration. This outline is rather simplistic; in reality there are other almost-leading diagrams whose contributions have to be analyzed separately. Moreover, the problem of estimating the diagrams is probabilistically critical in the sense of \cite{DNY3}, which is added to the factorial growth of the number of diagrams, to make the execution of this outline far from trivial. We will review some elements of that proof in Section \ref{sec:sum of DH} below, and also refer the reader to Section 3 of \cite{DH} for a more detailed exposition.

\smallskip
In particular, the proof in \cite{DH} does not require establishing propagation of chaos for the higher moments of the solution in order to obtain the effective equation for the second moment, {in sharp contrast with the earlier works that make use of the BBGKY and other similar hierarchies.} 
 This brings us to the main goal of this manuscript, which is to establish propagation of chaos and the corresponding (wave kinetic) hierarchy \emph{a posteriori} relying on the analysis introduced in \cite{DH}. Highly interesting results and unique features will appear, for the higher order statistics, depending on the initial distribution of the data, as we discuss both Gaussian and non-Gaussian initial distributions (for concreteness, only the Gaussian case was treated in \cite{DH}). In the former case, we will prove propagation of Gaussianity, which states that the asymptotic distribution of the modes remain Gaussian as it is initially. In the latter case, we will derive the limiting equations for the probability density function. We remark that this gives a \emph{complete description of the statistics} of the solutions in the kinetic limit, for both Gaussian and non-Gaussian initial distributions. 
\subsection{The kinetic setup}\label{setup} To state our results more precisely, let us first recall the wave kinetic setup starting with the microscopic system given by the cubic nonlinear Schr\"odinger equation. In dimension $d\geq 3$, we set this equation on a large torus of size $L$. The torus may be rational or irrational, which can always be rescaled to the square torus $\Tb_L^d=[0,L]^d$ but with the twisted Laplacian
\[\Delta_\beta=(2\pi)^{-1}(\beta^1\partial_1^2+\cdots+\beta^d\partial_d^2),\]
where $\beta=(\beta^1,\cdots,\beta^d)\in(\Rb^+)^d$ determines the aspect ratios of the torus. 
Consider the cubic NLS equation
\begin{equation}\label{nls}\tag{NLS}
\left\{
\begin{aligned}&(i\partial_t-\Delta_\beta)u+\lambda^2|u|^2u=0,\\
&u(0,x)=u_{\mathrm{in}}(x),
\end{aligned}
\right.
\end{equation} with random initial data $u(0)=u_{\mathrm{in}}$, and
\begin{equation}\label{data}\tag{DAT}u_{\mathrm{in}}(x)=\frac{1}{L^d}\sum_{k\in\Zb_L^d}\widehat{u_{\mathrm{in}}}(k)e^{2\pi i k\cdot x},\quad \widehat{u_{\mathrm{in}}}(k)=\sqrt{n_{\mathrm{in}}(k)}\eta_k(\omega).\end{equation} Here $\Zb_L^d=(L^{-1}\Zb)^d$, $n_{\mathrm{in}}$ is a nonnegative Schwartz function on $\Rb^d$, and $\eta_k(\omega)$ are i.i.d. random variables satisfying 
$$
\mathbb E \eta_k=0, \qquad \mathbb E|\eta_k|^2=1.
$$
This distribution of initial data will be called Gaussian if the law of each $\eta_k$ is a standard complex Gaussian, and called non-Gaussian otherwise.  
Define
\[\alpha=\lambda^2L^{-d},\quad T_{\mathrm{kin}}=\frac{1}{2\alpha^2}=\frac{1}{2}\cdot\frac{L^{2d}}{\lambda^4}.\] 

The parameter $\alpha$ stands for the strength of the nonlinearity\footnote{With overwhelming probability for large $L$, it can be shown that the size of the nonlinearity (say in $L^2$ norm) is comparable to $\alpha$. This follows from the probabilistic analysis performed in \cite{DH}, but can also be seen by simple heuristic considerations (cf. the introduction of \cite{DH}). We also note that it is common in the physics literature to use a different parametrization of the Fourier series in \eqref{data} by replacing the $L^{-d}$ factor in (\ref{data}) with $L^{-d/2}$, in which case $\alpha$ would be defined as $\lambda^2$ and $T_{\mathrm{kin}}=1/2\lambda^4$.} and $T_{\mathrm{kin}}$ is the kinetic timescale at which the NLS dynamics is approximated by that of the WKE. The \emph{kinetic limit} is taken by letting $L\to \infty$ (large box limit) and $\alpha\to 0$ (weak nonlinearity limit), according to some scaling law that specifies the relative rate of those two limits. 

The general form of a scaling law is $\alpha=L^{-\gamma}$ where $0\leq \gamma\leq \infty$, with the understanding that if $\gamma=0$ then the $\alpha\to 0$ limit is taken after the $L\to \infty$ limit, and vice versa for $\gamma=\infty$. As explained in the introduction of \cite{DH}, not all scaling laws are admissible for the kinetic theory, and the admissibility range can depend on the shape of the torus (i.e. the diophantine properties of $\beta$). Indeed, without any diophantine conditions on $\beta$, the admissible range of $\gamma$ is $0\leq \gamma\leq 1$, and one can show (e.g. \cite{CG2}) that if $\gamma>1$, then the kinetic description does not hold, for example when $\beta=(1, \ldots, 1)$. Imposing generic diophantine conditions on $\beta$ by removing a set of ``bad" vectors of zero Lebesgue measure, widens the admissible range of $\gamma$ to $0\leq \gamma \leq \frac{d}{2}$.

In \cite{DH}, we treated scaling laws of the form $\alpha=L^{-\gamma}$ with $\gamma\leq 1$ but sufficiently close to 1. When $\gamma<1$ no requirements on the shape of the torus are needed, {but for the endpoint} $\gamma=1$, the torus needs to have generic shape, i.e. $\beta$ should belong to the complement of a fixed Lebesgue null set $\Zf$ defined by a set of explicit Diophantine conditions (Lemma \ref{generic}). We remark that the approach in \cite{DH} can be used to cover the full range $\gamma\in(0,1)$; this will be addressed in a forthcoming work {under preparation}. In the current paper, for the sake of concreteness, we will stick to the setup in \cite{DH} and adopt the scaling law $\alpha=L^{-1}$, with the understanding that the result also applies to $\gamma$ smaller but sufficiently close to 1 and without any diophantine condition on $\beta$. As such, throughout the proof we will assume $\beta$ is generic in the above sense, $\lambda=L^{(d-1)/2}$, and $T_{\mathrm{kin}}=L^2/2$.

\medskip

For $0<\delta\ll 1$ depending on $n_{\mathrm{in}}$, define the solution $n=n(t,k)$, for $t\in[0,\delta]$ and $k\in\Rb^d$, to the wave kinetic equation
\begin{equation}\label{wke}\tag{WKE}
\left\{
\begin{aligned}&\partial_t n(t,k)=\Kc(n(t),n(t),n(t))(k),\\
&n(0,k)=n_{\mathrm{in}}(k),
\end{aligned}
\right.
\end{equation} where the nonlinearity
\begin{multline}\label{wkenon}\tag{KIN}\Kc(\phi_1,\phi_2,\phi_3)(k)=\int_{(\Rb^d)^3}\big\{\phi_1(k_1)\phi_2(k_2)\phi_3(k_3)-\phi_1(k)\phi_2(k_2)\phi_3(k_3)+\phi_1(k_1)\phi_2(k)\phi_3(k_3)\\-\phi_1(k_1)\phi_2(k_2)\phi_3(k)\big\}\times\dirac(k_1-k_2+k_3-k)\cdot \dirac(|k_1|_\beta^2-|k_2|_\beta^2+|k_3|_\beta^2-|k|_\beta^2)\,\mathrm{d}k_1\mathrm{d}k_2\mathrm{d}k_3.
\end{multline} Here $\dirac$ is the Dirac delta, and for $k=(k^1,\cdots,k^d)$ and $\ell=(\ell^1,\cdots,\ell^d)$ we denote
\[|k|_\beta^2:=\langle k,k\rangle_\beta,\quad \langle k,\ell\rangle_\beta:=\beta^1k^1\ell^1+\cdots+\beta^dk^d\ell^d.\]

\medskip
The following theorem is the main result of \cite{DH}, which describes the evolution of the variance $\Eb|\widehat{u}(t,k)|^2$ in the limit. Here and below, the expectation $\Eb$ is always taken under the assumption that (\ref{nls}) has a smooth solution on $[0,\delta\cdot T_{\mathrm{kin}}]$, which happens with overwhelming probability.
\begin{thm}[Theorem 1.1 of \cite{DH}]\label{main0} Fix $A\geq 40d$, $\beta\in(\Rb^+)^d\backslash\Zf$, and a function $n_{\mathrm{in}}\geq 0$ {such that
\[\|n_{\mathrm{in}}\|_{\mathfrak{S}^{40d}}:=\max_{|\alpha|,|\beta|\leq 40d}\|k^\alpha \partial_k^\beta n_{\mathrm{in}}\|_{L^2}\leq C_1<\infty.\]}Assume the law of each $\eta_k$ is Gaussian. Let $\delta$ be small enough depending on $(A,\beta,C_1)$, and $L$ be sufficiently large depending on $\delta$. Set $\lambda=L^{(d-1)/2}$ so $\alpha=L^{-1}$ and $T_{\mathrm{kin}}=L^2/2$. Then, the equation (\ref{nls}), with random initial data (\ref{data}), has a smooth solution up to time
\[T=\frac{\delta L^2}{2}=\delta\cdot T_{\mathrm{kin}},\] with probability $\geq 1-L^{-A}$. Moreover, we have
\begin{equation}\label{approximation0}\lim_{L\to\infty}\sup_{t\in[0,T]}\sup_{k\in\Zb_L^d}\left|\Eb\left|\widehat{u}(t,k)\right|^2-n\bigg(\frac{t}{T_{\mathrm{kin}}},k\bigg)\right|=0,
\end{equation} where $n(t,k)$ is the solution to (\ref{wke}).
\end{thm}
{\begin{rem} Theorem \ref{main0} is stated in \cite{DH} for Schwartz $n_{\mathrm{in}}$. {A closer look at the proof shows that} it remains true as long as $n_{\mathrm{in}}\in\Sf^{40d}$, and $\delta$ should only depend on the $\Sf^{40d}$ norm of $n_{\mathrm{in}}$; see the remarks after Theorem 1.1 in \cite{DH}. The same comment also applies to all the main results of the current paper.
\end{rem}}

\subsection{Propagation of chaos: The Gaussian case}

As mentioned above, the proof of Theorem \ref{main0} does not require obtaining asymptotics on the higher Fourier moments. 
Such information is provided in our first main result, which can be viewed as an extension of Theorem \ref{main0}.
\begin{thm}[Propagation of chaos and Gaussianity]\label{main} Under the same assumptions as Theorem \ref{main0} above, fix a positive integer $r$ and nonnegative integers $p_1,\cdots,p_r$ and $q_1,\cdots,q_r$. Then, if at least one $p_j\neq q_j\,(1\leq j\leq r)$, we have
\begin{equation}\label{approximation1}\lim_{L\to\infty}\sup_{t\in[0,T]}\sup_{\substack{(k_1,\cdots,k_r)\in(\Zb_L^{d})^r\\ k_i\neq k_j\,(\forall i\neq j)}}\bigg|\Eb\bigg(\prod_{j=1}^r\big(\widehat{u}(t,k_j)\big)^{p_j}\big(\overline{\widehat{u}(t,k_j)}\big)^{q_j}\bigg)\bigg|=0.
\end{equation} Here, as in Theorem \ref{main0}, the expectation is taken only when (\ref{nls}) has a smooth solution on $[0,T]$ where $T=\delta\cdot T_{\mathrm{kin}}$ (which has probability $\geq 1-L^{-A}$). If $p_j=q_j$ for each $1\leq j\leq r$, then we have
\begin{equation}\label{approximation2}\lim_{L\to\infty}\sup_{t\in[0,T]}\sup_{\substack{(k_1,\cdots,k_r)\in(\Zb_L^{d})^r\\ k_i\neq k_j\,(\forall i\neq j)}}\bigg|\Eb\bigg(\prod_{j=1}^r\left|\widehat{u}(t,k_j)\right|^{2p_j}\bigg)-\prod_{j=1}^r (p_j)!\cdot n\bigg(\frac{t}{T_{\mathrm{kin}}},k_j\bigg)^{p_j}\bigg|=0.
\end{equation}
\end{thm}
A key feature of Theorem \ref{main} is that, up to error terms that vanish as $L\to\infty$, we have
\begin{equation}\label{chaos}\Eb\bigg(\prod_{j=1}^r|\widehat{u}(t,k_j)|^{2p_j}\bigg)\approx\prod_{j=1}^r\Eb|\widehat{u}(t,k_j)|^{2p_j},\quad\mathrm{and\ all\ other\ moments}\approx 0.\end{equation} This means that, for fixed $t$, the random variables $\widehat{u}(t,k)$ for different $k$ become independent in the limit (at least in terms of the marginal distributions of any finitely many of them), which justifies rigorously the \emph{propagation of chaos} assumption in the literature, as described in the beginning of this paper.

Note that, these coefficients cannot be independent without taking limits, because correlations will always be produced by the nonlinear interactions in the NLS equation. Nonetheless, this independence reappears in the kinetic limit as $L\to \infty$ and $\alpha\to 0$, for the same subtle and deep reason that makes the kinetic approximation in \eqref{approximation0} hold. Namely, the only non-vanishing interactions contributing to the expectations in \eqref{approximation0}--\eqref{approximation2} are those obtained by concatenating blocks of basic interactions called $(1,1)$-mini couples and mini trees (see Figures \ref{fig:minicpl}--\ref{fig:steps}), thus forming what we call \emph{regular couples} (for second moments) or \emph{regular multi-couples} (for higher order moments, see Section \ref{ideas}). Such interactions can only be built if $p_j=q_j$ in the notation of the Theorem \ref{main}; moreover, in the higher order case, the associated structure actually decouples into second order structures, hence (\ref{chaos}) naturally occurs. The same reasoning also holds in the non-Gaussian case below (Section \ref{nongauss}), for which \eqref{chaos} remains valid.

In addition, in this Gaussian setting we have 
\begin{equation}\label{gauss}\Eb|\widehat{u}(t,k)|^{2p}\approx p!\cdot n\bigg(\frac{t}{T_{\mathrm{kin}}},k\bigg)^p,
\end{equation} which means that the law of $\widehat{u}(t,k)$ in the limit is Gaussian with variance $n\big(\frac{t}{T_{\mathrm{kin}}},k\big)$ as long as the initial state at $t=0$ is Gaussian. This has been conjectured in the physics literature under the name of \emph{propagation of Gaussianity} (see also the discussion following Theorem \ref{main3}).
\subsection{The non-Gaussian case}\label{nongauss} Highly interesting results appear in the non-Gaussian case, where unlike Theorems \ref{main0} and \ref{main}, the law of $\eta_k$ may not be Gaussian. While the second moments still follow the WKE in this setting, the non-Gaussianity of the initial law starts to exhibit itself at the higher ($\geq 4$) order moments and statistics. We will assume the law of $\eta_k$ is rotationally symmetric\footnote{Though rotation symmetry seems to be always assumed in physics literature; it would be interesting to see what happens without this assumption, in particular if (\ref{evolpdf3}) remains true. Here the loss of gauge invariance may lead to additional contributions, but probably they will be error terms in the end.}, and has exponential tails. Then, we have the following modification to Theorem \ref{main}:
\begin{thm}[Evolution of moments]
\label{main2} Suppose the i.i.d. random variables $\{\eta_k(\omega)\}$ have a law that is rotation symmetric, and satisfies that \[\mu_r:=\Eb|\eta_k|^{2r}\leq (C_0r)!,\quad \mu_1=1.\] for some constant $C_0$ {(this is equivalent to $\Eb(e^{|\eta_k|^\beta})<\infty$ for small $\beta>0$)}. Then the same limits in (\ref{approximation0}) and (\ref{approximation1}) remain true. Moreover, instead of (\ref{approximation2}), we have
\begin{equation}\label{approximation2new}\lim_{L\to\infty}\sup_{t\in[0,T]}\sup_{\substack{(k_1,\cdots,k_r)\in(\Zb_L^{d})^r\\ k_i\neq k_j\,(\forall i\neq j)}}\bigg|\Eb\bigg(\prod_{j=1}^r\left|\widehat{u}(t,k_j)\right|^{2p_j}\bigg)-\prod_{j=1}^r\mu_{p_j}\bigg(\frac{t}{T_{\mathrm{kin}}},k_j\bigg)\bigg|=0.
\end{equation} Here the functions $\mu_r(t,k)$ is defined as follows: recall $n(t,k)$ is the solution to(\ref{wke}). Let $n_0(t,k)$ be the solution to the following equation
\begin{equation}\label{wke0}\tag{WKE-0}
\left\{
\begin{aligned}&\partial_t n_0(t,k)=\Kc_0(t,k),\\
&n(0,k)=n_{\mathrm{in}}(k),
\end{aligned}
\right.
\end{equation} where
\begin{multline}\label{wkenon0}\tag{KIN-0}\Kc_0(t,k)=\int_{(\Rb^d)^3}n_0(t,k)\big\{n(t,k_1)n(t,k_3)-n(t,k_2)n(t,k_3)-n(t,k_1)n(t,k_2)\big\}\\\times\dirac(k_1-k_2+k_3-k)\cdot \dirac(|k_1|_\beta^2-|k_2|_\beta^2+|k_3|_\beta^2-|k|_\beta^2)\,\mathrm{d}k_1\mathrm{d}k_2\mathrm{d}k_3.
\end{multline} Define also $n_+(t,k)=n(t,k)-n_0(t,k)$. Then we have
\begin{equation}\label{defmur}
\mu_q(t,k)=\sum_{p=0}^q\binom{q}{p}^2(q-p)!\mu_p\cdot(n_0(t,k))^p(n_+(t,k))^{q-p}.
\end{equation} Note that if $\{\eta_k\}$ is Gaussian, then $\mu_p=p!$, so (\ref{defmur}) yields that $\mu_q(t,k)=q!(n(t,k))^q$, and we recover Theorem \ref{main}. Similarly, for $q=1$ we have $\mu_1(t,k)=n(t,k)$, so Theorem \ref{main0} remains true in the non-Gaussian case.
\end{thm}
Note that in Theorem \ref{main2} we still have (\ref{chaos}), thus propagation of chaos remains true in the non-Gaussian case. In addition, instead of (\ref{gauss}) we have $\Eb|\widehat{u}(t,k)|^{2p}\approx\mu_p\big(\frac{t}{T_{\mathrm{kin}}},k\big)$ where $\mu_p(t,k)$ is defined as in (\ref{defmur}). As far as we know, these expressions for higher order moments are new.

We remark that Theorems \ref{main} and \ref{main2} actually hold for moments whose degree (given by $\sum_{j=1}^r(p_j+q_j)$ in the notation of \eqref{approximation1}) may diverge as $L\to\infty$. Indeed, we will see in the proof that this degree can be taken as big as $\log L$ (for Theorem \ref{main}) or $\frac{\log L}{(\log \log L)^2}$ (for Theorem \ref{main2}).

\smallskip
{Under slightly stronger assumptions, Theorem \ref{main2} allows us to describe the evolution of the law of individual Fourier modes in terms of the density function, which then provides a full description of the statistics of the NLS solution in the limit.} This is summarized in our next theorem below.
\begin{thm}[Evolution of density]
\label{main3} In Theorem \ref{main2}, assume further that $\mu_r\leq C^r(2r)!$ for some constant $C$ {(this is equivalent to $\Eb(e^{\beta|\eta_k|})<\infty$ for small $\beta>0$)}. Recall the solution $n=n(t,k)$ to (\ref{wke}), and define
\begin{equation}\label{evolpdf1}\sigma_k(t)=\int_{(\Rb^d)^3}n(t,k_1)n(t,k_2)n(t,k_3)\dirac(k_1-k_2+k_3-k)\dirac(|k_1|_\beta^2-|k_2|_\beta^2+|k_3|_\beta^2-|k|_\beta^2)\,\mathrm{d}k_1\mathrm{d}k_2\mathrm{d}k_3,
\end{equation}
\begin{multline}\label{evolpdf2}\gamma_k(t)=\int_{(\Rb^d)^3}\big\{n(t,k_1)n(t,k_3)-n(t,k_2)n(t,k_3)-n(t,k_1)n(t,k_2)\big\}\\\times\dirac(k_1-k_2+k_3-k)\dirac(|k_1|_\beta^2-|k_2|_\beta^2+|k_3|_\beta^2-|k|_\beta^2)\,\mathrm{d}k_1\mathrm{d}k_2\mathrm{d}k_3.
\end{multline} Let the density function of each $\eta_k(\omega)$ be $\rho_{*}=\rho_{*}(v)$, where $v\in\Cb$ is also viewed as an $\Rb^2$ vector; assume $\rho_*$ is a radial function. Let $\rho_k=\rho_k(t,v)$ be the solution to the following linear equation
\begin{equation}\label{evolpdf3}
\left\{
\begin{aligned}\partial_t\rho_k&=\frac{\sigma_k(t)}{4}\Delta\rho_k-\frac{\gamma_k(t)}{2}\nabla\cdot(v\rho_k),\\
\rho_k(0)&=\frac{1}{n_{\mathrm{in}}(k)}\rho_*\bigg(\frac{v}{\sqrt{n_{\mathrm{in}}(k)}}\bigg).
\end{aligned}
\right.
\end{equation} Clearly each $\rho_k$ is also radial. Fix $t\in[0,\delta]$, a positive integer $r$ and distinct vectors $k_j\in\Rb^d\,(1\leq j\leq r)$. Let $k_j^{(L)}\in\Zb_L^d\,(1\leq j\leq r)$ be such that $k_j^{(L)}\to k_j\,(1\leq j\leq r)$ as $L\to\infty$, then the random variables \begin{equation}\label{evolpdf4}\big(\widehat{u}\big(t\cdot T_{\mathrm{kin}},k_1^{(L)}\big),\widehat{u}\big(t\cdot T_{\mathrm{kin}},k_2^{(L)}\big),\cdots,\widehat{u}\big(t\cdot T_{\mathrm{kin}},k_r^{(L)}\big)\big)\end{equation} converge in law, as $L\to\infty$, to the random variable with density function
\begin{equation}\label{evolpdf5}\rho_{k_1}(t,v_1)\cdot \rho_{k_2}(t,v_2)\cdots\rho_{k_r}(t,v_r).
\end{equation}
\end{thm}
{The factorization structure in (\ref{evolpdf5}) is a consequence of propagation of chaos, which has been established in Theorem \ref{main2}; thus the main feature of Theorem \ref{main3} is the evolution of the individual density (\ref{evolpdf3})}. It appears that this equation has only been discovered fairly recently in the physics literature (see \cite{JN, CLN}, and Section 6.6 of \cite{Nazarenko}).

Note that in the Gaussian case (Theorem \ref{main}) we have $\rho_*(v)=\pi^{-1}e^{-|v|^2}$. Then the solution to (\ref{evolpdf3}) equals $\rho_k(t,v)=(\pi n(t,k))^{-1}e^{-|v|^2/n(t,k)}$, so by (\ref{evolpdf5}), the limit distribution is given by independent Gaussians with variance $n(t,k)$, which provides another manifestation of the propagation of Gaussianity. Other solutions to (\ref{evolpdf3}) can be obtained and analyzed using the method of characteristics in Fourier space, see \cite{CJKN}.

\subsection{The wave kinetic hierarchy}\label{sec:wkh0}
By taking $p_j=1$ in (\ref{approximation2}) or (\ref{approximation2new}) we obtain the limits \begin{equation}\label{mixedmoment}n_r(t,k_1,\cdots,k_r):=\lim_{L\to\infty}\mathbb E \bigg( \prod_{j=1}^r\left|\widehat{u}(t,k_j)\right|^{2}\bigg).\end{equation} These limit quantities are conjectured to solve an infinite hierarchy of equations called the \emph{wave kinetic hierarchy} (WKH), which is a linear system for symmetric functions $n_r=n_r(t,k_1,\cdots,k_r)$, and has the form
\begin{equation}\label{WKH}\tag{WKH}
\begin{split}
\partial_t n_r(t, k_1, \cdots, k_r)&=\sum_{j=1}^r\int_{(\Rb^d)^3}\dirac(\ell_1-\ell_2+\ell_3-k_j)\cdot \dirac(|\ell_1|_\beta^2-|\ell_2|_\beta^2+|\ell_3|_\beta^2-|k_j|_\beta^2)\,\mathrm{d}\ell_1\mathrm{d}\ell_2\mathrm{d}\ell_3\\
&\!\!\!\!\!\!\!\!\!\!\!\!\!\!\!\!\!\!\!\!\!\!\!\!\!\!\!\!\!\!\times \bigg[ n_{r+2}(t, k_1, \cdots, k_{j-1}, \ell_1, \ell_2, \ell_3, k_{j+1},\cdots, k_r)+n_{r+2}(t, k_1, \ldots, k_{j-1}, \ell_1, k_j, \ell_3, k_{j+1},\cdots, k_r)\\
& \!\!\!\!\!\!\!\!\!\!\!\!\!\!\!\!\!\!\!\!\!\!\!\!\!\!\!\!\!\!-n_{r+2}(t, k_1, \cdots, k_{j-1}, k_j, \ell_2, \ell_3, k_{j+1},\cdots, k_r)-n_{r+2}(t, k_1, \cdots, k_{j-1}, \ell_1, \ell_2, k_j ,k_{j+1},\cdots, k_r)\bigg].
\end{split}
\end{equation} This hierarchy is the analog of Boltzmann and Gross-Pitaevski hierarchies, and is formally derived in recent works such as Chibarro \emph{et al.} \cite{CDJR1,CDJR2}, Eyink-Shi \cite{EShi} and Newell-Nazarenko-Biven \cite{NNB}, though it also follows from much earlier works including the foundational work of Peierls, see \cite{Peierls, BP, Nazarenko}.

The key property of (\ref{WKH}) is \emph{factorizability}: factorized initial data of form $(n_r)_{\mathrm{in}}(k_1, \cdots, k_r)=\prod_{j=1}^r n_{{\mathrm{in}}} (k_j)$ leads to factorized solutions of form $n_r(t, k_1, \ldots, k_r)=\prod_{j=1}^r n(t,k_j)$ where $n(t,k)$ solves (\ref{wke}) with initial data $n_{\mathrm{in}}$. This follows from direct calculations together with a suitable uniqueness theorem, which is recently proved by Rosenzweig and Staffilani in \cite{RosStaff}.

In the above sense, we can view (\ref{WKH}) as a generalization of (\ref{wke}) that allows for dependent Fourier modes. Indeed, suppose the initial data $u_{\mathrm{in}}$ of (\ref{nls}) is given by (\ref{data}) with $\widehat{u_{\mathrm{in}}}(k)$ being independent for different $k$, then Theorem \ref{main2} implies that the limit (\ref{mixedmoment}) will be a factorized solution to (\ref{WKH}) with factorized initial data, which is in fact the tensor product of the solution to (\ref{wke}). However, if $u_{\mathrm{in}}$ does not have independent Fourier modes, then the initial data $(n_r)_{\mathrm{in}}(k_1,\cdots,k_r)$, i.e. (\ref{mixedmoment}) at time $0$, will not have factorized form, in which case (\ref{mixedmoment}) at time $t$ is conjectured to be a more general solution to (\ref{WKH}).

Such scenario may arise, as discussed in Section 1.3 of \cite{RosStaff}, if one considers a hybrid, or {``twice randomized data" problem of (\ref{nls}) as follows: Instead of taking $n_{\mathrm{in}}$ deterministic in \eqref{data}, we choose it randomly according to a probability measure $\zeta$ defined on the space of all nonnegative functions $n_{\mathrm{in}}$, in such a way that new random function $n_{\mathrm{in}}$ is \emph{independent} of the pre-fixed i.i.d. random variables $\{\eta_k\}$. In the case when $\eta_k$ are random phases ($\eta_k(\omega)=e^{i\theta_k(\omega)}$ with $\theta_k$ uniformly distributed on the circle), this process of randomization is referred to as ``Random Phase and Amplitude" assumption in the wave turbulence theory literature, where in this general setup different amplitudes are not necessarily independent. 

In other words, we are choosing a random initial data whose law of distribution (as a probability measure) is given by a suitable average of those specific probability measures which are laws of distribution of random data of form (\ref{data}), i.e. having independent Fourier coefficients. This averaging is achieved by first generating a random nonnegative function $n_{\mathrm{in}}$ according to the probability measure $\zeta$ on the space of all nonnegative functions, and then selecting the random initial data as (\ref{data}) with some pre-fixed i.i.d. random variables $\{\eta_k\}$. Since independent Fourier modes in (\ref{data}) correspond to factorized solutions to (\ref{WKH}), we know, using also the linearity of (\ref{WKH}), that the above process will result in a solution to (\ref{WKH})  which is an average of certain factorized solutions. These are referred to as ``super-statistical solutions" in Eyink-Shi \cite{EShi} and may provide a possible explanation of intermittency in wave turbulence.}

Just like (\ref{wke}), the rigorous derivation of (\ref{WKH}) has been an outstanding open problem. In fact these two problems are closely related; as mentioned in the beginning of this paper, there are many earlier works on similar problems that first derive the corresponding hierarchies and then restrict to factorized solutions to obtain the kinetic equations. In the wave turbulence context, such an approach is theoretically possible but has not yet been successful. Instead, we are following the \emph{exactly opposite} route: we first derive the kinetic equation (\ref{wke}) in \cite{DH}, then apply the same techniques to derive the hierarchy (\ref{WKH}) \emph{a posterori}, in the current paper. So our last main result is the rigorous derivation of (\ref{WKH}) for general non-factorized initial data, which we state as follows.
\begin{thm}[Derivation of (\ref{WKH})]
\label{main4} Fix a positive number $\mathfrak{X}>0$ and a sequence of i.i.d. random variables $\{\eta_k\}$ as in Section \ref{setup} that satisfy the requirements of Theorem \ref{main2}. Suppose $(n_r)_{\mathrm{in}}=(n_r)_{\mathrm{in}}(k_1,\cdots,k_r)$ are nonnegative symmetric functions of $k_j\in\Rb^d\,(1\leq j\leq r)$, such that \begin{equation}\label{bounded}
\|(n_r)_{\mathrm{in}}\|_{\Sf^{40d;r}}:=\sup_{|\alpha_j|,|\beta_j|\leq 40d}\big\|k_1^{\alpha_1}\cdots k_r^{\alpha_r}\partial_{k_1}^{\beta_1}\cdots\partial_{k_r}^{\beta_r}(n_r)_{\mathrm{in}}(k_1,\cdots,k_r)\big\|_{L^2}\leq C_1^r
\end{equation} for some large constant $C_1$ (note $C_1\gtrsim \Xf$). We say $(n_r)_{\mathrm{in}}$ is \emph{admissible}, if for any $r\geq 2$ we have
\begin{equation}\label{admissible}\int_{\Rb^d}(n_r)_{\mathrm{in}}(k_1,\cdots,k_r)\,\mathrm{d}k_r=\Xf\cdot (n_{r-1})_{\mathrm{in}}(k_1,\cdots,k_{r-1}),\quad \int_{\Rb^d}(n_1)_{\mathrm{in}}(k_1)\,\mathrm{d}k_1=\Xf.
\end{equation} 

Consider a probability measure $\zeta$ on the set $\Ac$ of nonnegative functions $m=m(k)$ on $\Rb^d$, which is defined by \begin{equation}\Ac:=\label{measurezeta}\bigg\{m\geq 0:\Rb^d\to \Rb,\quad \|m\|_{\Sf^{40d}}\leq C_1,\quad \int_{\Rb^d}m(k)\,\mathrm{d}k=\Xf\bigg\}.\end{equation} For this $\zeta$, consider the hybrid initial data $u_{\mathrm{in}}$ which is given by (\ref{data}), except that $n_{\mathrm{in}}$ should be replaced by $m$, which is another random variable with values in $\Ac$, such that $m$ is independent with all the $\eta_k$ and the law of $m$ is given by $\zeta$. We say $(n_r)_{\mathrm{in}}$ is \emph{hybrid}, if there exists a $\zeta$ such that for the above choice of $u_{\mathrm{in}}$, it holds that
\begin{equation}\label{initcorr}\Eb\bigg(\prod_{j=1}^r|\widehat{u_{\mathrm{in}}}(k_j)|^2\bigg)=(n_r)_{\mathrm{in}}(k_1,\cdots,k_r)
\end{equation} for any $L$ and any distinct $k_j\in\Zb_L^d\,(1\leq j\leq r)$.

Let $T=\delta\cdot T_{\mathrm{kin}}$ where $\delta$ is as in Theorem \ref{main0} (except $C_1$ is now defined by (\ref{bounded})); the other parameters are as in Theorem \ref{main0}. Then we have the followings.

\begin{enumerate}
\item The sequence $(n_r)_{\mathrm{in}}$ is hybrid if and only if it is admissible; in this case the measure $\zeta$ is unique.
\item Assume $(n_r)_{\mathrm{in}}$ is admissible. Then with the hybrid initial data defined above, the equation (\ref{nls}) has a smooth solution up to time $T$ with probability $\geq 1-L^{-A}$. Moreover, for any fixed $r$ we have
\begin{equation}\label{approximationwkh}\lim_{L\to\infty}\sup_{t\in[0,T]}\sup_{\substack{(k_1,\cdots,k_r)\in(\Zb_L^{d})^r\\ k_i\neq k_j\,(\forall i\neq j)}}\bigg|\Eb\bigg(\prod_{j=1}^r\left|\widehat{u}(t,k_j)\right|^{2}\bigg)-n_r\bigg(\frac{t}{T_{\mathrm{kin}}},k_1,\cdots,k_r\bigg)\bigg|=0,
\end{equation} where $n_r(t,k_1,\cdots,k_r)$ is the unique solution to (\ref{WKH}) constructed in \cite{RosStaff} with initial data $(n_r)_{\mathrm{in}}$. For any $0\leq t\leq \delta$, this solution $(n_r)(t)$ is admissible in the sense of (\ref{admissible}) for the same $\Xf$.
\end{enumerate}
\end{thm}
We make two remarks regarding Theorem \ref{main4}. First, the $\Sf^{40d;r}$ norms defined in (\ref{bounded}) are much stronger than the $\Lf_{s,\epsilon}^\infty$ norms defined in \cite{RosStaff}, because of the strong $\Sf^{40d}$ norm used in Theorem \ref{main0}. It may be possible to relax this regularity assumption to match \cite{RosStaff}, but this requires refining the proof of Theorem \ref{main0} (and Theorems \ref{main}--\ref{main3}), which we are not doing here.

Second, the admissibility requirement (\ref{admissible}) seems natural in view of the conclusion (1): anything that actually arises from these hybrid initial data must be admissible. Non-admissible solutions to (\ref{WKH}) do exist, but they are probably not physically meaningful as pointed out in \cite{RosStaff}.
\subsection{Background literature} The proof of Theorems \ref{main}--\ref{main4} are based on the framework introduced in \cite{DH} to prove Theorem \ref{main0}. The latter work comes as a culmination of an extensive research effort over the past years to provide a rigorous justification of the wave kinetic equation starting from the nonlinear dispersive PDEs as first principle \cite{LukSpohn,BGHS2, Faou, DK1, DK2, DKMV, DH0, CG1,CG2}. This is Hilbert's sixth problem for \emph{waves}; its \emph{particle} analog is the rigorous derivation of the Boltzmann equation from Newtonian mechanics (see \cite{Grad, Lanford, GSRT, BGSRS} and references therein). We refer the reader to the introduction of \cite{DH} for a discussion of the developments leading up to it. 

{We should remark on the progress that has happened since the submission of \cite{DH}. First, we mention the work of Staffilani and Tran \cite{ST}. In this work, the authors consider a high ($\geq 14$) dimensional discrete KdV-type equation, with a Stratonovich-type stochastic multiplicative noise, which has the effect of regularly randomizing the phases of the Fourier modes. In the presence of this noise, the authors derive the associated kinetic equation at the kinetic timescale $T_{\mathrm{kin}}$ and in the scaling law $\alpha=L^{-0}$.} 
The authors also have a conditional result in the absence of the noise, which assumes that some a priori estimates hold for the solution, and they verify that these conditions are met for some more restrictive sets of initial data.

Another work in this direction is due to Ampatzoglou-Collot-Germain \cite{ACG} which considers the problem of deriving the WKE in an inhomogeneous setting. The authors derive this equation from a quadratic NLS-type equation for short (asymptotically vanishing) timescales, which, similar to \cite{DH0}, is a subcritical version of the critical setting considered here and in \cite{DH}.

Note that the works \cite{BGHS2,CG1,CG2,DH0,DH,DK1, DK2, DKMV,LukSpohn} concern cubic nonlinearities or $4$-wave interactions, while the works \cite{ACG,Faou,ST} concern quadratic nonlinearities or $3$-wave interactions. Both models represent a lot of important physical scenarios. Although the cubic case is considered in the current paper and in \cite{DH}, we believe that the quadratic case can be treated in the same way without much difference in strategy (as exhibited by \cite{ACG}).
\subsection{Idea of the proof}\label{ideas} Before discussing the main ideas, we first review the proof of Theorem \ref{main0} in \cite{DH}. The basic strategy is to perform a high order expansion of the NLS solution in Fourier space as
\begin{equation}\label{expandnls}
\widehat{u}(t,k)=\sum_{n=0}^N \Jc_n(t, k) +\Rc_N(t,k), \qquad k \in \Zb^d_L.
\end{equation}
Here, $N$ is the order of the expansion which diverges appropriately with the size $L$ of the domain, $\Jc_n$ is the $n$-th Picard iterate, and $\Rc_N$ is the remainder. The iterates $\Jc_n$ can be written as the sum of $\Jc_{\Tc}$, where $\Tc$ runs over all ternary trees that have $n$ branches; these are often called Feynman diagrams. To derive (\ref{wke}) in \cite{DH}, one has to compute the asymptotics of the second moments $\mathbb E|\widehat u(t, k)|^2$ which leads to the analysis of the correlations $\mathbb E (\Jc_{\Tc_1} \overline{\Jc_{\Tc_2}})$ for trees $\Tc_1$ and $\Tc_2$ of at most $N$ branches. These expressions naturally lead to the notion of \emph{couples} which consist of \emph{two} ternary trees whose leaves are paired to each other. The key observation is that the leading couples in the expansion take a very special form, which we call \emph{regular couples}, namely they are obtained by appropriately concatenating $(1,1)$-mini couples and mini trees (see Figures \ref{fig:minicpl}--\ref{fig:steps}). The proof in \cite{DH}, as described before, then reduces to (a) establishing the precise asymptotics of the regular couples, which is made possible by their precise, albeit highly complex, structure, (b) showing the the remaining couples are of lower order, which constitutes the heart of the proof, and (c) showing the remainder $\Rc_N$ is also of lower order. 

Now, in Theorem \ref{main}, we are interested in the higher order moments of the solutions, where the order $R$ can be arbitrarily large (or even grow to infinity with $L$). If we perform the same expansion (\ref{expandnls}), then we need to consider expressions of the form \[\mathbb E \big(\Jc_{\Tc_1}(t, k_1)^{\pm} \ldots \Jc_{\Tc_R}(t, k_R)^\pm\big)\] where, as usual, a minus superscript denotes complex conjugation. This leads to the key new concept in the current paper, which we call \emph{gardens}\footnote{This name is partly inspired by the song \emph{Spiritual Garden} of Yukari Tamura (2005).}, that are formed by $R$ trees whose leaves are paired to each other.

In the Gaussian setting of Theorem \ref{main}, gardens are the only new structures that emerge. Since $R$ can be arbitrarily large and may even grow to infinity with $L$, the analysis of gardens of $R$ trees will be a lot more complicated than that of couples of two trees. However, the methodology introduced in \cite{DH}, originally designed to treat couples, is in fact so robust that it can be extended to gardens---even for very large $R$---with some additional twists. Indeed, the leading contributions here come from those gardens that are formed by putting together $R/2$ couples (we call them \emph{multi-couples}), which can be analyzed using the results of \cite{DH}. In particular, as shown in \cite{DH}, only the \emph{regular} multi-couples, where each of the $R/2$ couples is regular, provide the top order contributions; these can be explicitly calculated as in \cite{DH} to match the desired right-hand side expressions, and the rest is of lower order.

As for the gardens that are not multi-couples, we apply the procedure of \cite{DH} (which are defined for couples but can be easily generalized to gardens) and conclude that they are of lower order (Proposition \ref{mainprop}). A few technical differences occur here (such as in combinatorics, cf. Proposition \ref{reconstruction} and Proposition 9.6 of \cite{DH}), but the most important one, which is also the reason why these terms are of lower order, comes from the structure of the molecules (see Section \ref{improvecount}) associated with such gardens. This is stated in Proposition \ref{basemol} (for comparison, we have $\chi=m$ instead of $\chi\leq m-R/2$ for multi-couples), which can be used to establish a power gain in the counting estimates (Proposition \ref{gain}, note the $m-R/2$ in the exponent), and subsequently the lower order bounds.

In the non-Gaussian setting (Theorem \ref{main2}), we need to introduce even more general structures. In fact, gardens appear from dividing the leaves of the $R$ trees as above into two-leaf pairs. In the Gaussian case, due to Isserlis' theorem, only expressions associated with gardens need to be considered; in the non-Gaussian case, we have a substitute of Isserlis' theorem (Lemma \ref{isserlisnew}), which is reminiscent of the cumulant expansions of the moments of random variables, but with the important quantitative estimates included. This leads to the notion of \emph{over-gardens} which are basically the same as gardens but allow pairings of more than two leaves. Again, in this setting, we identify the leading over-gardens (called regular ones) and prove that the complementary set is of lower order. It is here that the non-Gaussianity starts to exhibit itself, as regular overgardens contribute to the leading terms in addition to regular gardens, which explains the difference between \eqref{approximation2} and \eqref{approximation2new}.

In all the proofs above, as well as in \cite{DH}, the leading structures (regular couples, multi-couples and over-gardens) are still highly complex objects, whose number grows exponentially (rather than factorially) in their size. However, their redeeming feature is that one can write down \emph{exact} expressions for them in the kinetic limit which allows to match their contribution, order by order, with the solutions of the kinetic equations that appear in \eqref{wke}, \eqref{wke0}, or \eqref{WKH}.

Finally, Theorem \ref{main3} is a direct consequence of \eqref{approximation2} and \eqref{defmur}, and uniqueness of the moment problem in this setting (i.e. the moments uniquely define the law), see Lemma \ref{convlaw}, and Theorem \ref{main4} basically follows from averaging the results of Theorem \ref{main} in different scenarios, and applying the Hewitt-Savage theorem (see Lemma \ref{hewsav}) to represent arbitrary densities by tensor products.

\smallskip
We remark that the proof in this paper relies heavily on the notions and framework introduced in \cite{DH}. On the other hand, despite a few places where we briefly go over the results and proofs of \cite{DH}, the majority of this paper is devoted to the \emph{new} components needed in the higher order setting. In particular, the gardens we introduce are fundamental objects with important new features (such as Proposition \ref{basemol}), which will play significant roles in future studies of wave turbulence.
\subsection{Organization of the paper} The paper is organized as follows: In Section \ref{prelim} we review the setup and present some reductions to the problem. In Section \ref{sec:sum of DH}, we review the argument in \cite{DH} and the needed results from there. In Section \ref{sec:gardens}, we introduce the notion of gardens, their elementary combinatorial properties, and state the needed estimates to prove Theorem \ref{main}. These estimates are then proved in Sections \ref{irre}--\ref{sec:conclusion}. In Section \ref{sec:nongaussian} we deal with the non-Gaussian case and prove Theorems \ref{main2} and \ref{main3}, and in Section \ref{sec:wkh} we prove Theorem \ref{main4}.
\subsection{Acknowledgements} Yu Deng is supported in part by NSF grant DMS-1900251 and Sloan Fellowship. Zaher Hani is supported in part by NSF grant DMS-1654692 and a Simons Collaboration Grant on Wave Turbulence. The authors thank Sergey Nazarenko and Herbert Spohn for enlightening conversations and pointing out some references. Part of this work was done while the authors were visiting ICERM (Brown University), which they wish to thank for its hospitality. The first author thanks Matthew Rosenzweig for helpful discussions related to Theorem \ref{main4}.

\section{Preliminary reductions} \label{prelim}
\subsection{Reduction of (\ref{nls})} As in \cite{DH} we make the following reductions. Suppose $u$ is a solution to (\ref{nls}), define $a=a_k(t)$ such that
\begin{equation}\label{change}a_k(t)=e^{-\delta\pi iL^2|k|_\beta^2t}\cdot e^{-2i\lambda^2M\delta T_{\mathrm{kin}}t}\cdot\widehat{u}(\delta T_{\mathrm{kin}}\cdot t,k),\end{equation} where $M$ is the conserved mass of $u$, then it solves the equation
\begin{equation}\label{eqnak}
\left\{
\begin{aligned}\partial_ta_k&=\Cc_+(a,\overline{a},a)_k(t),\\
a_k(0)&=(a_k)_{\mathrm{in}}=\sqrt{n_{\mathrm{in}}(k)}\eta_k(\omega),
\end{aligned}
\right.
\end{equation} with the nonlinearity
\begin{equation}\label{akeqn2}\Cc_\zeta(f,g,h)(t):=\frac{\delta}{2L^{d-1}}\cdot (i\zeta)\sum_{k_1-k_2+k_3=k}\epsilon_{k_1k_2k_3}e^{\zeta\delta\pi iL^2\Omega(k_1,k_2,k_3,k)t}f_{k_1}(t)g_{k_2}(t)h_{k_3}(t)
\end{equation} for $\zeta\in\{\pm\}$. Here in (\ref{akeqn2}) and below, the summation is taken over $(k_1,k_2,k_3)\in(\Zb_L^d)^3$, and \begin{equation}\label{defcoef0}\epsilon_{k_1k_2k_3}=
\left\{
\begin{aligned}&1,&&\mathrm{if\ }k_2\not\in\{k_1,k_3\};\\
-&1,&&\mathrm{if\ }k_1=k_2=k_3;\\
&0,&&\mathrm{otherwise},
\end{aligned}
\right.\end{equation} and the resonance factor
\begin{equation}\label{res}
\Omega=\Omega(k_1,k_2,k_3,k):=|k_1|_\beta^2-|k_2|_\beta^2+|k_3|_\beta^2-|k|_\beta^2=2\langle k_1-k,k-k_3\rangle_\beta.\end{equation} Note that $\epsilon_{k_1k_2k_3}$ is always supported in the non-degenerate set \begin{equation}\label{defset}\Sf:=\big\{(k_1,k_2,k_3):\mathrm{\ either\ }k_2\not\in\{k_1,k_3\},\mathrm{\ or\ }k_1=k_2=k_3\big\}.\end{equation} Below we will focus on the system (\ref{eqnak})--(\ref{akeqn2}), with the relevant terms defined in (\ref{defcoef0})--(\ref{defset}), for time $t\in[0,1]$.
\subsection{Reduction of Theorem \ref{main}} By plugging in (\ref{change}) we can reduce Theorem \ref{main} to proving the following bounds
\begin{equation}\label{bound1}\bigg|\Eb\bigg(\prod_{j=1}^r\big(a_{k_j}(t)\big)^{p_j}\big(\overline{a_{k_j}(t)}\big)^{q_j}\bigg)\bigg|\lesssim_R L^{-\nu}
\end{equation} if $p_j\neq q_j$ for some $1\leq j\leq r$, and
\begin{equation}\label{bound2}\bigg|\Eb\bigg(\prod_{j=1}^r|a_{k_j}(t)|^{2p_j}\bigg)-\prod_{j=1}^r(p_j)!n(\delta t, k_j)^{p_j}\bigg|\lesssim_R L^{-\nu},
\end{equation} uniformly in $t\in[0,1]$ and in $k_1,\cdots,k_r\in\Zb_L^d$ satisfying $k_i\neq k_j$, with $\nu>0$ being an absolute constant and the implicit constants depending on $R$, where $R:=(p_1+\cdots+p_r+q_1+\cdots+q_r)/2$.

Note that if $a_k(t)$ solves (\ref{eqnak}) then $e^{i\theta}a_k(t)$ solves the same equation, with the initial data obeying the same law. From this it is easy to deduce that
\[\Eb\bigg(\prod_{j=1}^r\big(a_{k_j}(t)\big)^{p_j}\big(\overline{a_{k_j}(t)}\big)^{q_j}\bigg)=0,\qquad\textrm{if}\ \ p_1+\cdots+p_r\neq q_1+\cdots+q_r.\]

Below we will always assume $p_1+\cdots+p_r= q_1+\cdots+q_r=R$. As we consider the limit $L\to\infty$ with $R$ fixed, we may assume $R\leq \log L$. We shall introduce a simpler notation as follows. For $1\leq j\leq r$, take $p_j$ copies of the variable $k_j$ with sign $+$ and $q_j$ copies of the variable $k_j$ with sign $-$, and rename them as $(k_1^*,\cdots,k_{2R}^*)$ with associated signs $\zeta_j\,(1\leq j\leq 2R)$. For simplicity we will write $k_j$ instead of $k_j^*$ below. Then (\ref{bound1}) and (\ref{bound2}) result from the following unified and more precise estimate, namely
\begin{equation}\label{unibound}\bigg|\Eb\bigg(\prod_{j=1}^{2R}a_{k_j}^{\zeta_j}(t)\bigg)-\sum_\Pc\prod_{\{j,j'\}\in\Pc}\mathbf{1}_{k_j=k_{j'}}\prod_{j}^{(+)}n(\delta t,k_j)\bigg|\lesssim R!\cdot M_{\mathrm{kin}}^R\cdot L^{-\nu}.
\end{equation} Here we denote $z^+=z$ and $z^-=\overline{z}$, and the sum is taken over all partitions $\Pc$ of $\{1,\cdots,2R\}$ into two-element subsets $\{j,j'\}$ such that $\zeta_{j'}=-\zeta_j$. The first product is taken over all $\{j,j'\}\in\Pc$, and the second product is taken over all $1\leq j\leq 2R$ such that $\zeta_j=+$. Finally $M_{\mathrm{kin}}$ is defined as
\begin{equation}\label{defmkin}M_{\mathrm{kin}}:=1+\sup_{t\in[0,1],k\in\Rb^d}|n(\delta t,k)|,
\end{equation} and the implicit constant in (\ref{unibound}) depends only on $(d,\beta,n_{\mathrm{in}})$ but not on $R$.

\medskip
The goal for the rest of the paper is then to prove (\ref{unibound}).
\subsection{Parameters and notations} Most of our parameters and notations are taken from \cite{DH}. First, we fix $\beta\in(\Rb^+)^d\backslash\Zf$, where $\Zf$ is defined by the following lemma.
\begin{lem}[Lemma A.1 of \cite{DH}]\label{generic} There exists a Lebesgue null set $\Zf\subset(\Rb^+)^d$ such that the followings hold for any $\beta=(\beta^1,\cdots,\beta^d)\in(\Rb^+)^d\backslash\Zf$.
\begin{enumerate}
\item For any integers $(K^1,K^2)\neq (0,0)$, we have
\begin{equation}\label{generic1}|\beta^1K^1+\beta^2K^2|\gtrsim(1+|K^1|+|K^2|)^{-1}\log^{-4}(2+|K^1|+|K^2|);
\end{equation}
\item The numbers $\beta^1,\cdots,\beta^d$ are algebraically independent over $\Qb$, and for any $R$ we have
\begin{equation}\label{generic2}
\#\left\{(X,Y,Z)\in(\Zb^d)^3:|X|,|Y|,|Z|\leq R,\,X\neq 0,\,\max(|\langle X,Y\rangle_\beta|,|\langle X,Z\rangle_\beta|)\leq 1\right\}\lesssim R^{3d-4+\frac{1}{6}}.
\end{equation}
\end{enumerate}
\end{lem}
\begin{proof} See \cite{DH}, Lemma A.1.
\end{proof}
Throughout this paper, we will use $C$ to denote any large constant that depends only on the dimension $d$, and use $C^+$ to denote any large constant that depends on $(d,\beta,n_{\mathrm{in}})$; these may differ from line to line, and note in particular that they do \emph{not} depend on the value of $R$ in (\ref{unibound}). The notations $X\lesssim Y$ and $X=O(Y)$ will mean $X\leq C^+Y$ unless otherwise stated.

Recall that $A\geq 40d$ and $\delta$, which is small enough depending on $A$ and $C^+$, are fixed as in Theorem \ref{main}. We also fix $\nu=(100d)^{-1}\ll 1$ and define $N=\lfloor (\log L)^4\rfloor$. Note that the value of $N$ is different from the one in \cite{DH}. As  above we assume $R\leq \log L$. For later purposes we may need slightly larger values (like $2R$), but all our proofs work equally fine as long as $R\leq 2\log L$, which will be satisfied throughout the paper. Note that we do not assume any inequality between $\delta$ and $R$.

We adopt the shorthand notation $k[A]=(k_j)_{j\in A}$ and similarly for other vectors, and also define $\mathrm{d}\alpha[A]=\prod_{j\in A}\mathrm{d}\alpha_j$. We also use multi-indices $\rho$ with the usual notations. Define the time Fourier transform (the meaning of $\widehat{\cdot}$ later may depend on the context)
\[\widehat{u}(\lambda)=\int_\Rb u(t) e^{-2\pi i\lambda t}\,\mathrm{d}t,\quad u(t)=\int_\Rb \widehat{u}(\lambda)e^{2\pi i\lambda t}\,\mathrm{d}\lambda.\] Define the $X^\kappa$ norm for functions $F=F(t,k)$ or $G=G(t,s,k)$ by \[\|F\|_{X^\kappa}=\int_\Rb\langle\lambda\rangle^{\frac{1}{9}}\sup_k\langle k\rangle^{\kappa}|\widehat{F}(\lambda,k)|\,\mathrm{d}\lambda,\quad \|G\|_{X^\kappa}=\int_{\Rb^2}(\langle\lambda\rangle+\langle\mu\rangle)^{\frac{1}{9}}\sup_k\langle k\rangle^{\kappa}|\widehat{G}(\lambda,\mu,k)|\,\mathrm{d}\lambda\mathrm{d}\mu,\] where $\widehat{\cdot}$ denotes the Fourier transform in $t$ or $(t,s)$. In the case when $F$ or $G$ does not depend on $k$, this norm will not depend on $\kappa$ and will be denote by $X$. Define the localized version $X_{\mathrm{loc}}^\kappa$ (and similarly $X_{\mathrm{loc}}$) as 
\[\|F\|_{X_{\mathrm{loc}}^\kappa}=\inf\big\{\|\widetilde{F}\|_{X^\kappa}:\widetilde{F}=F\mathrm{\ for\ }0\leq t\leq 1\big\};\quad \|G\|_{X_{\mathrm{loc}}^\kappa}=\inf\big\{\|\widetilde{G}\|_{X^\kappa}:\widetilde{G}=G\mathrm{\ for\ }0\leq t,s\leq 1\big\}.\] If we will only use the value of $G$ in some set (for example $\{t>s\}$ in Proposition \ref{propreg}), then in the above definition we may only require $\widetilde{G}=G$ in this set. Define the $Z$ norm for function $a=a_k(t)$,
\begin{equation}\label{defznorm}\|a\|_Z^2=\sup_{0\leq t\leq 1}L^{-d}\sum_{k\in\Zb_L^d}\langle k\rangle^{10d}|a_k(t)|^2.
\end{equation}
\section{A brief summary of \cite{DH}}\label{sec:sum of DH} The results of this section are proved in \cite{DH}. Here we state the relevant propositions and definitions that will be needed in the proof below.
\subsection{Trees, couples, and decorations} We first recall the definitions of trees, couples, and decorations, which are drawn directly from \cite{DH}.
\begin{df}[Definition 2.1 in \cite{DH}]\label{deftree} A \emph{ternary tree} $\Tc$ (we will simply say a \emph{tree} below) is a rooted tree where each non-leaf (or \emph{branching}) node has exactly three children nodes, which we shall distinguish as the \emph{left}, \emph{mid} and \emph{right} ones. We say $\Tc$ is \emph{trivial} (and write $\Tc=\bullet$) if it consists only of the root, in which case this root is also viewed as a leaf.

We denote generic nodes by $\nf$, generic leaves by $\lf$, the root by $\rf$, the set of leaves by $\Lc$ and the set of branching nodes by $\Nc$. The \emph{scale} of a tree $\Tc$ is defined by $n(\Tc)=|\Nc|$, so if $n(\Tc)=n$ then $|\Lc|=2n+1$ and $|\Tc|=3n+1$.

A tree $\Tc$ may have sign $+$ or $-$. If its sign is fixed then we decide the signs of its nodes as follows: the root $\rf$ has the same sign as $\Tc$, and for any branching node $\nf\in\Nc$, the signs of the three children nodes of $\nf$ from left to right are $(\zeta,-\zeta,\zeta)$ if $\nf$ has sign $\zeta\in\{\pm\}$. Once the sign of $\Tc$ is fixed, we will denote the sign of $\nf\in\Tc$ by $\zeta_\nf$. Define the conjugate $\overline{\Tc}$ of a tree $\Tc$ to be the same tree but with opposite sign.
\end{df}
\begin{df}[Definition 2.2 in \cite{DH}]\label{defcpl}A \emph{couple} $\Qc$ is an unordered pair $(\Tc^+,\Tc^-)$ of two trees $\Tc^\pm$ with signs $+$ and $-$ respectively, together with a partition $\Ps$ of the set $\Lc^+\cup\Lc^-$ into $(n+1)$ pairwise disjoint two-element subsets, where $\Lc^\pm$ is the set of leaves for $\Tc^\pm$, and $n=n^++n^-$ where $n^\pm$ is the scale of $\Tc^\pm$. This $n$ is also called the \emph{scale} of $\Qc$, denoted by $n(\Qc)$. The subsets $\{\lf,\lf'\}\in\Ps$ are referred to as \emph{pairs}, and we require that $\zeta_{\lf'}=-\zeta_\lf$, i.e. the signs of paired leaves must be opposite. If both $\Tc^\pm$ are trivial, we call $\Qc$ the \emph{trivial couple} (and write $\Qc=\times$).

For a couple $\Qc=(\Tc^+,\Tc^-,\Ps)$ we denote the set of branching nodes by $\Nc^*=\Nc^+\cup\Nc^-$, and the set of leave by $\Lc^*=\Lc^+\cup\Lc^-$; for simplicity we will abuse notation and write $\Qc=\Tc^+\cup\Tc^-$. We also define a \emph{paired tree} to be a tree where \emph{some} leaves are paired to each other, according to the same pairing rule for couples. We say a paired tree is \emph{saturated} if there is only one unpaired leaf (called the \emph{lone leaf}). In this case the tree forms a couple with the trivial tree $\bullet$.
\end{df}
\begin{df}[Definition 2.3 in \cite{DH}]\label{defdec} A \emph{decoration} $\Ds$ of a tree $\Tc$ is a set of vectors $(k_\nf)_{\nf\in\Tc}$, such that $k_\nf\in\Zb_L^d$ for each node $\nf$, and that \[k_\nf=k_{\nf_1}-k_{\nf_2}+k_{\nf_3},\quad \mathrm{or\ equivalently}\quad \zeta_\nf k_\nf=\zeta_{\nf_1}k_{\nf_1}+\zeta_{\nf_2}k_{\nf_2}+\zeta_{\nf_3}k_{\nf_3},\] for each branching node $\nf\in\Nc$, where $\zeta_\nf$ is the sign of $\nf$ as in Definition \ref{deftree}, and $\nf_1,\nf_2,\nf_3$ are the three children nodes of $\nf$ from left to right. Clearly a decoration $\Ds$ is uniquely determined by the values of $(k_\lf)_{\lf\in\Lc}$. For $k\in\Zb_L^d$, we say $\Ds$ is a $k$-decoration if $k_\rf=k$ for the root $\rf$.

Given a decoration $\Ds$, we define the coefficient
\begin{equation}\label{defcoef}\epsilon_\Ds:=\prod_{\nf\in\Nc}\epsilon_{k_{\nf_1}k_{\nf_2}k_{\nf_3}}\end{equation} where $\epsilon_{k_1k_2k_3}$ is as in (\ref{defcoef0}). Note that in the support of $\epsilon_\Ds$ we have that $(k_{\nf_1},k_{\nf_2},k_{\nf_3})\in\Sf$ for each $\nf\in\Nc$. We also define the resonance factor $\Omega_\nf$ for each $\nf\in\Nc$ by
\begin{equation}\label{defres}\Omega_\nf=\Omega(k_{\nf_1},k_{\nf_2},k_{\nf_3},k_\nf)=|k_{\nf_1}|_\beta^2-|k_{\nf_2}|_\beta^2+|k_{\nf_3}|_\beta^2-|k_\nf|_\beta^2.\end{equation}

A decoration $\Es$ of a couple $\Qc=(\Tc^+,\Tc^-,\Ps)$, is a set of vectors $(k_\nf)_{\nf\in\Qc}$, such that $\Ds^\pm:=(k_\nf)_{\nf\in\Tc^\pm}$ is a decoration of $\Tc^\pm$, and moreover $k_\lf=k_{\lf'}$ for each pair $\{\lf,\lf'\}\in\Ps$. We define $\epsilon_\Es:=\epsilon_{\Ds^+}\epsilon_{\Ds^-}$, and define the resonance factors $\Omega_\nf$ for $\nf\in\Nc^*$ as in (\ref{defres}). Note that we must have $k_{\rf^+}=k_{\rf^-}$ where $\rf^\pm$ is the root of $\Tc^\pm$; again we say $\Es$ is a $k$-decoration if $k_{\rf^+}=k_{\rf^-}=k$. Finally, we can define decorations $\Ds$ of paired trees, as well as $\epsilon_\Ds$ and $\Omega_\nf$ etc., similar to the above.
\end{df}
\begin{df}[Definition 4.2 in \cite{DH}]\label{defreg} Define a \emph{regular couple} to be a couple formed from the trivial couple $\times$ by repeatedly applying one of the steps $\Ab$ and $\Bb$, where in step $\Ab$ one replaces a pair of leaves with a $(1,1)$-mini couple, and in step $\Bb$ one replaces a node with a mini tree. Here a $(1,1)$-mini couple is a couple formed by two trees each of scale $1$ such that no siblings are paired, and a mini tree is a saturated paired tree of scale $2$ such that no siblings are paired. See Figures \ref{fig:minicpl}--\ref{fig:steps}.
We also define a \emph{regular tree} to be a saturated paired tree $\Tc$, such that $\Tc$ forms a \emph{regular} couple with the trivial tree. This is equivalent to the definition in Remark 4.15 of \cite{DH}, namely that $\Tc$ can be obtained from a regular chain  by replacing each leaf pair with a regular couple. Here a regular chain (see Definition 4.6 of \cite{DH}) is defined to be the result of repeatedly applying step $\Bb$ at a branching node or the lone leaf starting from the trivial tree $\bullet$. Note that the scale of a regular couple or a regular tree is always even.
  \begin{figure}[h!]
  \includegraphics[scale=.45]{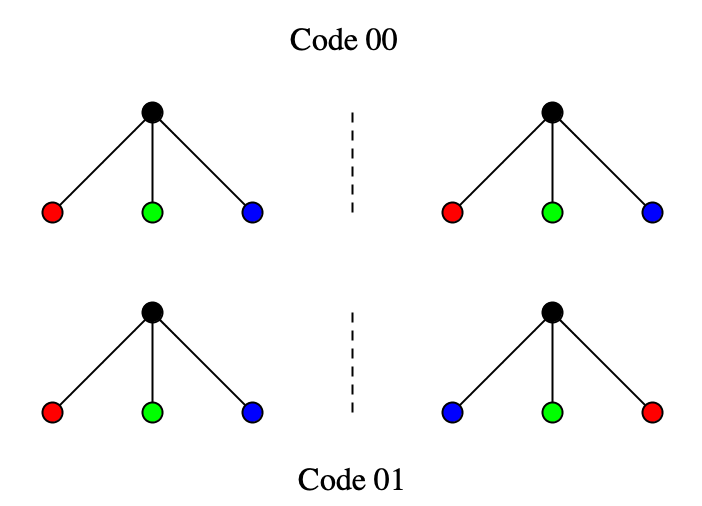}
  \caption{A $(1,1)$-mini couple. Here and below two leaves of same color are paired There are two possibilities indicated by codes $00$ and $01$ as in \cite{DH}.}
  \label{fig:minicpl}
\end{figure} 
  \begin{figure}[h!]
  \includegraphics[scale=.45]{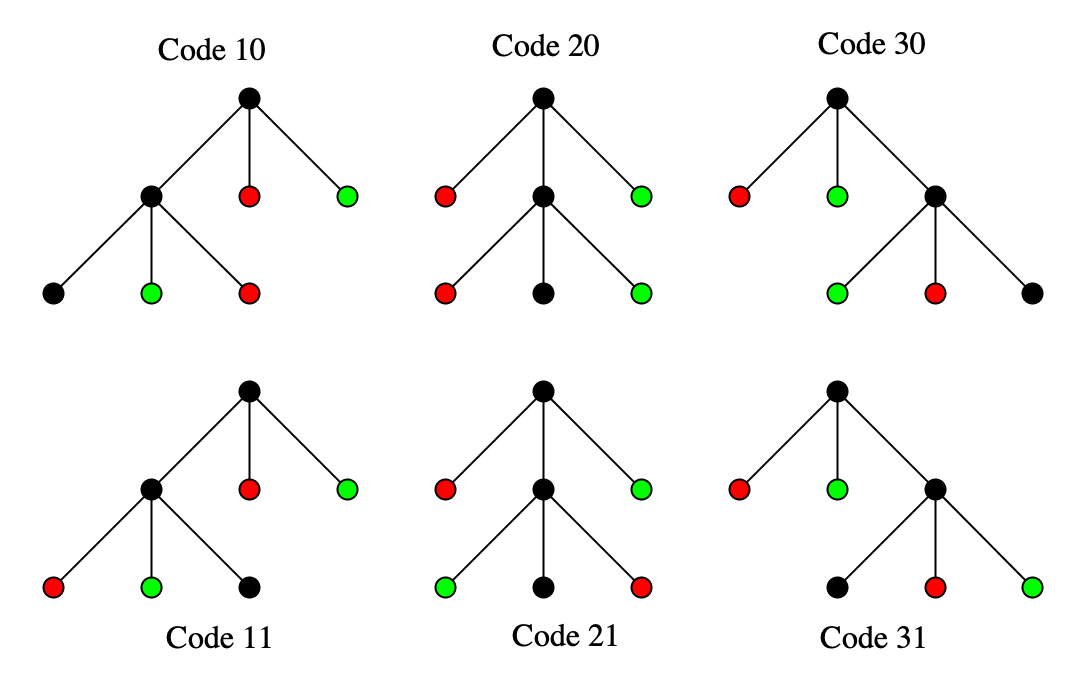}
  \caption{A mini tree. There are six possibilities indicated by codes $10\sim 31$ as in \cite{DH}.}
  \label{fig:minitree}
\end{figure} 
  \begin{figure}[h!]
  \includegraphics[scale=.45]{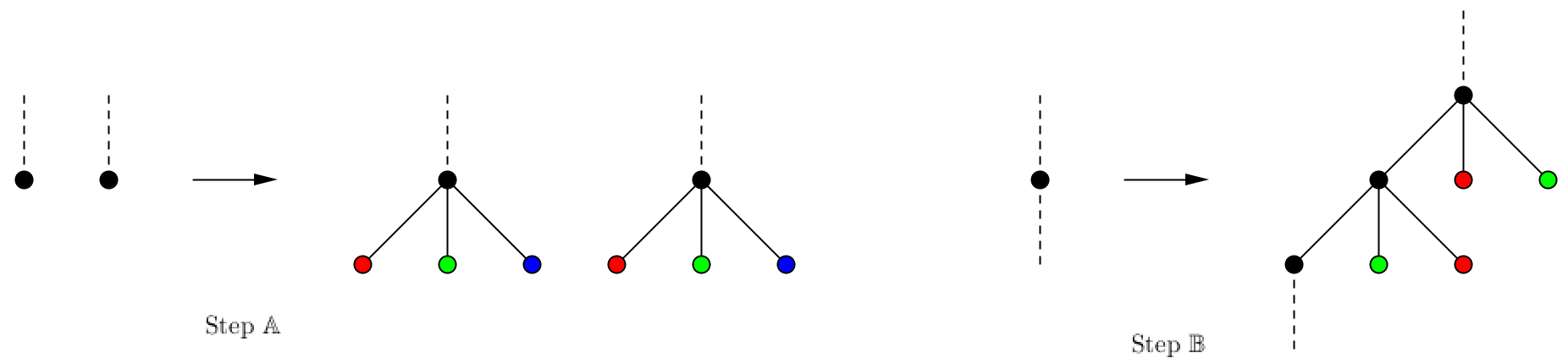}
  \caption{Steps $\Ab$ and $\Bb$ as in Definition \ref{defreg}.}
  \label{fig:steps}
\end{figure} 
\end{df}
\begin{prop}\label{regcount} The number of regular couples and regular trees of scale $n$ is at most $C^n$.
\end{prop}
\begin{proof} See \cite{DH}, Corollary 4.9.
\end{proof}
\begin{lem}\label{auxlem} Let $\Tc$ be a tree of scale $n$. For any node $\nf\in\Tc$ define $\mu_\nf$ to be the number of leaves in the subtree rooted at $\nf$. Then, for any $\nf\in\Nc$, consider the values of $\mu_\mf$ where $\mf$ is a child of $\nf$, and let the \emph{second maximum} of these values be $\mu_\nf^\circ$. Then we have
\begin{equation}\label{auxineq}\prod_{\nf\in\Nc}\mu_\nf^\circ\leq \frac{3^n}{2n+1}.\end{equation}
\end{lem}
\begin{proof}
See \cite{DH}, Lemma 6.6.
\end{proof}
\subsection{Expansion ansatz and regular couples} The following results are taken from \cite{DH}.
\begin{prop}\label{propexp} For any tree $\Tc$, define
\begin{equation}\label{defjt}(\Jc_\Tc)_k(t)=\bigg(\frac{\delta}{2L^{d-1}}\bigg)^n\widetilde{\zeta}(\Tc)\sum_{\Ds}\epsilon_\Ds\int_\Dc\prod_{\nf\in\Nc}e^{\pi i\zeta_\nf\delta L^2\Omega_\nf t_\nf}\,\mathrm{d}t_\nf\prod_{\lf\in\Lc}\sqrt{n_{\mathrm{in}}(k_\lf)}\cdot\eta_{k_\lf}^{\zeta_\lf}(\omega).
\end{equation} Here in (\ref{defjt}), $n$ is the scale of $\Tc$, $\widetilde{\zeta}(\Tc)=\prod_{\nf\in\Nc}(i\zeta_\nf)$, $\Ds$ runs over all $k$-decorations of $\Tc$, and $\Dc$ is the domain 
\begin{equation}\label{defdomain}\Dc=\left\{t[\Nc]:0<t_{\nf'}<t_\nf<t,\mathrm{\ whenever\ \nf'\ is\ a\ child\ of\ \nf}\right\}.
\end{equation}

We may expand $a_k(t)$ as 
\begin{equation}\label{expand}a_k(t)=\sum_{n=0}^N(\Jc_n)_k(t)+b_k(t);\quad(\Jc_n)_k(t)=\sum_{n(\Tc^+)=n}(\Jc_{\Tc^+})_k(t),
\end{equation} where the second sum is taken over all trees $\Tc^+$ of sigh $+$ such that $n(\Tc^+)=n$. 

The remainder $b$ satisfies the equation
\begin{equation}\label{eqnbk}b=\Rc+\Ls b+\Bs(b,b)+\Cs(b,b,b),
\end{equation} where the terms on the right hand side are defined by
\begin{multline}\label{eqnbk1.5}\Rc=\sum_{(0)}\Ic\Cc_+(u,\overline{v},w),\,\,\,\Ls b=\sum_{(1)}\Ic\Cc_+(u,\overline{v},w),\\\Bs(b,b)=\sum_{(2)}\Ic\Cc_+(u,\overline{v},w),\,\,\, \Cs(b,b,b)=\Ic\Cc_+(b,\overline{b},b).\end{multline} In (\ref{eqnbk1.5}) the summations are taken over $(u,v,w)$, each of which being either $b$ or $\Jc_n$ for some $0\leq n\leq N$; moreover in the summation $\sum_{(j)}$ for $0\leq j\leq 2$, exactly $j$ inputs in $(u,v,w)$ equals $b$, and in the summation $\sum_{(0)}$ we require that the sum of the three $n$'s in the $\Jc_n$'s is at least $N$.

Lastly, uniformly in $t\in[0,1]$ and $k\in\Zb_L^d$, we have that 
\begin{equation}\label{twopoint}\bigg|\sum_{0\leq n_1,n_2\leq N}\Eb\big((\Jc_{n_1})_k(t)\overline{(\Jc_{n_2})_k(t)}\big)-n(\delta t,k)\bigg|\lesssim L^{-\nu}.
\end{equation}
\end{prop}
\begin{proof} The expansion (\ref{expand}) is introduced in Sections 2.2.1 and 2.2.2 in \cite{DH}, and (\ref{defjt}) follows by combining the formulas in Section 5.1 of \cite{DH}. The equation (\ref{eqnbk}) for $b$ is deduced in Section 2.2.2 of \cite{DH}. Finally (\ref{twopoint}) is a qualitative version of Theorem \ref{main0}, which is proved in Section 12 of \cite{DH}. Note that here we are choosing $N=\lfloor (\log L)^4\rfloor$ instead of $N=\lfloor \log L\rfloor$, but the proof is not affected as long as (say) $N\ll L^{\delta^2}$.
\end{proof}
\begin{prop}\label{propreg} For any couple $\Qc$, define
\begin{equation}\label{defkq}\Kc_\Qc(t,s,k)=\bigg(\frac{\delta}{2L^{d-1}}\bigg)^n\zeta^*(\Qc)\sum_{\Es}\epsilon_\Es\int_\Ec\prod_{\nf\in\Nc^*}e^{\pi i\zeta_\nf\delta L^2\Omega_\nf t_\nf}\,\mathrm{d}t_\nf\prod_{\lf\in\Lc^*}^{(+)}n_{\mathrm{in}}(k_\lf).
\end{equation} Here in (\ref{defkq}), $n$ is the scale of $\Qc$, $\zeta^*(\Qc)=\prod_{\nf\in\Nc^*}(i\zeta_\nf)$, $\Es$ runs over all $k$-decorations of $\Qc$, the last product is taken over all $\lf\in\Lc^*$ with sign $+$, and $\Ec$ is the domain 
\begin{equation}\label{defdomain2}\Ec=\left\{t[\Nc^*]:0<t_{\nf'}<t_\nf,\mathrm{\ when\ \nf'\ is\ a\ child\ of\ \nf};\ t_\nf<t\mathrm{\ for\ \nf\in\Nc^+},t_\nf<s\mathrm{\ for\ \nf\in\Nc^-}\right\}.
\end{equation} Now suppose $\Qc$ is a \emph{regular} couple with scale $2n$ where $n\leq (\log L)^{50}$, then there exist a function $(\Kc_\Qc)_{\mathrm{app}}(t,s,k)$, which is the sum of at most $2^n$ terms, such that each term has the form $\delta^n\cdot \Jc\Ac(t,s)\cdot\Mc(k)$ (with possibly different $\Jc\Ac$ and $\Mc$ for different terms), and that
\begin{equation}\label{regbound1}\|\Jc\Ac\|_{X_{\mathrm{loc}}}\leq (C^+)^n,\quad \sup_{|\rho|\leq 40d}|\partial^\rho\Mc(k)|\leq (C^+)^n\langle k\rangle^{-40d}\quad \mathrm{for\ each\ term};
\end{equation}
\begin{equation}\label{regbound2}\|\Kc_\Qc(t,s,k)-(\Kc_\Qc)_{\mathrm{app}}(t,s,k)\|_{X_{\mathrm{loc}}^{40d}}\leq (C^+\delta)^nL^{-2\nu}.
\end{equation}

Similarly, for any regular tree $\Tc$ with lone leaf $\lf_*$, define \begin{equation}\label{defkqnew}\Kc_\Tc^*(t,s,k)=\bigg(\frac{\delta}{2L^{d-1}}\bigg)^n\widetilde{\zeta}(\Tc)\sum_{\Ds}\int_\Dc\prod_{\nf\in\Nc}e^{\pi i\zeta_\nf\delta L^2\Omega_\nf t_\nf}\,\mathrm{d}t_\nf\prod_{\lf\in\Lc\backslash\{\lf_*\}}^{(+)}n_{\mathrm{in}}(k_\lf).
\end{equation} Here in (\ref{defkqnew}), $n$ is the scale of $\Tc$, $\widetilde{\zeta}(\Tc)=\prod_{\nf\in\Nc}(i\zeta_\nf)$, $\Ds$ runs over all $k$-decorations of $\Qc$, the last product is taken over all $\lf\in\Lc\backslash\{\lf_*\}$ with sign $+$, and $\Dc$ is the domain 
\begin{equation}\label{defdomainnew}\Dc=\left\{t[\Nc]:0<t_{\nf'}<t_\nf<t,\mathrm{\ when\ \nf'\ is\ a\ child\ of\ \nf};\ t_{(\lf_*)^p}>s\right\}
\end{equation} where $(\lf_*)^p$ is the parent of $\lf_*$. Suppose $\Tc$ has scale $2n$ where $n\leq (\log L)^{20}$, then there exist a function $(\Kc_\Tc^*)_{\mathrm{app}}(t,s,k)$, which is the sum of at most $2^n$ terms, such that each term has the form $\delta^n\cdot \Jc\Ac^*(t,s)\cdot\Mc^*(k)$ (with possibly different $\Jc\Ac^*$ and $\Mc^*$ for different terms), and that
\begin{equation}\label{regbound3}\|\Jc\Ac^*\|_{X_{\mathrm{loc}}}\leq (C^+)^n,\quad \sup_{|\rho|\leq 40d}|\partial^\rho\Mc^*(k)|\leq (C^+)^n\quad\mathrm{for\ each\ term};
\end{equation}
\begin{equation}\label{regbound4}\|\Kc_\Tc^*(t,s,k)-(\Kc_\Tc^*)_{\mathrm{app}}(t,s,k)\|_{X_{\mathrm{loc}}^{0}}\leq (C^+\delta)^nL^{-2\nu}.
\end{equation}
\end{prop}
\begin{proof} This follows from Propositions 6.7 and 6.10 of \cite{DH}. Whether the upper bound for $n$ is $(\log L)^6$ or $(\log L)^{50}$ does not affect the proof (again, as long as $n\ll L^{\delta^2}$).
\end{proof}
\section{Gardens} \label{sec:gardens}
\subsection{Structure of gardens} The key concept of this paper is a generalization of couples, which we call \emph{gardens}.
\begin{df}\label{defgarden} Given a sequence $(\zeta_1,\cdots,\zeta_{2R})$, where $\zeta_j\in\{\pm\}$ and exactly $R$ of them are $+$, we define a \emph{garden} $\Gc$ of signature $(\zeta_1,\cdots,\zeta_{2R})$, to be an ordered collection of trees $(\Tc_1,\cdots,\Tc_{2R})$, such that $\Tc_j$ has sign $\zeta_j$ for $1\leq j\leq 2R$, together with a partition $\Ps$ of the set of leaves in all $\Tc_j$ into two-element subsets (again called pairings) such that the two paired leaves have opposite signs, see Figure \ref{fig:garden}. The \emph{width} of the garden is defined to be $2R$, which is always an even number. The \emph{scale} $n(\Gc)$ of a garden $\Gc$ is the sum of scales of all $\Tc_j\,(1\leq j\leq 2R)$. We denote $\Lc^*=\Lc_1\cup\cdots\cup\Lc_{2R}$ to be the set of leaves and $\Nc^*=\Nc_1\cup\cdots\cup\Nc_{2R}$ to be the set of branching nodes, where $\Lc_j$ and $\Nc_j$ are the sets of leaves and branching nodes of $\Tc_j$.

Note that a garden of width $2$ is just a couple. If the set $\{1,\cdots,2R\}$ can be partitioned into two-element subsets such that for each such subset $\{j,j'\}$, the leaves in $\Tc_{j}$ and $\Tc_{j'}$ are all paired with each other (in particular $\zeta_{j'}=-\zeta_{j}$), then we say this garden is a \emph{multi-couple}. In this case, this garden is formed by $R$ couples $(\Tc_{j},\Tc_{j'})$. If each of them is a regular couple then we say the multi-couple is \emph{regular}. A trivial garden is a garden when all $\Tc_j$ are trivial trees; note that it is always a regular multi-couple (formed by $R$ trivial couples). If in a garden $\Gc$, no two trees $\Tc_j$ and $\Tc_{j'}$ have all their leaves paired with each other, then we say the garden is \emph{mixed}.
\end{df}
\begin{df}\label{defdecgarden} Given a garden $\Gc$, a \emph{decoration} of $\Gc$, denoted by $\Is$, is a set of vectors $(k_\nf)_{\nf\in\Gc}$ where $\nf$ runs over all nodes of $\Gc$, such that $(k_\nf)_{\nf\in\Tc_j}$ is a decoration of $\Tc_j$ for each $1\leq j\leq 2R$, and $k_{\nf'}=k_{\nf}$ for each pair of leaves $\{\nf,\nf'\}$. Given vectors $(k_1,\cdots,k_{2R})$, we say a decoration is a $(k_1,\cdots,k_{2R})$-decoration, if $k_{\rf_j}=k_j$ for each $1\leq j\leq 2R$, where $\rf_j$ is the root of $\Tc_j$. For any branching node $\nf\in \Nc^*$, define $\Omega_\nf$ as in (\ref{defres}), see Figure \ref{fig:garden}. We also define $\epsilon_\Is=\prod_{j=1}^{2R}\epsilon_{\Ds_j}$, where $\epsilon_{\Ds_j}$ is defined as in (\ref{defcoef}), with $\Ds_j$ being the restriction of $\Is$ to $\Tc_j$.
\end{df}
  \begin{figure}[h!]
  \includegraphics[scale=.45]{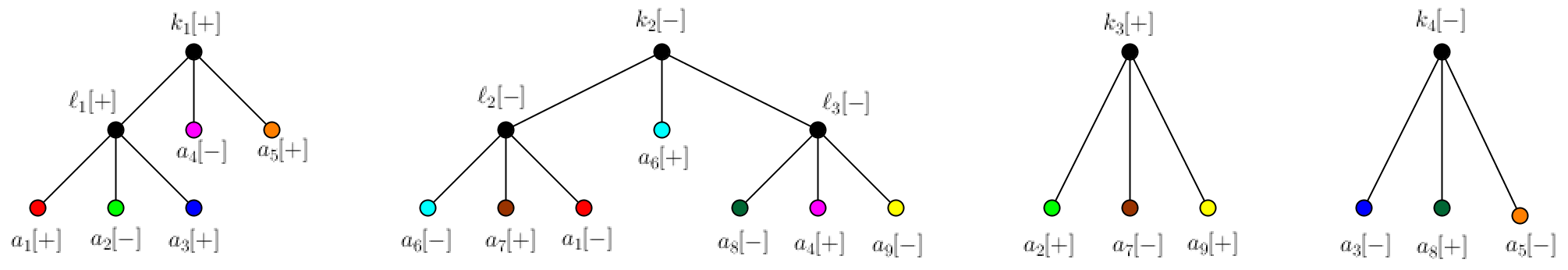}
  \caption{An example of a garden (Definition \ref{defgarden}) of width $4$ and scale $7$, together with a decoration (Definition \ref{defdecgarden}). The signs of nodes are also indicated; the signature is $(+,-,+,-)$.}
  \label{fig:garden}
\end{figure} 
\begin{df} Define the steps $\Ab$ and $\Bb$ for gardens in the same way as for couples in Definition \ref{defreg}, see Figure \ref{fig:steps}. Define a garden $\Gc$ to be \emph{prime} if it is not obtained from any other garden by performing steps $\Ab$ or $\Bb$.
\end{df}
\begin{prop}\label{defskeleton} For any garden $\Gc$ there exists a unique prime garden $\Gc_{sk}$ such that $\Gc$ is obtained from $\Gc_{sk}$ by applying steps $\Ab$ and $\Bb$. This $\Gc_{sk}$ is called the \emph{skeleton} of $\Gc$. Finally, $\Gc_{sk}$ is a trivial garden, if and only if $\Gc$ is a regular multi-couple.
\end{prop}
\begin{proof} The proof is the same as Proposition 4.13 of \cite{DH}. For the convenience of the reader we present the proof here. Denote the inverse operations of $\Ab$ and $\Bb$ by $\overline{\Ab}$ and $\overline{\Bb}$, where one collapses a $(1,1)$-mini couple or a mini tree to a leaf pair or a single node. To prove existence of $\Gc_{sk}$, by definition, one just needs to repeatedly apply $\overline{\Ab}$ and $\overline{\Bb}$ until no such operation is possible.

To prove uniqueness of $\Gc_{sk}$, we just make one key observation: if $\Gc$ contains two basic objects (i.e. $(1,1)$-mini couples or mini trees), and let $\overline{\Db}_1$ and $\overline{\Db}_2$ be the inverse operations associated with them, then $\overline{\Db}_1\overline{\Db}_2=\overline{\Db}_2\overline{\Db}_1$. In fact, this just shows that collapsing one of the basic objects does not affect the other, which can be directly verified by definition.

Now we can prove the uniqueness of $\Gc_{sk}$ by induction. The base case is easy, suppose uniqueness is true for $\Gc$ of smaller scale, then for any $\Gc$ we shall look for $(1,1)$-couples and mini trees (Definition \ref{defreg}). If there is none then $\Gc$ is already prime; if there is only one, then we apply $\overline{\Ab}$ or $\overline{\Bb}$ to collapse it and apply induction hypothesis for the resulting garden. Suppose there are more than one, then we apply $\overline{\Ab}$ or $\overline{\Bb}$ to collapse any one of them and apply induction hypothesis for the resulting garden. The final result does not depend on the first $\overline{\Ab}$ or $\overline{\Bb}$ we choose, because any two such steps, which can be performed for the original $\Gc$, must commute as proved above. Therefore $\Gc_{sk}$ is unique.
\end{proof}
\begin{prop}\label{structuregarden}
Suppose $\Gc$ is a garden with skeleton $\Gc_{sk}$. Then $\Gc$ is formed from $\Gc_{sk}$ by replacing each leaf pair with a regular couple and each branching node with a regular tree, see Figure \ref{fig:gardensk}. This representation is unique.
  \begin{figure}[h!]
  \includegraphics[scale=.45]{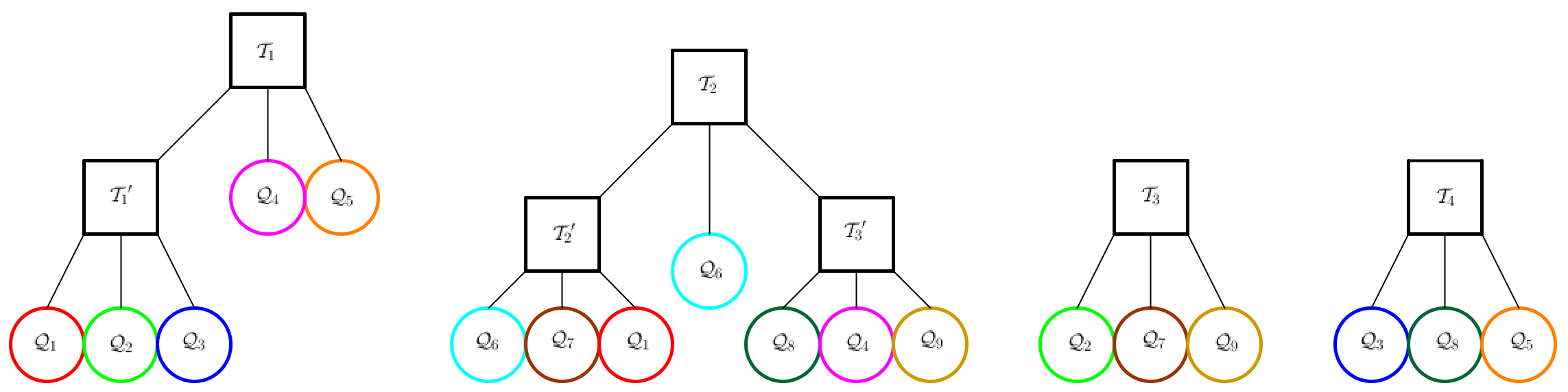}
  \caption{A garden whose skeleton is the garden in Figure \ref{fig:garden}, see Proposition \ref{structuregarden}. Here each $\Tc_j$ and $\Tc_j'$ represents a regular tree, and each $\Qc_j$ represents a regular couple.}
  \label{fig:gardensk}
\end{figure} 
\end{prop}
\begin{proof} The proof is basically the same as Proposition 4.14 of \cite{DH}. To prove existence, we can induct on the scale of $\Gc$. The base case $\Gc=\Gc_{sk}$ is obvious. Suppose the result is true for $\Gc$, and let $\Gc_+$ be obtained from $\Gc$ by applying $\Ab$ or $\Bb$. We know that $\Gc$ is obtained from $\Gc_{sk}$ by replacing each branching node with a regular tree $\Tc_j\,(1\leq j\leq n)$, and replacing each leaf pair by a regular couple $\Qc_j\,(1\leq j\leq m)$. Then:

(1) If one applies step $\Ab$, then this step $\Ab$ must be applied, either at a leaf pair belonging to some regular couple $\Qc_i\,(1\leq i\leq m)$, or at a leaf pair belonging to some regular tree $\Tc_i\,(1\leq i\leq n)$. In this case the other regular trees and regular couples remain the same, and the regular tree $\Tc_i$ or regular couples $\Qc_i$ is replaced by $\Ab\Tc_i$ or $\Ab\Qc_i$.

(2) If one applies step $\Bb$, then this step $\Bb$ must be applied, either at node belonging to some regular couple $\Qc_i\,(1\leq i\leq m)$, or at a node belonging to some regular tree $\Tc_i\,(1\leq i\leq n)$. In this case the other regular trees and regular couples remain the same, and the regular tree $\Tc_i$ or regular couples $\Qc_i$ is replaced by $\Bb\Tc_i$ or $\Bb\Qc_i$.

In either case, notice that a regular tree or a regular couple still remains a regular tree or a regular couple after applying step $\Ab$ or $\Bb$. This proves existence.

Now to prove uniqueness of the representation, note that by Definition \ref{defreg}, the process of forming $\Gc$ from $\Gc_{sk}$ can also be described as follows: (i) first replace each branching node of $\Gc_{sk}$ by a \emph{regular chain}, forming a garden $\Gc_{int}$; (ii) replacing each leaf pair in $\Gc_{int}$ by a regular couple to form $\Gc$. Given $\Gc_{sk}$, clearly $\Gc_{int}$ uniquely determines the regular chains in step (i), and also uniquely determines the regular couples in step (ii) replacing the leaf pairs in $\Gc_{int}$, so it suffices to show that $\Gc$ uniquely determines $\Gc_{int}$. Now we can show, via a case-by-case argument, that $\Gc_{int}$ contains no nontrivial regular sub-couple (i.e. no two subtrees rooted at two nodes in $\Gc_{int}$ form a nontrivial regular couple). Since $\Gc$ is formed from $\Gc_{int}$ by replacing each leaf pair with a regular couple, we see that $\Gc_{int}$ can be reconstructed by collapsing each \emph{maximal regular sub-couple} (under inclusion) in $\Gc$ to a leaf pair (because any regular sub-couple of $\Gc$ must be a sub-couple of one of the regular couples in $\Gc$ replacing a leaf pair in $\Gc_{int}$). Clearly this collapsing process is commutative as explained in the proof of Proposition \ref{defskeleton}, hence the resulting couple $\Gc_{int}$ is unique. This completes the proof.
\end{proof}
\begin{prop}\label{skeletonrecover} Given any $\Gc_{sk}$, the number of gardens $\Gc$ that has scale $m$, width $2R$ and skeleton $\Gc_{sk}$ is at most $C^{m+R}$.
\end{prop}
\begin{proof} This is basically the same as Corollary 4.16 in \cite{DH}. If $\Gc$ has scale $m$ and width $2R$, then $\Gc_{sk}$ has scale at most $m$ and width at most $2R$. Given $\Gc_{sk}$, to construct $\Gc$, using Proposition \ref{structuregarden}, we just need to choose a regular tree at each branching node of $\Gc_{sk}$, and a regular couple at each leaf pair of $\Gc_{sk}$. Note that the number of branching nodes in $\Gc_{sk}$ is at most $m$, and the number of leaf pairs in $\Gc_{sk}$ is at most $m+R$. Thus the number of choices for $\Gc$ is at most
\[\sum_{n_1+\cdots+n_{m'}\leq m}C_0^{n_1}\cdots C_0^{n_{m'}}\leq C^{m+R},\] where $m'=2m+R$, and $C_0$ is an absolute constant as in Proposition \ref{regcount}.
\end{proof}
\subsection{Expressions $\Mc_\Gc$ for gardens $\Gc$}
Given a garden $\Gc=(\Tc_1,\cdots,\Tc_{2R})$ with width $2R$, signature $(\zeta_1,\cdots,\zeta_{2R})$ and scale $m$, and $k_j\in\Zb_L^d$ for each $1\leq j\leq 2R$, and time $t\in[0,1]$, define
\begin{equation}\label{defkg}\Mc_\Gc(t,k_1,\cdots,k_{2R})=\bigg(\frac{\delta}{2L^{d-1}}\bigg)^m\zeta^*(\Gc)\sum_\Is\epsilon_\Is\int_{\Ic}\prod_{\nf\in\Nc^*}e^{\pi i\zeta_\nf\cdot\delta L^2\Omega_\nf t_\nf}\,\mathrm{d}t_\nf\cdot\prod_{\lf\in\Lc^*}^{(+)}n_{\mathrm{in}}(k_\lf).
\end{equation} Here in (\ref{defkg}), $\zeta^*(\Gc)=\prod_{\nf\in\Nc^*}(i\zeta_\nf)$ and $\epsilon_\Is=\prod_{j=1}^{2R}\epsilon_{\Ds_j}$ where $\Ds_j$ is the restriction of $\Is$ to $\Tc_j$ (which is a $k_j$-decoration of $\Tc_j$), the sum is taken over all $(k_1,\cdots,k_{2R})$-decorations $\Is$, the last product is taken over all $\lf\in\Lc^*$ with sign $+$, and $\Ic$ is the domain
\begin{equation}\label{timegarden}\Ic=\left\{t[\Nc^*]:0<t_{\nf'}<t_\nf<t,\mathrm{\ whenever\ \nf'\ is\ a\ child\ of\ \nf}\right\}.
\end{equation}

By using Isserlis' theorem (Lemma A.2 in \cite{DH}) and repeating the arguments in Section 2.2.3 of \cite{DH}, we can obtain, for any tree $\Tc_j\,(1\leq j\leq 2R)$ with sign $\zeta_j$, that
\begin{equation}\label{corrgarden}\Eb\bigg(\prod_{j=1}^{2R}(\Jc_{\Tc_j})_{k_j}^{\zeta_j}(t)\bigg)=\sum_{\Ps}\Mc_\Gc(t,k_1,\cdots,k_{2R}).
\end{equation} Here the sum is taken over all possible pairings $\Ps$ that make $(\Tc_1,\cdots,\Tc_{2R})$ a garden, and $\Gc$ is the resulting garden.

We can reduce (\ref{unibound}) to the following two propositions. Here Proposition \ref{mainprop} is the key component, and Proposition \ref{mainprop2} follows from similar arguments. Note also that Proposition \ref{mainprop2} is actually an improvement of Propositions 12.1--12.2 of \cite{DH}, where the decay of exceptional probability is improved from $L^{-A}$ to $e^{-(\log L)^3}$.
\begin{prop}\label{mainprop} Fix $R$ and $(\zeta_1,\cdots,\zeta_{2R})$ and $(k_1,\cdots,k_{2R})$, and $(m_1,\cdots,m_{2R})$. Assume $R\leq 2\log L$, and $m_j\leq N\,(1\leq j\leq 2R)$, and set $m:=m_1+\cdots+m_{2R}$. Consider the sum
\begin{equation}\label{mixedsum}\Ss=\sum_\Gc\Mc_\Gc(t,k_1,\cdots,k_{2R}),
\end{equation} where the sum is taken over all \emph{mixed} gardens $\Gc=(\Tc_1,\cdots,\Tc_{2R})$ of width $2R$ and signature $(\zeta_1,\cdots,\zeta_{2R})$ such that the scale of $\Tc_j$ is $m_j$ for $1\leq j\leq 2R$, then we have
\begin{equation}\label{mixedest}|\Ss|\lesssim (C^+\delta^{1/4})^m L^{-\nu R}
\end{equation} uniformly in $t$ and in $(k_1,\cdots,k_{2R})$.
\end{prop}
\begin{prop}\label{mainprop2} With probability $\geq 1-e^{-(\log L)^3}$, we have
\begin{equation}\label{largedevest}|(\Jc_n)_k(t)|\lesssim\langle k\rangle^{-9d}(C^+\sqrt{\delta})^{n/2}e^{(\log L)^3},\quad |\Rc_k(t)|\lesssim \langle k\rangle^{-9d}(C^+\sqrt{\delta})^{N/2}e^{(\log L)^3}
\end{equation} for all $0\leq n\leq N^3$, as well as
\begin{equation}\label{operatordev}\|\Ls^n\|_{Z\to Z}\lesssim (C^+\sqrt{\delta})^{n/2}e^{(\log L)^3},
\end{equation} for all $0\leq n\leq N$, uniformly in any $k\in\Zb_L^d$ and $t\in[0,1]$. Here $\Rc$ and $\Ls$ are defined in (\ref{eqnbk1.5}), and the $Z$ norm is defined in (\ref{defznorm}).
\end{prop}
Before proceeding, we first illustrate how Propositions \ref{mainprop}--\ref{mainprop2} imply Theorem \ref{main}.
\begin{proof}[Proof of Theorem \ref{main}] We only need to prove (\ref{unibound}). Let $E_1$ be the event that (\ref{nls}) has a smooth solution on $[0,T]$, and $E\subset E_1$ be the event that Proposition \ref{mainprop2} holds, then $\Pb(E_1)\geq \Pb(E)\geq 1-e^{-(\log L)^3}$. 

Note that, under the assumption $E$, we can bound the remainder $b$ defined in (\ref{expand}) by $\|b\|_Z\leq e^{-(\log L)^4}$. This can be proved similarly as in Proposition 12.3 of \cite{DH}. In fact, the equation (\ref{eqnbk}) satisfied by $b$ can be written as
\begin{equation}\label{eqnbk2} b=(1-\Ls^N)^{-1}(1+\Ls+\cdots+\Ls^{N-1})(\Rc+\Bs(b,b)+\Cs(b,b,b)).
\end{equation} We view this as the fixed point equation for a contraction mapping from the set $\{b\in Z:\|b\|_Z\leq e^{-(\log L)^4}\}$ to itself, hence the solution $b$ is unique and satisfies the desired bound. The contraction mapping property follows from the estimates (using also the definition of $\Bs$ and $\Cs$, see (\ref{eqnbk1.5}))
\begin{align}
\label{estimate1}&\|\Rc\|_Z\leq e^{-2(\log L)^4},\\
\label{estimate2}&\|\Jc_n\|_Z\leq e^{2(\log L)^3}\,\,(\forall\, 0\leq n\leq N^3),\\
\label{estimate3}&\|\Ic\Cc_+(f,g,h)\|_Z\leq L^{10d}\|f\|_Z\|g\|_Z\|h\|_Z,\\
\label{estimate4}&\|\Ls^n\|_{Z\to Z}\leq e^{2(\log L)^3}\,\,(\forall\,0\leq n\leq N),\\
\label{estimate5}&\|(1-\Ls^N)^{-1}\|_{Z\to Z}\leq 2.
\end{align} Here (\ref{estimate1})--(\ref{estimate2}) follow from (\ref{largedevest}) and our choice $N=\lfloor (\log L)^4\rfloor$, (\ref{estimate3}) is elementary, and (\ref{estimate4})--(\ref{estimate5}) follow from (\ref{operatordev}), our choice $N=\lfloor (\log L)^4\rfloor$ and Neumann series expansions.

Now, to prove (\ref{unibound}) we need to calculate
\begin{equation}\label{expect0}\Eb\bigg(\mathbf{1}_{E_1}\cdot\prod_{j=1}^{2R}a_{k_j}^{\zeta_j}(t)\bigg).\end{equation} As assumed we have $R\leq \log L$. Using mass conservation we can bound $|a_{k_j}(t)|\lesssim L^d$ for each $k_j$, so if $E_1$ is replaced by $E_1\backslash E$ in (\ref{expect0}), the corresponding contribution is bounded by
\[\bigg|\Eb\bigg(\mathbf{1}_{E_1\backslash E}\cdot\prod_{j=1}^{2R}a_{k_j}^{\zeta_j}(t)\bigg)\bigg|\leq \Pb(E^c)L^{2dR}\leq e^{-(\log L)^3}L^{2d\log L}\leq L^{-10},\] so we may replace $E_1$ by $E$ in (\ref{expect0}). Under the assumption $E$, we may expand $a_{k_j}(t)$ using (\ref{expand}), which leads to different combinations of terms.

Consider the terms where all factors are of form $\Jc_n$. For such factors we will also replace $\mathbf{1}_E$ by $1$, and deal with the resulting error term later. As such, we get a contribution
\begin{equation}\label{contribution1}\sum_{1\leq m_1,\cdots,m_{2R}\leq N}\Eb\bigg(\prod_{j=1}^{2R}(\Jc_{m_j})_{k_j}^{\zeta_j}(t)\bigg).
\end{equation} For fixed $(m_1,\cdots,m_{2R})$, using the second expansion in (\ref{expand}) and (\ref{corrgarden}), we can write
\begin{equation}\label{contribution2}\Eb\bigg(\prod_{j=1}^{2R}((\Jc_{m_j})_{k_j}(t))^{\zeta_j}\bigg)=\sum_\Gc\Mc_\Gc(t,k_1,\cdots,k_{2R}),
\end{equation} where the sum is taken over \emph{all gardens} $\Gc=(\Tc_1,\cdots,\Tc_{2R})$ of width $2R$ and signature $(\zeta_1,\cdots,\zeta_{2R})$ such that the scale of $\Tc_j$ is $m_j$ for $1\leq j\leq 2R$. Note that by definition, each $\Gc$ is uniquely expressed as the union of some couples and a mixed garden; suppose the number of couples is $R_1\leq R$, and $R_2:=R-R_1$. If $R_2=0$, then there is a unique partition $\Pc$ of $\{1,\cdots,2R\}$ into two-element subsets $\{j,j'\}$ such that $\zeta_{j'}=-\zeta_j$ and $\{\Tc_j,\Tc_{j'}\}$ is a couple for each pair $\{j,j'\}$, moreover for $\Mc_\Gc(t,k_1,\cdots,k_{2R})$ to be nonzero one must have $k_j=k_{j'}$. For $\Pc$ fixed, the contribution of this part of sum equals
\begin{equation}\label{contribution3}\prod_{\{j,j'\}\in\Pc}\mathbf{1}_{k_j=k_{j'}}\sum_{\Qc}\Mc_\Qc(t,k_j,k_j)=\prod_{\{j,j'\}\in\Pc}\mathbf{1}_{k_j=k_{j'}}\cdot\Eb\big((\Jc_{m_j})_{k_j}^{\zeta_j}(t)(\Jc_{m_{j'}})_{k_j}^{\zeta_{j'}}(t)\big),\end{equation}where for fixed $\{j,j'\}\in\Pc$, the sum is taken over all couples $\Qc=\{\Tc_j,\Tc_{j'}\}$ such that the two trees have signs $\zeta_j$ and $\zeta_{j'}$ and scales $m_j$ and $m_{j'}$ respectively, and the equality in (\ref{contribution3}) follows from (\ref{corrgarden}). Now, upon summing over all choices for $(m_1,\cdots,m_{2R})$ and using (\ref{twopoint}), we obtain that this contribution equals
\[\sum_\Pc\prod_{\{j,j'\}\in\Pc}\mathbf{1}_{k_j=k_{j'}}\prod_j^{(+)}\big(n(\delta t,k_j)+O(L^{-\nu})\big)=\sum_\Pc\prod_{\{j,j'\}\in\Pc}\mathbf{1}_{k_j=k_{j'}}\prod_j^{(+)}n(\delta t,k_j)+O(R!M_{\mathrm{kin}}^RL^{-\nu}),\] where in the last inequality we have used that $1+|n(\delta t,k_j)|\leq M_{\mathrm{kin}}$ for each $j$.

Next, consider the contribution where $R_2>0$. Up to a factor $\binom{2R}{2R_2}R_1!\leq (2R)^{2R_2}R!$ and a permutation, we may assume $\{\Tc_{2j-1},\Tc_{2j}\}$ is a couple for $R_2+1\leq j\leq R$ and $\Gc_2:=(\Tc_1,\cdots,\Tc_{2R_2})$ is a mixed garden. Again we must have $k_{2j-1}=k_{2j}$ for $R_2+1\leq j\leq R$; if we fix $(m_1,\cdots,m_{2R_2})$ and sum over the other $m_j$, then in the same way as above, we can bound the corresponding contribution by
\begin{equation}\label{contribution4}(2R)^{2R_2}R!\prod_{j=R_2+1}^R\big(n(\delta t,k_{2j-1})+O(L^{-\nu})\big)
\cdot\sum_{\Gc_2}\Kc_{\Gc_2}(t,k_1,\cdots,k_{2R_2}),\end{equation} where the sum is taken over all mixed gardens $\Gc_2=(\Tc_1,\cdots,\Tc_{2R_2})$ of width $2R_2$ and signature $(\zeta_1,\cdots,\zeta_{2R_2})$ such that the scale of $\Tc_j$ is $m_j$. By Proposition \ref{mainprop} we have that
\[(\ref{contribution4})\leq (2R)^{2R_2}R!\cdot M_{\mathrm{kin}}^R\cdot \delta^{(m_1+\cdots +m_{2R_2})/8}L^{-\nu R_2}.\] Upon summing over $(m_1,\cdots,m_{2R_2})$ and using that $R\leq \log L$, we can bound this contribution by the right hand side of (\ref{unibound}).

Finally, we show that all the remainder terms are bounded by the right hand side of (\ref{unibound}). In fact, the above arguments imply that
\[\bigg|\Eb\bigg(\prod_{j=1}^{2R}(\widetilde{\Jc})_{k_j}^{\zeta_j}(t)\bigg)\bigg|\lesssim R! \cdot M_{\mathrm{kin}}^R,\quad\mathrm{where\ }\widetilde{\Jc}=\sum_{n=0}^N\Jc_n;\] in particular we have
\begin{equation}\label{contributionfin}\Eb\bigg(\prod_{j=1}^{2R}|(\widetilde{\Jc})_{k_j}(t)|^2\bigg)\lesssim (2R)! \cdot M_{\mathrm{kin}}^{2R}.\end{equation} Since $R\leq \log L$, this allows to control the terms where all factors are of form $\Jc_n$, but with $\mathbf{1}_E$ replaced by $1-\mathbf{1}_E$ (where we simply apply Cauchy-Schwartz and use the fact $\Pb(E^c)\leq e^{-(\log L)^3}$); similarly, if at least one factor in the expansion is the remainder $b$, then we can also apply Cauchy-Schwartz and use the bound $|b_k(t)|\leq e^{-(\log L)^4}$ together with (\ref{contributionfin}) to control this term. This completes the proof.
\end{proof}
From now on we will focus on the proof  of Propositions \ref{mainprop}--\ref{mainprop2}.
\section{Irregular chains}\label{irre}
\subsection{Reduction to prime gardens} Let $\Gc_{sk}$ be the skeleton of a garden $\Gc$, which is then a prime garden. By Proposition \ref{structuregarden}, $\Gc$ can be obtained from $\Gc_{sk}$ by replacing each branching node $\mf$ with a regular tree $\Tc^{(\mf)}$, and replacing each leaf pair $\{\mf,\mf'\}$ in $\Gc_{sk}$ with a regular couple $\Qc^{(\mf,\mf')}$. Similar to Section 8.1 of \cite{DH}, using Proposition \ref{propreg}, we can reduce $\Mc_\Gc(t,k_1,\cdots,k_{2R})$ to an expression that has similar form with $\Mc_{\Gc_{sk}}(t,k_1,\cdots,k_{2R})$. For the sake of completeness we briefly recall the reduction process below.

Recall that
\begin{equation}\label{defkg2}\Mc_\Gc(t,k_1,\cdots,k_{2R})=\bigg(\frac{\delta}{2L^{d-1}}\bigg)^m\zeta^*(\Gc)\sum_\Is\epsilon_\Is\int_{\Ic}\prod_{\nf\in\Nc^*}e^{\pi i\zeta_\nf\cdot\delta L^2\Omega_\nf t_\nf}\,\mathrm{d}t_\nf\cdot\prod_{\lf\in\Lc^*}^{(+)}n_{\mathrm{in}}(k_\lf),
\end{equation} where $m$ is the scale of $\Gc$, $\Ic$ is the domain defined in (\ref{timegarden}), $\Is$ is a $(k_1,\cdots,k_{2R})$-decoration and other objects are defined as before, all associated to the garden $\Gc$. By definition, the restriction of $\Is$ to nodes in $\Gc_{sk}$ forms a $(k_1,\cdots,k_{2R})$-decoration of $\Gc_{sk}$, and the relevant quantities such as $\Omega_\nf$ are the same for both decorations (i.e. each $\Omega_\nf$ in the decoration of $\Gc_{sk}$ uniquely corresponds to some $\Omega_\nf$ in the corresponding decoration of $\Gc$).

Now, let $\{\mf,\mf'\}$ be a leaf pair in $\Gc_{sk}$, which becomes the roots of the regular sub-couple $\Qc^{(\mf,\mf')}$ in $\Gc$. We must have $k_\mf=k_{\mf'}$. In (\ref{defkg2}), consider the summation in the variables $k_\nf$, where $\nf$ runs over all nodes in $\Qc^{(\mf,\mf')}$ other than $\mf$ and $\mf'$ (these variables, together with $k_\mf$ and $k_{\mf'}$, form a $k_\mf$-decoration of $\Qc^{(\mf,\mf')}$), and the integration in the variables $t_\nf$, where $\nf$ runs over all branching nodes in $\Qc^{(\mf,\mf')}$, with all the other variables fixed. By definition, this summation and integration equals, up to some sign $\zeta^*(\Qc^{(\mf,\mf')})$ and some power of $\delta(2L^{d-1})^{-1}$, the \emph{exact} expression $\Kc_{\Qc^{(\mf,\mf')}}(t_{\mf^p},t_{(\mf')^p},k_\mf)$. Here we assume $\zeta_\mf=+$ and $\zeta_{\mf'}=-$, and $\mf^p$ is the parent of $\mf$ (if $\mf$ is the root of some tree then $t_{\mf^p}$ should be replaced by $t$; similarly for $(\mf')^p$). The relevant notations here and below are defined as in Proposition \ref{propreg}.

Similarly, let $\mf$ be a branching node in $\Gc_{sk}$, which becomes the root $\pf$ and lone leaf $\qf$ of a regular tree $\Tc^{(\mf)}$ in $\Qc$. We must have $k_\pf=k_\qf$. In (\ref{defkg2}), consider the summation in the variables $k_\nf$, where $\nf$ runs over all nodes in $\Tc^{(\mf)}$ other than $\pf$ and $\qf$ (these variables, together with $k_\pf$ and $k_{\qf}$, form a $k_\mf$-decoration of $\Tc^{(\mf)}$), and the integration in the variables $t_\nf$, where $\nf$ runs over all branching nodes in $\Tc^{(\mf)}$, with all the other variables fixed. In the same way, this summation and integration equals, up to some sign $\widetilde{\zeta}(\Tc^{(\mf)})$ and some power of $\delta(2L^{d-1})^{-1}$, the \emph{exact} expression $\Kc_{\Tc^{(\mf)}}^*(t_{\pf^p},t_{\qf},k_\pf)$. Here $\pf^p$ is the parent of $\pf$ (again, if $\pf$ is the root of some tree then $t_{\pf^p}$ should be replaced by $t$).

After performing this reduction for each leaf pair and branching node of $\Gc_{sk}$, we can reduce the summation in (\ref{defkg2}) to the summation in $k_\mf$ for all leaves and branching nodes $\mf$ of $\Gc_{sk}$, i.e. a $(k_1,\cdots,k_{2R})$-decoration of $\Gc_{sk}$. Moreover, we can reduce  the integration in (\ref{defkg2}) to the integration in $t_\mf$ for all branching nodes $\mf$ of $\Gc_{sk}$ (for a regular tree, the time variables $t_{\pf^p}$ and $t_{\qf}$ for $\Gc$ correspond to $t_{\mf^p}$ and $t_\mf$ for $\Gc_{sk}$ where $\mf^p$ is the parent of $\mf$). This implies that
\begin{multline}\label{bigformula2}\Mc_\Gc(t,k_1,\cdots,k_{2R})=\bigg(\frac{\delta}{2L^{d-1}}\bigg)^{m_0}\zeta^*(\Gc_{sk})\sum_{\Is_{sk}}\int_{\Ic_{sk}}\epsilon_{\Is_{sk}}\prod_{\nf\in \Nc_{sk}^*} e^{\zeta_\nf\pi i\cdot\delta L^2\Omega_\nf t_\nf}\,\mathrm{d}t_\nf\\\times{\prod_{\mf\in\Lc_{sk}^*}^{(+)}\Kc_{\Qc^{(\mf,\mf')}}(t_{\mf^p},t_{(\mf')^p},k_\mf)}\prod_{\mf\in\Nc_{sk}^*}\Kc_{\Tc^{(\mf)}}^*(t_{\mf^p},t_{\mf},k_\mf),
\end{multline} where $m_0$ is the scale of $\Gc_{sk}$, $\Ic_{sk}$ is the domain defined in (\ref{timegarden}), $\Is_{sk}$ is a $(k_1,\cdots,k_{2R})$-decoration of $\Gc_{sk}$, the other objects are as before but associated to the garden $\Gc_{sk}$. Moreover in (\ref{bigformula2}), { the first product is taken over all leaves $\mf$ of sign $+$ with $\mf'$ being the leaf paired to $\mf$}, the second product is taken over all branching nodes $\mf$, and $\mf^p$ is the parent of $\mf$.

Using Proposition \ref{propreg}, in (\ref{bigformula2}) we can decompose
\begin{equation}\label{inputdecomp}\Kc_{\Qc^{(\mf,\mf')}}=(\Kc_{\Qc^{(\mf,\mf')}})_{\mathrm{app}}+\Rs_{\Qc^{(\mf,\mf')}},\quad \Kc_{\Tc^{(\mf)}}^*=(\Kc_{\Tc^{(\mf)}}^*)_{\mathrm{app}}+\Rs_{\Tc^{(\mf)}}^*.
\end{equation} Here $(\Kc_{\Qc^{(\mf,\mf')}})_{\mathrm{app}}$ and $(\Kc_{\Tc^{(\mf)}}^*)_{\mathrm{app}}$ are the leading terms in Proposition \ref{propreg}, and each of them is a linear combination of functions of $(t,s)$ multiplied by functions of $k$, which in turn satisfy (\ref{regbound1}) and (\ref{regbound3}); the remainders $\Rs_{\Qc^{(\mf,\mf')}}$ and $\Rs_{\Tc^{(\mf)}}^*$ satisfy (\ref{regbound2}) and (\ref{regbound4}).

We may fix a \emph{mark} in $\{\Lf,\Rf\}$ for each leaf pair and each branching node in $\Gc_{sk}$ which indicates whether we select the \emph{leading} term $(\cdots)_{\mathrm{app}}$ or the \emph{remainder} term $\Rs$ or $\Rs^*$; for a general garden $\Gc$ we can do the same but only for the nodes of its skeleton $\Gc_{sk}$. In this way we can define \emph{marked} gardens, which we still denote by $\Gc$, and expressions of form (\ref{bigformula2}) but with $\Kc_{\Qc^{(\mf,\mf')}}$ and $\Kc_{\Tc^{(\mf)}}^*$ replaced by the corresponding leading or remainder terms, which we still denote by $\Mc_{\Gc}$. By definition, any sum of $\Mc_\Gc$ over unmarked gardens $\Gc$ equals the corresponding sum over marked gardens $\Gc$ for all possible unmarked gardens and all possible markings.

\smallskip
In the next Section we will define the notion of \emph{irregular chains} to exhibit the cancellation between $\Mc_\Gc$ for some different gardens $\Gc$ with specific symmetries.
\subsection{Irregular chains and congruence} The notion of irregular chains for gardens is defined in the same way as for couples, see Section 8.2 of \cite{DH}.
\begin{df}[Definition 8.1 of \cite{DH}]\label{irrechain} Given a garden $\Gc$ (or a paired tree $\Tc$), we define an \emph{irregular chain} to be a sequence of nodes $(\nf_0,\cdots,\nf_q)$, such that (i) $\nf_{j+1}$ is a child of $\nf_j$ for $0\leq j\leq q-1$, and the other two children of $\nf_j$ are leaves, and (ii) for $0\leq j\leq q-1$, there is a child $\mf_j$ of $\nf_j$, which has opposite sign with $\nf_{j+1}$, and is paired (as a leaf) to a child $\pf_{j+1}$ of $\nf_{j+1}$. We also define $\pf_0$ to be the child of $\nf_0$ other than $\nf_1$ and $\mf_0$.
\end{df}
\begin{df}[Definition 8.2 of \cite{DH}]\label{equivirrechain} Consider any irregular chain $\Hc=(\nf_0,\cdots,\nf_q)$. By Definition \ref{irrechain}, we know $\pf_j$ is the child of $\nf_j$ other than $\nf_{j+1}$ and $\mf_j$ for $0\leq j\leq q-1$, thus $\pf_j$ has the same sign with $\nf_j$ (hence it is either its first or third child). Now for two irregular chains $\Hc=(\nf_0,\cdots,\nf_q)$ and $\Hc'=(\nf_0',\cdots,\nf_q')$, with $\pf_j$ and $\pf_j'$ etc. defined accordingly, we say they are \emph{congruent}, if $\zeta_{\nf_0}=\zeta_{\nf_0'}$, and for each $0\leq j\leq q-1$, either $\pf_j$ is the first child of $\nf_j$ and $\pf_j'$ is the first child of $\nf_j'$, or $\pf_j$ is the third child of $\nf_j$ and $\pf_j'$ is the third child of $\nf_j'$, counting from left to right.

In particular, if $q$ and the congruence class (and hence $\zeta_{\nf_0}$) are fixed, then an irregular chain $\Hc$ is uniquely determined by the signs $\zeta_{\nf_j}$ for $1\leq j\leq q$. We relabel the nodes $\nf_j,\pf_j\,(0\leq j\leq q)$ by defining $\{\bff_j,\cf_j\}=\{\nf_j,\pf_j\}$, and that $\bff_j=\nf_j$ if and only if $\zeta_{\nf_j}=+$. Further, we label the two children of $\nf_q$ other than $\pf_q$ as $\ef$ and $\ff$, with $\zeta_\ef=+$ and $\zeta_\ff=-$.
\end{df}
\begin{prop}[Proposition 8.3 of \cite{DH}]\label{congdec} Let $\Hc=(\nf_0,\cdots,\nf_q)$ be an irregular chain. For any decoration $\Ds$ (or $\Es$), its restriction to $\nf_j\,(0\leq j\leq q)$ and their children is uniquely determined by $2(q+2)$ vectors $k_j,\ell_j\in\Zb_L^d\,(0\leq j\leq q+1)$, such that $k_{\bff_j}=k_j$ and $k_{\cf_j}=\ell_j$ for $0\leq j\leq q$, and $k_\ef=k_{q+1}$ and $k_\ff=\ell_{q+1}$. These vectors satisfy
\[k_0-\ell_0=k_1-\ell_1=\cdots =k_{q+1}-\ell_{q+1}:=h,\] and for each $0\leq j\leq q$ we have $\zeta_{\nf_j}\Omega_{\nf_j}=2\langle h,k_{j+1}-k_j\rangle_\beta$. Moreover $\epsilon_{k_{\nf_{j1}}k_{\nf_{j2}}k_{\nf_{j3}}}=\epsilon_{k_{j+1}\ell_{j+1}\ell_j}$, where $(\nf_{j1},\nf_{j2},\nf_{j3})$ are the children of $\nf_j$ from left to right. We say this decoration has \emph{small gap}, \emph{large gap} or \emph{zero gap} with respect to $\Hc$, if we have $0<|h|\leq \frac{1}{100\delta L}$, $|h|\geq \frac{1}{100\delta L}$ or $h=0$.
\end{prop}
\begin{proof} See Proposition 8.3 of \cite{DH}.
\end{proof}
\begin{df}[Definition 8.4 of \cite{DH}]\label{conggen} Let $\Hc=(\nf_0,\cdots,\nf_q)$ be an irregular chain contained in a garden $\Gc$ or a paired tree $\Tc$. If we replace $\Hc$ by a congruent irregular chain $\Hc'=(\nf_0',\cdots,\nf_q')$, then we obtain a modified couple $\Gc'$ or paired tree $\Tc'$ by (i) attaching the same subtree of $\ef$ and $\ff$ in $\Gc$ (or $\Tc$) to the bottom of $\ef'$ and $\ff'$, and (ii) assigning to $\nf_0'$ the same parent of $\nf_0$ and keeping the rest of the couple unchanged.

Given a marked prime garden $\Gc_{sk}$, we identify all the maximal irregular chains $\Hc=(\nf_0,\cdots,\nf_q)$, such that $q\geq 10^3d$, and all $\nf_j$ and their children have mark $\Lf$. For each such maximal irregular chain $\Hc$, consider $\Hc^\circ=(\nf_5,\cdots,\nf_{q-5})$ formed by omitting $5$ nodes at both ends (so that it does not affect other possible irregular chains). We define another marked prime couple $\widetilde{\Gc}_{sk}$ to be \emph{congruent} to $\Gc_{sk}$, if it can be obtained from $\Gc_{sk}$ by changing each of the irregular chains $\Hc^\circ$ to a congruent irregular chain, as described above.

Given a marked garden $\Gc$, we define $\widetilde{\Gc}$ to be congruent to $\Gc$, if it can be formed as follows. First obtain the (marked) skeleton $\Gc_{sk}$ and change it to a congruent marked prime couple $\widetilde{\Gc}_{sk}$. Then, we attach the regular couples $\Qc^{(\mf,\mf')}$ and regular trees $\Tc^{(\mf)}$ from $\Gc$ to the relevant leaf pairs and branching nodes of $\widetilde{\Gc}_{sk}$. Note that if an irregular chain $\Hc^\circ=(\nf_0,\cdots,\nf_q)$ in $\Gc_{sk}$ is replaced by $(\Hc^\circ)'=(\nf_0',\cdots,\nf_q')$ in $\widetilde{\Gc}_{sk}$, with relevant nodes $\mf_j, \pf_j$ etc. as in Definition \ref{irrechain}, then for $0\leq j\leq q-1$, the same regular couple $\Qc^{(\mf_j,\pf_{j+1})}$ is attached to the leaf pair $\{\mf_j',\pf_{j+1}'\}$ in $\widetilde{\Gc}_{sk}$. Similarly, for $1\leq j\leq q$, if $\zeta_{\nf_j'}=\zeta_{\nf_j}$ then the same regular tree $\Tc^{(\nf_j)}$ is placed at the branching node $\nf_j'$ in $\widetilde{\Gc}_{sk}$; otherwise the conjugate regular tree $\overline{\Tc^{(\nf_j)}}$ is placed at $\nf_j'$.

Note that the congruence relation preserves the scale of each tree of a garden; i.e. if $\Gc=(\Tc_1,\cdots,\Tc_{2R})$ and $\widetilde{\Gc}=(\widetilde{\Tc}_1,\cdots,\widetilde{\Tc}_{2R})$ are congruent, then the scale of $\Tc_j$ equals the scale of $\widetilde{\Tc}_j$ for $1\leq j\leq 2R$.
\end{df}
\subsection{Expressions associated with irregular chains}\label{irrered} We shall analyze the expressions associated with irregular chains, in the same way as Section 8.3 of \cite{DH}.

Given \emph{one congruence class} $\Fs$ of marked gardens as in Definition \ref{conggen}, consider the sum
\begin{equation}\label{irrechainsum}\sum_{\Gc\in\Fs}\Mc_\Gc(t,k_1,\cdots,k_{2R}),\end{equation} which is taken over all \emph{marked gardens} $\Gc\in\Fs$. Let the lengths of all the irregular chains $\Hc^\circ$ involved in the congruence class $\Fs$, as in Definition \ref{conggen}, be $q_1,\cdots,q_r$, then $|\Fs|=2^Q$ where $Q=q_1+\cdots+ q_r$. Since these irregular chains do not affect each other, we may focus on one individual chain, say $\Hc^\circ=(\nf_0,\cdots,\nf_q)$; that is, we only sum over $\Gc\in \Fs$ obtained by altering this irregular chain $\Hc^\circ$.

In the summation and integration in (\ref{bigformula2}), we will first fix all the variables $k_\nf$ and $t_\nf$, \emph{except} $k_\nf$ with $\nf\in\{\nf_j,\pf_j,\mf_{j-1}\}\,(1\leq j\leq q)$ and $t_\nf$ with $\nf=\nf_j\,(1\leq j\leq q-1)$, and sum and integrate over these variables. Note that we are fixing $k_{\nf_0}$ and $k_{\pf_0}$ as well as $k_\ef$ and $k_\ff$, in the notation of Definition \ref{equivirrechain}, and are thus fixing $(k_0,\ell_0,k_{q+1},\ell_{q+1})$ and $k_0-\ell_0=k_{q+1}-\ell_{q+1}=h$ as in Proposition \ref{congdec}. It is easy to see that in the summation and integration in (\ref{bigformula2}) over the \emph{fixed variables} (i.e. those $k_\nf$ and $t_\nf$ not in the above list), the summand and integrand does not depend on the way $\Hc^\circ$ is changed, because the rest of the couple is preserved under the change of $\Hc^\circ$, by Definition \ref{conggen}.

We thus only need to consider the sum and integral over the variables listed above. By Proposition \ref{congdec}, this is the same as the sum over the variables $k_j\,(1\leq j\leq q)$, with $\ell_j:=k_j-h$, and integral over the variables $t_j:=t_{\nf_j}\,(1\leq j\leq q-1)$, which satisfies $t_0>t_1>\cdots >t_{q-1}>t_q$ with $t_0:=t_{\nf_0}$ and $t_q:=t_{\nf_q}$. For any possible choice of $\Hc^\circ$ (there are $2^q$ of them), the sum and integral can be written, using (\ref{bigformula2}) and Proposition \ref{congdec}, as
\begin{multline}\label{irrechainexp}\sum_{k_1,\cdots,k_q}\int_{t_0>t_1>\cdots >t_{q-1}>t_q}\bigg(\frac{\delta}{2L^{d-1}}\bigg)^q\prod_{j=1}^q(i\zeta_{\nf_j})\prod_{j=0}^{q}\epsilon_{k_{j+1}\ell_{j+1}\ell_j}\\\times\prod_{j=0}^q e^{2\pi i\delta L^2\langle h,k_{j+1}-k_j\rangle_\beta t_j}\prod_{j=1}^{q}\Kc_{j,\Hc^\circ}\cdot \Kc_{j,\Hc^\circ}^*\,\mathrm{d}t_1\cdots\mathrm{d}t_{q-1}.\end{multline} Here in (\ref{irrechainexp}), we have
\[\Kc_{j,\Hc^\circ}=\Kc_j(t_{j},t_{j-1},k_j-h),\quad \Kc_{j,\Hc^\circ}^*=\Kc_j^*(t_{j-1},t_j,k_j)\] if $\zeta_{\nf_j}=+$, and
\[\Kc_{j,\Hc^\circ}=\Kc_j(t_{j-1},t_{j},k_j),\quad \Kc_{j,\Hc^\circ}^*=\overline{\Kc_j^*(t_{j-1},t_j,k_j-h)}\] if $\zeta_{\nf_j}=-$, where $\Kc_j=(\Kc_{\Qc^{(\pf_j,\mf_{j-1})}})_{\mathrm{app}}$ and $\Kc_j^*=(\Kc_{\Tc^{(\nf_j)}}^*)_{\mathrm{app}}$ where $\Tc^{(\nf_j)}$ is chosen to have sign $+$; note that if $\overline{\Tc}$ is the regular tree conjugate to $\Tc$ then $\Kc_{\overline{\Tc}}^*=\overline{\Kc_\Tc^*}$, and the same holds for the leading contribution $(\cdots)_{\mathrm{app}}$.

Note that, to calculate the above-mentioned contribution (i.e. the sum (\ref{irrechainsum}) with only $\Hc^\circ$ altered), we need to sum over all possible choices of $\Hc^\circ$ (i.e. all possible choices of $\zeta_{\nf_j}\,(1\leq j\leq q)$), in addition to the summation and integration in (\ref{irrechainexp}). This results in the expression
\begin{equation}\label{irrechainexp2}\sum_{\zeta_{\nf_j}\in\{\pm\}\,(1\leq j\leq q)}(\ref{irrechainexp}) = \mathrm{some\ function\ of\ }(k_0,\ell_0,k_{q+1},\ell_{q+1},t_0,t_q).
\end{equation}

Now (\ref{irrechainexp2}) is exactly the same expression that is explicitly calculated in Sections 8.3.1 and 8.3.2 of \cite{DH}, so we shall take the results of such calculations from \cite{DH} and apply them below. There are three cases depending on the value of $h:=k_0-\ell_0$.
\begin{enumerate}
\item The zero gap case ($h=0$): this is very easy, as we have $k_j=\ell_j$, so in view of the $\epsilon_{k_{j+1}\ell_{j+1}\ell_j}$ factors we must have $k_1=\cdots=k_q=k_0$, so the expression (\ref{irrechainexp}) gains a large negative power of $L$, and can be treated in the same way as the small gap term below.
\item The small gap case ($0<|h|\leq (100\delta L)^{-1}$): we have
\begin{equation}\label{summary1}(\ref{irrechainexp2})=(C^+\delta)^{m_{\mathrm{tot}}}(i\delta/2)^q\int_\Rb\int_0^1G(\lambda)\Pc(\lambda,\sigma,k_0,\ell_0)\cdot\dirac(t_0-t_q-\sigma)e^{\pi i\delta L^2\Omega^*t_q}e^{\pi i\lambda t_q}\,\mathrm{d}\sigma\mathrm{d}\lambda.
\end{equation} Here $m_{\mathrm{tot}}$ is the sum of the scales of all regular couples $\Qc_{(\pf_j,\mf_{j-1})}$ and regular trees $\Tc_{(\nf_j)}$, $\Omega^*:=|k_{q+1}|_\beta^2-|\ell_{q+1}|_\beta^2+|\ell_0|_\beta^2-|k_0|_\beta^2$, and the functions $G$ and $\Pc$ satisfy
\begin{equation}\label{coefgpbd}\|\langle \lambda\rangle^{\frac{1}{18}}G\|_{L^1}\lesssim (C^+)^{m_{\mathrm{tot}}},\quad \sup_{\lambda,k_0,\ell_0}\int_0^1|\Pc(\lambda,\sigma,k_0,\ell_0)|\,\mathrm{d}\sigma\lesssim L^{-40d}.
\end{equation}
\item The large gap case ($|h|>(100\delta L)^{-1}$): we have the same expression (\ref{summary1}) and the same bound (\ref{coefgpbd}), but the factor $L^{-40d}$ on the right hand side of the second inequality of (\ref{coefgpbd}) should be replaced by $1$.
\end{enumerate}

Below we will ignore the zero gap case. In the other two cases, we define the new marked garden $\Gc_{sk}^<$ as follows. In the small gap case, and in the large gap case assuming also $k_0\neq k_{q+1}$, we remove the whole chain $\Hc^\circ$ by setting $(\pf_0,\ef,\ff)$ (see Definition \ref{equivirrechain}) to be the three children nodes of $\nf_0$, with the order determined by their signs and the relative position of $\pf_0$, and remove the other nodes (i.e. $(\nf_j,\pf_j)$ for $1\leq j\leq q$ and $\mf_j$ for $0\leq j\leq q-1$). In the large gap case assuming $k_0=k_{q+1}$, we must have $k_0\neq k_q$ since $k_q\neq k_{q+1}$ in view of the factor $\epsilon_{k_{q+1}\ell_{q+1}\ell_q}$ in (\ref{irrechainexp}), so in this case we remove the chain $(\nf_0,\cdots,\nf_{q-1})$, which is the chain $\Hc^\circ$ less one node, in the same way as above.

In either case, denote the scale of $\Gc_{sk}^<$ by $m_0^<$. Note that $\Gc_{sk}^<$ does not depend on the choice of $\Hc^\circ$ in the fixed congruence class (unless in the large gap case, where this dependence does not matter), and for the decoration of $\Gc_{sk}^<$ coming from the decoration of $\Gc_{sk}$, we have $\zeta_{\nf_0}\Omega_{\nf_0}=\Omega^*$ for each choice of $\Hc^\circ$. Then, we can reduce the expression
\begin{equation}\label{onechainsum}\sum_\Gc\Mc_\Gc^{(*)}(t,k_1,\cdots,k_{2R})\end{equation} using (\ref{summary1}), where in (\ref{onechainsum}) the sum is taken over all marked gardens $\Gc$ formed by altering the irregular chain $\Hc^\circ$ in $\Gc_{sk}$, and $(*)$ represents either ``$\mathrm{sg}$" or ``$\mathrm{lg}$", where we restrict to the small gap or large gap case. In fact, using (\ref{summary1}) we have
\begin{multline}\label{onechainsum2}(\ref{onechainsum})=(C^+\delta)^{m_{\mathrm{tot}}}(i\delta/2)^q\cdot\bigg(\frac{\delta}{2L^{d-1}}\bigg)^{m_0^<}\zeta^*(\Gc_{sk}^<)\int_\Rb G(\lambda)\,\mathrm{d}\lambda\int_0^1\mathrm{d}\sigma\cdot\sum_{\Is_{sk}^<}\int_{\widetilde{\Ic}_{sk}^<}\widetilde{\epsilon}_{\Is_{sk}^<}\cdot\Pc(\lambda,\sigma,k_0,\ell_0)\\\times e^{\pi i\lambda t_{\nf_0}}\prod_{\nf\in (\Nc_{sk}^<)^*} e^{\zeta_\nf\pi i\cdot\delta L^2\Omega_\nf t_\nf}\,\mathrm{d}t_\nf{\prod_{\mf\in(\Lc_{sk}^<)^*}^{(+)}\Kc_{\Qc^{(\mf,\mf')}}(t_{\mf^p},t_{(\mf')^p},k_\mf)}\prod_{\mf\in(\Nc_{sk}^<)^*}\Kc_{\Tc^{(\mf)}}^*(t_{\mf^p},t_{\mf},k_\mf).
\end{multline} Here in (\ref{onechainsum2}) the sum is taken over all $(k_1,\cdots,k_{2R})$-decorations $\Is_{sk}^<$ of $\Gc_{sk}^<$, and the other notations are all associated with $\Gc_{sk}^<$, except $\widetilde{\Ic}_{sk}^<$ and $\widetilde{\epsilon}_{\Is_{sk}^<}$; instead, for $\widetilde{\Ic}_{sk}^<$ we add the one extra condition $t_{\nf_0^p}>t_{\nf_0}+\sigma$ (where $\nf_0^p$ is the parent of $\nf_0$) to the original definition (\ref{timegarden}). As for $\widetilde{\epsilon}_{\Is_{sk}^<}$, in the ``$\mathrm{sg}$" case we remove the one factor $\epsilon_{k_{\nf_{01}}k_{\nf_{02}}k_{\nf_{03}}}$ (where $\nf_{0j}$ are the children of $\nf_0$ from left to right) from the original definition (\ref{defcoef}), while in the ``$\mathrm{lg}$" case we set it to be the same as $\epsilon_{\Is_{sk}^<}$. Moreover, the variables $(k_0,\ell_0)$ are defined as in Definition \ref{congdec}, and the functions $G$ and $\Pc$  etc., are as in (\ref{summary1}), which satisfy either (\ref{coefgpbd}) or the alternative version in the ``$\mathrm{lg}$" case. We also insert the corresponding ``$\mathrm{sg}$" or ``$\mathrm{lg}$" cutoffs restricting to $0<|h|\leq 1/(100\delta L)$ or $|h|>1/(100\delta L)$ in (\ref{onechainsum2}). Finally, in the functions $\Kc_{\Tc^{(\nf_0)}}^*$ and $\Kc_{\Qc^{(\mf,\mf')}}$ for the leaf pair $\{\mf,\mf'\}$ containing $\pf_0$, the input variable $t_{\nf_0}$ should be replaced by $t_{\nf_0}+\sigma$.
\begin{rem}\label{newres} In the small gap case, due to the absence of $\epsilon_{k_{\nf_{01}}k_{\nf_{02}}k_{\nf_{03}}}$ in $\widetilde{\epsilon}_{\Is_{sk}^<}$, in the summation in (\ref{onechainsum2}), the decoration $(k_\nf)$ may be resonant at the node $\nf_0$ (i.e. $(k_{\nf_{01}},k_{\nf_{02}},k_{\nf_{03}})\not\in\Sf$, see (\ref{defset})), but it must not be resonant at any other branching node. This resonance may lead to an (at most) $L^{4d}$ loss in the counting estimates in Proposition \ref{gain}, but this can always be covered by the $L^{-40d}$ gain from $\Pc$ in (\ref{coefgpbd}). See Remark \ref{countingrem} for further explanation.
\end{rem}

\subsection{Summary}\label{summarysec} Now we may repeat the reduction described above for every irregular chain $\Hc^\circ$ in $\Gc_{sk}$, noticing that these irregular chains do not affect each other, in the same way as in Section 8.4 of \cite{DH}. Let $\Gc_{sk}^\#$ be the marked garden obtained by removing all the irregular chains $\Hc^\circ$ from $\Gc_{sk}$ as described above in Section \ref{irrered}. This does not depend on the choice of $\Gc_{sk}$ in the fixed congruence class, nor on the choice of $\Gc\in\Fs$. We then have
\begin{multline}\label{irrechainsum2} (\ref{irrechainsum})=(C^+\delta)^{m_1}\bigg(\frac{\delta}{2L^{d-1}}\bigg)^{m_0'}\zeta^*(\Gc_{sk}^\#)\int_{\Rb^\Xi} G(\vlambda)\,\mathrm{d}\vlambda\int_{[0,1]^\Xi}\mathrm{d}\vsigma\sum_{\Is_{sk}^\#}\int_{\widetilde{\Ic}_{sk}^\#}\epsilon_{\Is_{sk}^\#}\Pc(\vlambda,\vsigma,k[\Gc_{sk}^\#])\prod_{\nf\in \Xi}e^{\pi i\lambda_\nf t_\nf}\\\times\prod_{\nf\in (\Nc_{sk}^\#)^*} e^{\zeta_\nf\pi i\cdot\delta L^2\Omega_\nf t_\nf}\,\mathrm{d}t_\nf{\prod_{\mf\in(\Lc_{sk}^\#)^*}^{(+)}\Kc_{\Qc^{(\mf,\mf')}}(t_{\mf^p},t_{(\mf')^p},k_\mf)}\prod_{\mf\in(\Nc_{sk}^\#)^*}\Kc_{\Tc^{(\mf)}}^*(t_{\mf^p},t_{\mf},k_\mf).
\end{multline} Here in (\ref{irrechainsum2}), $m_0'$ is the scale of $\Gc_{sk}^\#$ and $m_1$ is the sum of all the $m_{\mathrm{tot}}$ and $q$ in (\ref{onechainsum2}), the summation is taken over all $k$-decorations $\Is_{sk}^\#$ of $\Gc_{sk}^\#$, and the other notations are all associated with $\Gc_{sk}^\#$, except $\widetilde{\Ic}_{sk}^\#$; instead, for $\widetilde{\Ic}_{sk}^\#$ we add the extra conditions $t_{\nf^p}>t_{\nf}+\sigma_\nf$ (where $\nf^p$ is the parent of $\nf$) to the original definition (\ref{timegarden}), for $\nf\in\Xi$, where $\Xi$ is a subset of the set $(\Nc_{sk}^\#)^*$ of branching nodes. The vector parameters are $\vlambda=\lambda[\Xi]\in\Rb^\Xi$ and $\vsigma=\sigma[\Xi]\in[0,1]^\Xi$ respectively, and $k[\Gc_{sk}^\#]$ is the vector of all the $k_\nf$'s. The functions $G(\vlambda)$ and $\Pc(\vlambda,\vsigma,k[\Gc_{sk}^\#])$ satisfy the bounds
\begin{equation}\label{tensorbd}\bigg\|\prod_{\nf\in\Xi}\langle \lambda_\nf\rangle^{\frac{1}{18}}G\bigg\|_{L^1}\lesssim (C^+)^m,\quad \sup_{\vlambda,k[\Qc_{sk}^\#]}\int_{[0,1]^\Xi}|\Pc(\vlambda,\vsigma,k[\Qc_{sk}^\#])|\,\mathrm{d}\vsigma\lesssim 1.
\end{equation} We also insert various small gap or large gap cutoff functions, and some input variables in some of the $\Kc_{\Qc^{(\mf,\mf')}}$ or $\Kc_{\Tc^{(\mf)}}^*$ functions may be translated by some $\sigma_\nf$, in the same way as in (\ref{onechainsum2}). Finally, the function $\epsilon_{\Is_{sk}^\#}$ may miss a few $\epsilon_{k_{\nf}k_{\nf_1}k_{\nf_2}k_{\nf_3}}$ factors compared to the original definition (\ref{defcoef}), but for each such missing factor we can gain a power $L^{-40d}$ on the right hand side in the second inequality in (\ref{tensorbd}).

At this point, we may expand the functions $\Kc_{\Qc^{(\mf,\mf')}}$ and $\Kc_{\Tc^{(\mf)}}^*$ (or their leading or remainder contributions) using their Fourier $L^1$ (or $X_{\mathrm{loc}}^\kappa$) bounds, and combine the $\Kc$ factors and the $\Pc$ factor in (\ref{irrechainsum2}), to further reduce to the expression
\begin{multline}\label{irrechainsum3} (\ref{irrechainsum})=(C^+\delta)^{\frac{m-m_0'}{2}}\bigg(\frac{\delta}{2L^{d-1}}\bigg)^{m_0'}\zeta^*(\Gc_{sk}^\#){\int_{\Rb^{\Lambda}\times\Rb^2} G(\vlambda)\,\mathrm{d}\vlambda\cdot e^{\pi i(\lambda t+\mu s)}}\int_{[0,1]^\Xi}\mathrm{d}\vsigma\sum_{\Is_{sk}^\#}\int_{\widetilde{\Ic}_{sk}^\#}\epsilon_{\Is_{sk}^\#}\\\times\prod_{\nf\in {\Lambda}} e^{\zeta_\nf\pi i\cdot\delta L^2\Omega_\nf t_\nf}\prod_{\nf\in \Lambda}e^{\pi i\lambda_\nf t_\nf}\,\mathrm{d}t_\nf\cdot \Xc_{\mathrm{tot}}(\vlambda,\vsigma,k[\Gc_{sk}^\#]),
\end{multline}
Here in (\ref{irrechainsum3}) the set {$\Lambda=(\Nc_{sk}^\#)^*$ and $\vlambda=(\lambda[\Lambda],\lambda,\mu)\in\Rb^\Lambda\times\Rb^2$}, the function $G$ is different from the one in (\ref{irrechainsum2}), but still satisfies the same first inequality in (\ref{tensorbd}) with weights in $\lambda$ and $\mu$ also included. Using the second inequality in (\ref{tensorbd}), the $X_{\mathrm{loc}}^\kappa$ bounds for $\Kc_{\Qc^{(\mf,\mf')}}$ and $\Kc_{\Tc^{(\mf)}}^*$ and their components, and the definition of markings $\Lf$ and $\Rf$, we deduce that the function $\Xc_{\mathrm{tot}}$ satisfies
\begin{equation}\label{xtotbound}\int_{[0,1]^\Xi}|\Xc_{\mathrm{tot}}(\vlambda,\vsigma,k[\Gc_{sk}^\#])|\,\mathrm{d}\vsigma\lesssim{\prod_{\lf\in(\Lc_{sk}^\#)^*}^{(+)}\langle k_\lf\rangle^{-40d}}\cdot L^{-2\nu r_0}
\end{equation} uniformly in $\vlambda$, where $r_0$ is the total number of branching nodes and leaf pairs that are marked $\Rf$ in the marked garden $\Gc_{sk}^\#$. In (\ref{xtotbound}) we can also gain a power $L^{-40d}$ per missing factor $\epsilon_{k_{\nf}k_{\nf_1}k_{\nf_2}k_{\nf_3}}$ in $\epsilon_{\Is_{sk}^\#}$, as described above.

Note that the garden $\Gc_{sk}^\#$ is still mixed, and prime. Moreover by definition, it \emph{does not contain an irregular chain of length $>10^3d$ with all branching nodes and leaf pairs marked $\Lf$}. In particular, if $r_0$ is the number of branching nodes and leaf pairs that are marked $\Rf$, $r_{\mathrm{irr}}$ is the number of maximal irregular chains, and $Q$ is the total length of these irregular chains, then we have \begin{equation}\label{typeIcontrol}Q\leq C(r_0+r_{\mathrm{irr}}).
\end{equation}Based on this information, as well as the first inequality in (\ref{tensorbd}) and (\ref{xtotbound}), we will establish an \emph{absolute upper bound} for the expression (\ref{irrechainsum3}). This will be done in the following two sections.
\section{Gardens and molecules}\label{improvecount}
\begin{df}[Definition 9.1 in \cite{DH}]\label{defmol} A \emph{molecule} $\Mb$ is a directed graph, formed by vertices (called \emph{atoms}) and edges (called \emph{bonds}), where multiple and self-connecting bonds are allowed. We will write $v\in \Mb$ and $\ell\in\Mb$ for atoms $v$ and bonds $\ell$ in $\Mb$; we also write $\ell\sim v$ if $v$ is an endpoint of $\ell$. We further require that (i) each atom has at most 2 outgoing bonds and at most 2 incoming bonds (a self-connecting bond counts as outgoing once and incoming once), and (ii) there is no \emph{saturated} (connected) component, where connectedness is always understood in terms of undirected graphs, and a component is saturated if it contains only degree 4 atoms. For a molecule $\Mb$ we define $V$ to be the number of atoms, $E$ the number of bonds and $F$ the number of components. Define $\chi:=E-V+F$.
\end{df}
\begin{df}[Definition 9.3 in \cite{DH}]\label{gardenmol} Given a garden $\Gc$, define the molecule $\Mb$ associated with $\Gc$, as follows. The atoms of $\Gc$ are all the $4$-element subsets formed by a branching node in $\nf\in\Nc^*$ and its three children nodes. For any two atoms, we connect them by a bond if either (i) a branching node is the parent in one atom and a child in the other, or (ii) two leaves from these two atoms are paired with each other. We call this bond a PC (parent-child) bond in case (i) and a LP (leaf-pair) bond in case (ii). Note that multiple bonds are possible, and a self-connecting bond occurs when two sibling leaves are paired.

We fix a direction of each bond as follows. If a bond corresponds to a leaf pair, then it goes from the atom containing the leaf with $-$ sign to the atom containing the leaf with $+$ sign. If a bond corresponds to a branching node $\nf$ that is not a root, suppose $\nf$ is the parent in the atom $v_1$ and is a child in the atom $v_2$, then the bond goes from $v_1$ to $v_2$ if $\nf$ has $+$ sign, and go from $v_2$ to $v_1$ otherwise. See Figure \ref{fig:gardenmol} for an example.
  \begin{figure}[h!]
  \includegraphics[scale=.45]{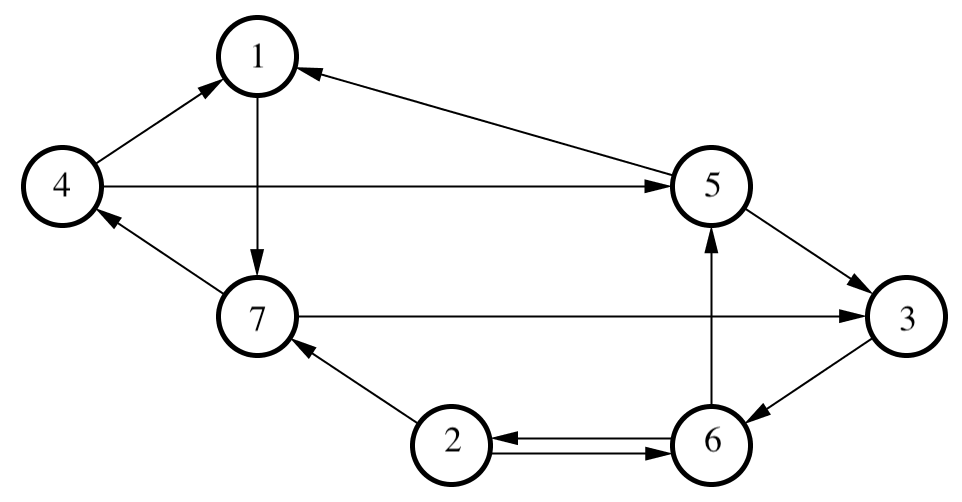}
  \caption{The molecule associated with the garden $\Gc$ in Figure \ref{fig:garden}. Here the atoms $1\!\sim\!4$ correspond to branching nodes $k_1\!\sim\!k_4$, and atoms $5\!\sim\!7$ correspond to branching nodes $\ell_1\!\sim\!\ell_3$ in $\Gc$.}
  \label{fig:gardenmol}
\end{figure}
\end{df}
\begin{prop}\label{basemol} Let $\Gc$ be a \emph{mixed} garden of width $2R$ and scale $m$. Then for the molecule $\Mb$ associated with $\Gc$ as in Definition \ref{gardenmol}, we have $\chi\leq m-R/2$.
\end{prop}
\begin{proof} Let $\Gc=(\Tc_1,\cdots,\Tc_{2R})$. By definition of mixed gardens we know that no $\Tc_i$ and $\Tc_j$ have their leaves completely paired. For the molecule $\Mb$, clearly the number of atoms $V=m$, since each atom in $\Mb$ corresponds to a unique branching node in $\Gc$. Moreover the number of bonds $E=2m-R$. This is because each bond corresponds to either a unique non-root leaf pair or a non-root branching node. The total number of leaf pairs and branching nodes (including roots) is $(m+R)+m=2m+R$, however each root should be subtracted once (it should be excluded from the set of branching nodes if it is a branching node, and should be excluded from the set of leaf pairs if it is a leaf and is paired to another leaf), and only once (because there do not exist two roots that are both leaves and are paired to each other). This implies $E=2m-R$ as there are $2R$ roots.

Finally, for any $\Tc_j\,(1\leq j\leq 2R)$, let $S_j$ be the set of atoms corresponding to branching nodes in $\Tc_j$, then $\Mb$ is the union of all $S_j\,(1\leq j\leq 2R)$. By definition all atoms in $S_j$ are connected to each other. Moreover, if some leaf in $\Tc_j$ is paired to some leaf in $\Tc_{j'}$ then $S_j$ and $S_{j'}$ are also connected to each other. Since the leaves in the union of any \emph{odd} number of trees $\Tc_j$ cannot all be paired with each other (since each $\Tc_j$ has an odd number of leaves), and also that the garden does not contain \emph{two} trees $\Tc_i$ and $\Tc_j$ with their leaves completely paired, we know that any connected component in $\Mb$ must be the union of at least \emph{four} $S_j$, in particular $F\leq R/2$. This implies that $\chi=E-V+F\leq m-R/2$.
\end{proof}
\begin{prop}\label{reconstruction} Fix $m$ and $R$. Given any molecule $\Mb$ of $m$ atoms, the number of gardens $\Gc$ of width $2R$ and scale $m$ that corresponds to $\Mb$ in the sense of Definition \ref{gardenmol} is at most $(CR)!C^m$.
\end{prop}
\begin{proof} This is basically the same as Proposition 9.6 in \cite{DH}. For each atom $v\in\Mb$, each bond $\ell\sim v$ corresponds to a unique node $\nf$ in the $4$-node subset corresponding to $v$. We may assign a code to this pair $(v,\ell)$ indicating the relative position of $\nf$ in this subset (say code $0$ if $\nf$ is the parent node, and codes $1$, $2$ or $3$ if $\nf$ is the left, mid or right child node). In this way we get an encoded molecule which has a code assigned to each pair $(v,\ell)$ where $\ell\sim v$. Clearly if $\Mb$ is fixed then the corresponding encoded molecule has at most $C^{m+R}$ possibilities, so it suffices to reconstruct $\Gc$ from the encoded molecule.

In fact, if the encoded molecule is fixed, then the branching nodes of $\Gc$ uniquely correspond to the atoms of $\Mb$. Moreover, the branching node corresponding to $v_2$ is the $\alpha$-th child of the branching node corresponding to $v_1$, if and only if $v_1$ and $v_2$ are connected by a bond $\ell$ such that the codes of $(v_1,\ell)$ and $(v_2,\ell)$ are $\alpha$ and $0$ respectively. Next, we can determine the leaves of $\Gc$ by putting a leaf as the $\alpha$-th child for each branching node and each $\alpha$, as long as this position is not occupied by another branching node; moreover, the $\alpha$-th child of the branching node corresponding to $v_1$ and the $\beta$-th child of the branching node corresponding to $v_2$ are paired, if and only if $v_1$ and $v_2$ are connected by a bond $\ell$ such that the codes of $(v_1,\ell)$ and $(v_2,\ell)$ are $\alpha$ and $\beta$ respectively.

Finally, note that a node $\nf$ is a root if and only if it is not a child of any other node, so we can uniquely identify the roots of the trees. Permuting these $2R$ roots leads to at most $(CR)!$ choices, and once a permutation is fixed, the garden $\Gc$ will also be fixed as the structure of each tree, as well as the leaf pairing structure, has been fixed as above. This gives at most $(CR)!C^m$ possible choices for $\Gc$. Note that, if one of the trees in $\Gc$ is trivial, then the reconstruction will be slightly different, but this does affect the result.
\end{proof}
\begin{df}[Definition 9.7 in \cite{DH}] We define the \emph{type I} and \emph{type II (molecular) chains} in a molecule $\Mb$, as in Figure \ref{fig:molechain}. Note that type I chains are formed by double bonds, and type II chains are formed by double bonds and pairs of single bonds. For type I chains, we require that the two bonds in any double bond have opposite directions. For type II chains, we require that any pair of single bonds have opposite directions, see Figure \ref{fig:molechain}.
  \begin{figure}[h!]
  \includegraphics[scale=.5]{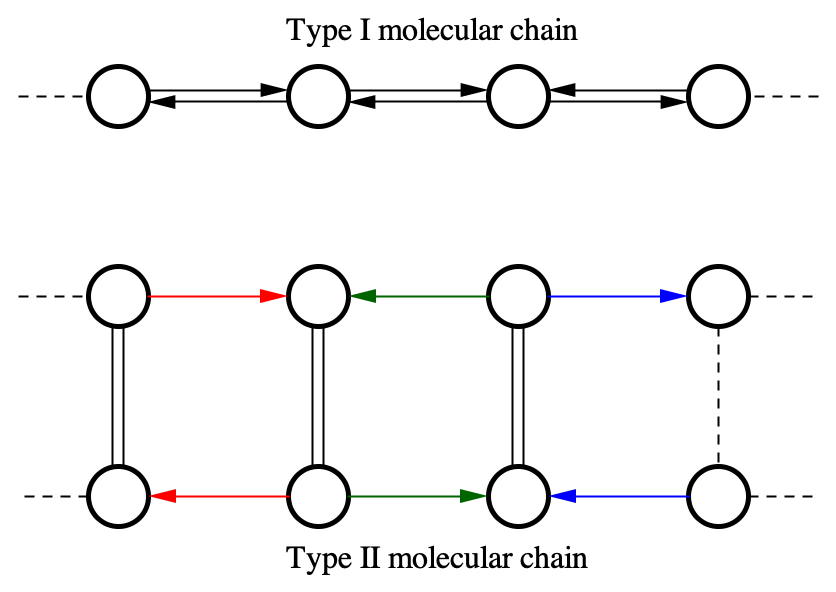}
  \caption{The two types of molecular chains. For type II, the single bonds of the same color are paired single bonds, and must have opposite directions.}
  \label{fig:molechain}
\end{figure} 
\end{df}
Given a molecule $\Mb$, the main subject of this section is the following counting problem associated with $\Mb$, similar to \cite{DH}.
\begin{df}[Definition 9.8 in \cite{DH}]\label{defdecmol} Given a molecule $\Mb$ and a set $S$ of atoms. Suppose we fix (i) $a_\ell\in\Zb_L^d$ for each bond $\ell\in\Mb$, (ii) $c_v\in\Zb_L^d$ for each (non-isolated, same below) atom $v\in\Mb$, assuming $c_v=0$ if $v$ has degree 4, (iii) $\Gamma_v\in\Rb$ for each atom $v$, and (iv) $f_v\in\Zb_L^d$ for each $v\in S$ with $d(v)<4$. Define $\Df(\Mb)$ to be the set of vectors $k[\Mb]:=(k_\ell)_{\ell\in\Mb}$, such that each $k_\ell\in\Zb_L^d$ and $|k_\ell-a_\ell|\leq 1$, and
\begin{equation}\label{decmole}\sum_{\ell\sim v}\zeta_{v,\ell}k_\ell=c_v,\quad\bigg|\sum_{\ell\sim v}\zeta_{v,\ell}|k_\ell|_\beta^2-\Gamma_v\bigg|\leq \delta^{-1}L^{-2}\end{equation} for each atom $v$. Here in (\ref{decmole}) the sum is taken over all bonds $\ell\sim v$, and $\zeta_{v,\ell}$ equals $1$ if $\ell$ is outgoing from $v$, and equals $-1$ otherwise. We also require that (a) the values of $k_\ell$ for different $\ell\sim v$ are all equal given each $v\in S$, and this value equals $f_v$ if also $d(v)<4$, and (b) for any $v\not\in S$ and any bonds $\ell_1,\ell_2\sim v$ of opposite directions (viewing from $v$), we have $k_{\ell_1}\neq k_{\ell_2}$. Note that this actually makes $\Df$ depending on $S$, but we will omit this dependence for simplicity. We say an atom $v$ is \emph{degenerate} if $v\in S$, and is \emph{tame} if moreover $d(v)<4$.

In addition, we may add some extra conditions to the definition of $\Df(\Mb)$. These conditions are independent of the parameters, and have the form of (combinations of) $(k_{\ell_1}-k_{\ell_2}\in E)$ for some bonds $\ell_1,\ell_2\in\Mb$ and fixed subsets $E\subset \Zb_L^d$. Let $\mathtt{Ext}$ be the set of these extra conditions, and denote the corresponding set of vectors $k[\Mb]$ be $\Df(\Mb,\mathtt{Ext})$. We are interested in the quantities $\sup\#\Df(\Mb,\mathtt{Ext})$, where the supremum is taken over all possible choices of parameters $(a_\ell,c_v,\Gamma_v,f_v)$.
\end{df}
\begin{rem}\label{moledec} The vectors $k[\Mb]$ will come from decorations of the garden $\Gc$ from which $\Mb$ is obtained. In fact, if $k[\Gc]$ is a $(k_1,\cdots,k_{2R})$-decoration of $\Gc$, then it uniquely corresponds to a vector $k[\Mb]$. Let $v\in\Mb$ be an atom corresponding to a branching node $\nf\in\Gc$. Then $d(v)=4$ unless $\nf$ is the root of some $\Tc_j$, or some other $\Tc_i$ is a trivial tree paired with a child of $\nf$ (there may be more than one such $i$).

It is easy to check, using Definitions \ref{defdec} and \ref{defdecmol}, that the followings hold. If $d(v)=4$ then $\sum_{\ell\sim v}\zeta_{v,\ell}k_\ell=0$, and $\sum_{\ell\sim v}\zeta_{v,\ell}|k_\ell|_\beta^2=-\zeta_\nf\Omega_\nf$. If $d(v)<4$, then the right hand sides of the above equations should be corrected by suitable algebraic sums of $k_j$ and (or) $k_i$, and $|k_j|_\beta^2$ and (or) $|k_i|_\beta^2$, where $j$ and $i$ are associated with $\nf$ as stated above. Note that all these $k_j$ and $k_i$ are fixed when considering the decoration $k[\Gc]$. Moreover if $(k_{\nf_1},k_{\nf_2},k_{\nf_3})\in\Sf$, then either the values of $k_\ell$ for different $\ell\sim v$ are all equal (and this value equals $k_j$ if $d(v)<4$ where $j$ is as above), or for any bonds $\ell_1,\ell_2\sim v$ of opposite directions we have $k_{\ell_1}\neq k_{\ell_2}$. Note that a degenerate atom corresponds exactly to a branching node $\nf$ for which $\epsilon_{k_{\nf_1}k_{\nf_2}k_{\nf_3}}=-1$.
\end{rem}\begin{prop}\label{gain} Let $\Mb$ be the molecule associated with a mixed garden $\Gc$ of width $2R$ and scale $m$, where $R,m\leq (\log L)^{20}$. Suppose also that $\Mb$ \emph{does not contain any triple bond}. Then, $\Df(\Mb)$ is the union of at most $C^m$ subsets. Each subset has the form $\Df(\Mb,\mathtt{Ext})$, and there exists $0\leq r\leq m$, and a collection of at most $C(r+R)$ molecular chains of either type I or type II in $\Mb$, such that (i) the number of atoms not in one of these chains is at most $C(r+R)$, and (ii) for any type II chain in the collection and any two paired single bonds $(\ell_1,\ell_2)$ in this chain (see Figure \ref{fig:molechain}), the set $\mathtt{Ext}$ includes the condition $(k_{\ell_1}=k_{\ell_2})$. Moreover we have the estimate that  \begin{equation}\label{defect2}\sup\#\Df(\Mb,\mathtt{Ext})\leq (C^+)^m\delta^{-\frac{m+m_1}{2}}L^{(d-1)(m-R/2)-2\nu r},\end{equation} where $m_1$ is the number of atoms in the union of type I chains.
\end{prop}
\begin{rem}\label{countingrem} In view of Remark \ref{newres}, in Definition \ref{defdecmol} we may also fix some set $S^*$ of atoms such that neither (a) nor (b) is required for $v\in S^*$, but we are allowed to multiply the left hand side of (\ref{defect2}) by $L^{-40d\cdot|S^*|}$. In this way we can restate Proposition \ref{gain} appropriately, and the new result can be easily proved with little difference in the arguments, due to the large power gains. For simplicity we will not include this in the proof below.
\end{rem}
\begin{proof}[Proof of Proposition \ref{gain}] The proof is basically the same as the proof of Proposition 9.10 in \cite{DH}. We define the same steps as in Section 9.3 of \cite{DH}, including the good and normal steps, and apply the same algorithm as in Section 9.4 of \cite{DH}. Let the total number of good steps in the process be $r\geq 0$ (we may assume $r\leq m$ up to a constant because the total number of steps is at most $O(m)$), then we may repeat the proofs in Section 9.5 of \cite{DH}. The only difference here is the initial state of the molecule (as $\Mb$ is obtained from a mixed garden rather than a couple), but in the current case we still have $V_{<4}+F=O(R)$, where $V_{<4}$ and $F$ are the number of atoms with degree $<4$ and the number of connected components and the constant in $O$ depends only on $d$.

Note that in the proof of of Proposition 9.10 in \cite{DH}, the quantities that are monitored include $V$, $E$, $F$, $V_j$ for $0\leq j\leq 4$ (which is the number of atoms with degree $j$), $V_2^*$ (which is the number of degree $2$ atoms with two single bonds), and $\xi$ (which is the number of ``special bonds" connecting two degree $3$ atoms that have a special form, see Definition 9.12 of \cite{DH}). Since $V_{<4}+F=O(R)$, it is clear that in the beginning, the value of each of these quantities in the current case is the same as in \cite{DH}, up to errors  of size $O(R)$. Thus, the same proof as in \cite{DH} yields that $\Mb$ contains at most $C(R+r)$ type I or II molecular chains, such that the number of atoms not in one of these chains is at most $C(R+r)$. Moreover
\[\sup\#\Df(\Mb,\mathtt{Ext})\leq (C^+)^m\delta^{-\kappa}L^{(d-1)\gamma},\] where $\kappa$ and $\gamma$ are calculated retrospectively from the algorithm, as described in Section 9.2 of \cite{DH}. The calculation for $\kappa$ is the same up to $O(r)$ errors, so we have $\kappa=\frac{m+m_1}{2}$ up to errors $O(R+r)$ which are acceptable. To calculate $\gamma$, note that in \cite{DH} we are actually calculating $\gamma-\chi$, and the same proof yields that $(d-1)(\gamma-\chi)\leq -2\nu r$ for the initial molecule. Now by Proposition \ref{basemol} we know $\chi\leq m-R/2$, hence
\[\sup\#\Df(\Mb,\mathtt{Ext})\leq (C^+)^m\delta^{-\frac{m+m_1}{2}}L^{(d-1)(m-R/2)-2\nu r},\] as desired.
\end{proof}
\section{$L^1$ coefficient bounds}\label{l1bd} We now return to the study of the expression (\ref{irrechainsum3}). Let $\Gc_{sk}^\#$ and $(r_0,r_{\mathrm{irr}})$ be as in Section \ref{summarysec}. For simplicity, in this section we will write $\Gc_{sk}^\#$ simply as $\Gc$, and the associated sets $(\Nc_{sk}^\#)^*$ as $\Nc^*$ etc. Recall, by (\ref{typeIcontrol}), that the total length of the irregular chains in $\Gc$ is at most $C(r_0+r_{\mathrm{irr}})$. Let $\Xi$ be a subset of $\Nc^*$, we may define, as in (\ref{irrechainsum3}), the function
\begin{equation}\label{modifiedtimeint}\Uc_{\Gc}(t,s,\vsigma,\alpha[\Nc^*])=\int_{\widetilde{\Ic}}\prod_{\nf\in\Nc^*}e^{\pi i\alpha_\nf t_\nf}\,\mathrm{d}t_\nf,
\end{equation} where $\vsigma=\sigma[\Xi]\in[0,1]^\Xi$, and the domain $\widetilde{\Ic}$ is defined as in (\ref{timegarden}), but with the extra conditions $t_{\nf^p}>t_\nf+\sigma_\nf$ for $\nf\in\Xi$, where $\nf^p$ is the parent of $\nf$. Then, let $m_0'$ be the scale of $\Gc$, we can write
\begin{multline}\label{finalexp}(\ref{irrechainsum3})=(C^+\delta)^{\frac{m-m_0'}{2}}\bigg(\frac{\delta}{2L^{d-1}}\bigg)^{m_0'}\zeta^*(\Gc){\int_{\Rb^{\Nc^*}\times\Rb^2}} G(\vlambda)\,\mathrm{d}\vlambda\int_{[0,1]^\Xi}\mathrm{d}\vsigma\\\times\sum_{\Is}\epsilon_\Is\Uc_{\Gc}\big(t,s,\vsigma, (\delta L^2\zeta_\nf\Omega_\nf+\lambda_\nf)_{\nf\in\Nc^*}\big)\cdot \Xc_{\mathrm{tot}}(\vlambda,\vsigma,k[\Gc]).
\end{multline}

Let $\Mb$ be the molecule associated with $\Gc$ as in Definition \ref{gardenmol}. It is easy to see that $\Mb$ contains no triple bond, as triple bonds in $\Mb$ can only come from $(1,1)$-mini couples and mini trees (as in Definition \ref{defreg}) in $\Gc$. By the proofs in Section \ref{improvecount}, we can introduce at most $C^{m_0'}$  sets of extra conditions $\mathtt{Ext}$, such that the summation in $\Is=k[\Gc]$ in (\ref{finalexp}) can be decomposed into the summations with each of these sets of extra conditions imposed on $k[\Gc]$. Moreover, for each choice of $\mathtt{Ext}$ there is $1\leq r_1\leq n_0'$ such that the conclusion of Proposition \ref{gain}, including (\ref{defect2}), holds true (with $r$ replaced by $r_1$).

Notice that a type I chain in $\Mb$ can \emph{only} be obtained from either one irregular chain, or the union of two irregular chains in $\Gc$; for couples this can be proved in the same way as in Section 10.1.2 of \cite{DH} (which deals with type II chains), and the same proof works also for gardens. Therefore, the total length $p$ of type I chains in $\Mb$ is bounded by the total length of irregular chains in $\Gc$, which is at most $C(r_0+r_{\mathrm{irr}})$. However, each irregular chain in $\Gc$ also corresponds to a type I chain in the base molecule, so $r_{\mathrm{irr}}\leq C(r_1+R)$, hence $p\leq C(r+R)$, where $r=r_0+r_1$. This means the number of atoms in $\Mb$ that are not in one of the (at most $C(r+R)$) type II chains is at most $C(r+R)$.

Now, suppose $\nf$ and $\nf'$ are two branching nodes in $\Gc$ which correspond to two atoms in $\Mb$ \emph{that are connected by a double bond in a type II chain}, then we must have $\zeta_{\nf'}\Omega_{\nf'}=-\zeta_\nf\Omega_\nf$ under the extra conditions in $\mathtt{Ext}$, see Remark \ref{moledec}. In fact we will restrict $\{\nf,\nf'\}$ to the \emph{interior} of this type II chain by omitting $5$ pairs of atoms at both ends of the chain, in the same way as in Definition \ref{conggen}. Then, we make such $\{\nf,\nf'\}$ a pair, and choose one node from each such pair to form a set $\widetilde{\Nc}^{ch}$. If it happens that one of $\{\nf, \nf'\}$ is a parent of the other, we assume the parent belongs to $\widetilde \Nc^{ch}$. Let $\Nc^{rm}$ be the set of branching nodes not in these pairs, and define $\widetilde{\Nc}=\widetilde{\Nc}^{ch}\cup\Nc^{rm}$. 

We will be interested in estimates on the function $\Uc_\Gc$ in \eqref{modifiedtimeint} where $\alpha_\nf=\delta L^2\zeta_\nf\Omega_\nf+\lambda_\nf$, which means that $\alpha_{\nf}+\alpha_{\nf'}=\mu_\nf$ for each $\nf\in\widetilde{\Nc}^{ch}$, where $\nf'$ is the node paired to $\nf$ and $\mu_\nf=\lambda_\nf+\lambda_{\nf'}$ is a parameter depending on $\vlambda$. Under this assumption on $\alpha_{\nf}$, we can write
\begin{equation}\label{eqnuv}\Uc_{\Gc}(t,s,\vsigma,\alpha[\Nc^*])=\Vc_{\Gc}(t,s,\vsigma,\alpha[\widetilde{\Nc}])
\end{equation} for some function $\Vc_{\Gc}$. This function actually depends also on the parameters $\mu_\nf$ for $\nf\in\widetilde{\Nc}^{ch}$, but we will omit this for notational convenience. We then have prove the following:
\begin{prop}\label{section8main} {Suppose $\Gc$ has scale $m_0'$.} For each $\nf\in\widetilde{\Nc}$, suppose $S_\nf\subset\Zb$ and $\#S_\nf\leq L^{10d}$. Then, uniformly in $(t,s)$, in the choices of $(S_\nf)_{\nf\in\widetilde{\Nc}}$, and in the parameters $(\mu_\nf)_{\nf\in\widetilde{\Nc}^{ch}}$, we have
	\begin{equation}\label{uniformL1}\delta^{m_0'/4}\cdot\sum_{(m_\nf):m_\nf\in S_\nf}\sup_{(\alpha_\nf):|\alpha_\nf-m_\nf|\leq 1}\sup_{\vsigma}\big|\Vc_{\Gc}(t,s,\vsigma,\alpha[\widetilde{\Nc}])\big|\leq (C^+)^{m_0'}L^{C(r+R)\sqrt{\delta}}(\log L)^{C(r+R)},
	\end{equation} where $r=r_0+r_1$.
\end{prop}	
\begin{proof} The proof is exactly the same as the proof of Proposition 10.1 in \cite{DH}, except that all the ``couples" should be replaced by ``gardens". The reason is that, the proof in \cite{DH} goes by induction; moreover in each inductive step we remove either two branching nodes or one chain containing modules A and B (see Sections 10.1--10.2 of \cite{DH}). In either case this step involves at most two trees and the other trees are not affected, so the proof is the same for couples and for general gardens. 

In the proof in \cite{DH} we have also introduced the simpler structures (for the purpose of induction) of unsigned couples and double-trees, which are naturally replaced by unsigned gardens and multi-trees (the collection of $2R$ trees with some branching nodes paired, compared to two trees in \cite{DH}). The rest of the proof is exactly the same. Note also that the exponents $Cr$ in Proposition 10.1 of \cite{DH} are replaced by $C(r+R)$ as the number of type II chains in $\Mb$, as well as the number of atoms not in one of these chains, is now $C(r+R)$ instead of $Cr$ due to Proposition \ref{gain}.
\end{proof}
\section{Proof of Theorem \ref{main}} \label{sec:conclusion}In this section we prove Propositions \ref{mainprop} and \ref{mainprop2}, which completes the proof of Theorem \ref{main}. 
\subsection{Proof of Proposition \ref{mainprop}} Note that, if $\Gc$ and $\Gc'$ are congruent in the sense of Definition \ref{conggen}, then they have the same width, same signature, and the same scale for each of their component trees; moreover $\Gc'$ is mixed if and only if $\Gc$ is mixed. Thus, the sum in (\ref{mixedsum}) can be decomposed into different terms, where each term has the form (\ref{irrechainsum}) for one congruence class $\Fs$.

For any fixed $\Fs$, consider (\ref{irrechainsum}), which then equals (\ref{irrechainsum3}) and (\ref{finalexp}). Note that in (\ref{finalexp}) the $\Gc$ actually means $\Gc_{sk}^\#$ by our notation. Using the decay factors in (\ref{xtotbound}) we can restrict to the subset where $|k_\lf-a_\lf|\leq 1\,(\forall\lf\in(\Lc_{sk}^\#)^*)$ for some fixed parameters $(a_\lf)$, with summability in $(a_\lf)$ guaranteed. Using the first inequality in (\ref{tensorbd}), we may also fix the value of $\vlambda$ (and hence $\mu_\nf$).

As in Section \ref{improvecount}, by decomposing into at most $C^{m_0'}$ terms (where $m_0'$ is the scale of $\Gc_{sk}^\#$), we can add the set of extra conditions $\mathtt{Ext}$, which also defines the sets $\widetilde{\Nc}$ (as in Section \ref{l1bd}), etc., and the value $r_1\geq 1$. Let $r=r_0+r_1$ as above, then thanks to $\mathtt{Ext}$, we can use (\ref{eqnuv}) to reduce $\Uc_{\Gc_{sk}^\#}$ to $\Vc_{\Gc_{sk}^\#}$. Moreover, for each $\nf\in\widetilde{\Nc}$, the value $\delta L^2\zeta_\nf\Omega_\nf+\lambda_\nf$ belongs to some subset of $\Rb$ of cardinality at most $L^{3d}$, as $k[\Gc_{sk}^\#]$ varies (this is because each $k_\nf$ belongs to a ball of radius at most $n\leq (\log L)^{20}$ under our assumptions). In particular the value $m_\nf=\lfloor \delta L^2\zeta_\nf\Omega_\nf+\lambda_\nf\rfloor$ belongs to a set $S_\nf\subset\Zb$ with cardinality at most $L^{3d}$, for all possible choices of $k[\Gc_{sk}^\#]$.

To estimate (\ref{finalexp}) with $\vlambda$ fixed, we first integrate in $\vsigma$. Using (\ref{xtotbound}), we can estimate (\ref{finalexp}) using
\begin{equation}\label{finalexp2}\sum_{\Is_{sk}^\#}|\epsilon_{\Is_{sk}^\#}|\cdot\sup_{\vsigma}\big|\Vc_{\Gc_{sk}^\#}\big(t,s,\vsigma, (\delta L^2\zeta_\nf\Omega_\nf+\lambda_\nf)_{\nf\in\widetilde{\Nc}}\big)\big|,
\end{equation} where $\Is_{sk}^\#=k[\Gc_{sk}^\#]$ is a $k$-decoration of $\Gc_{sk}^\#$ {(we also have additional factors that will be collected at the end)}. We next fix the values of $m_\nf\in S_\nf$ for each $\nf$; note that then
\[\sup_{\vsigma}\big|\Vc_{\Gc_{sk}^\#}\big(t,s,\vsigma, (\delta L^2\zeta_\nf\Omega_\nf+\lambda_\nf)_{\nf\in\widetilde{\Nc}}\big)\big|\leq\sup_{(\alpha_\nf):|\alpha_\nf-m_\nf|\leq 1}\sup_{\vsigma}\big|\Vc_{\Gc_{sk}^\#}(t,s,\vsigma,\alpha[\widetilde{\Nc}])\big|\] by definition, so if we use (\ref{uniformL1}) to sum over $(m_\nf)$ in the end, we can further estimate (\ref{finalexp2}) using
\begin{equation}\label{finalexp3}\sum_{\Is_{sk}^\#}|\epsilon_{\Is_{sk}^\#}|\cdot\prod_\lf\mathbf{1}_{|k_\lf-a_\lf|\leq 1}\prod_\nf\mathbf{1}_{|\Omega_\nf-b_\nf|\leq \delta ^{-1}L^{-2}},
\end{equation} where $a_\lf$ and $b_\nf$ are constants, and we also include the conditions in $\mathtt{Ext}$. {Now (\ref{finalexp3}) is almost exactly the counting problem $\Df(\Mb,\mathtt{Ext})$ stated in Definition \ref{defdecmol}, due to Remark \ref{moledec}, except that we only assume $|k_\lf-a_\lf|\leq 1$ for \emph{leaves} $\lf$. However, for any branching node $\nf$ there exists a child $\nf'$ of $\nf$ such that $k_\nf\pm k_{\nf'}$ belongs to a fixed ball of radius $\mu_\nf^\circ$ (with $\mu_\nf^\circ$ defined Lemma \ref{auxlem}), so by using (\ref{auxineq}), one can reduce (\ref{finalexp3}) to at most $C^{m_0'}$ counting problems, each of which having exactly the same form as $\Df(\Mb,\mathtt{Ext})$ in Definition \ref{defdecmol}.} Therefore, (\ref{finalexp3}) can be bounded using Proposition \ref{gain} {(and using Remark \ref{countingrem} if necessary)}. Collecting all the factors appearing in the above estimates, we get that
\begin{multline}\label{finalest1}|(\ref{irrechainsum})|\leq (C^+)^m\delta^{(m-m_0')/2}\delta^{3m_0'/4} L^{-(d-1)m_0'} L^{-2\nu r _0}\\\times L^{C(r+R)\sqrt{\delta}}(\log L)^{C(r+R)}\delta^{-(m_0'+q)/2} L^{(d-1)(m_0'-R/2)-2\nu r_1},
\end{multline} which is then bounded by $(C^+\delta^{1/4})^{m} L^{-3\nu (r+R)/2}\delta^{-q/2}$, where $q$ is the total length of type I chains in the molecule obtained from $\Qc_{sk}^\#$. We know $q\leq C(r+R)$ so $\delta^{-q/2}\leq L^{\nu (r+R)/2}$, which implies that
\begin{equation}\label{finalbound}|(\ref{irrechainsum})|\leq (C^+\delta^{1/4})^{m} L^{-\nu(r+R)}.
\end{equation}

Finally, suppose we fix $r$, then the molecule associated with $\Gc_{sk}^\#$ (see Definition \ref{gardenmol}) is, up to at most $C(r+R)$ remaining atoms, a union of at most $C(r+R)$ type II chains with total length at most $m_0'$. This clearly has at most $(C(r+R))!C^m$ possibilities. By Proposition \ref{reconstruction}, the number of choices for $\Gc_{sk}^\#$ is also at most $(C(r+R))!C^m$. To form $\Gc_{sk}$ from $\Gc_{sk}^\#$ one needs to insert at most $C(r+R)$ irregular chains with total length at most $m$, which also has at most $C^m$ possibilities. Finally, using Corollary \ref{skeletonrecover}, we see that $\Gc$ has at most $(C(r+R))!C^m$ choices. The number of choices for markings, as well as $\mathtt{Ext}$, are also at most $C^m$ and can be accommodated. This means that, if we decompose (\ref{mixedsum}) into terms of form (\ref{irrechainsum}), then further decompose by markings and/or $\mathtt{Ext}$ etc., then each of the resulting term has an index $r\geq 0$, such that each term of index $r$ is bounded by $(C^+\delta^{1/4})^{m} L^{-\nu(r+R)}$, see (\ref{finalbound}), and that the number of terms with index $r$ is at most $(C(r+R))!C^m$. Therefore
\begin{equation}(\ref{mixedsum})\leq\sum_{r=0}^m (C^+\delta^{1/4})^{m} L^{-\nu(r+R)}\cdot (C(r+R))!C^m\leq  (C^+\delta^{1/4})^{m}L^{-\nu R},
\end{equation} because in any case $r+R$ is bounded by a power of $\log L$, which is $\ll L^\nu$. This completes the proof of Proposition \ref{mainprop}.
\subsection{Proof of Proposition \ref{mainprop2}}\label{remterm} The proof is almost identical with the corresponding proofs in \cite{DH}, which we briefly present here.

First, by Chebyshev's inequality, to prove (\ref{largedevest}) it suffices to show that
\begin{equation}\label{largedevest2}\mathbb{E}\big|\sup_{k,t}\langle k\rangle^{9d}(\Jc_n)_k(t)\big|^2\leq (C^+\sqrt \delta)^{n}L^{100d}
\end{equation} and
\begin{equation}\label{largedevest3}\mathbb{E}\big|\sup_{k,t}\langle  k\rangle^{9d}\Rc_k(t)\big|^2\leq(C^+\sqrt{\delta})^NL^{100d}.
\end{equation} Note that due to our choice $N=\lfloor(\log L)^4\rfloor$ instead of $N=\lfloor\log L\rfloor$, the proof of (\ref{largedevest}) is conceptually easier than \cite{DH} as we do not need the hypercontractivity property (Lemma A.3 of \cite{DH}) or the higher moment estimates.

Now, to prove (\ref{largedevest2}), we argue as in the proof of Proposition 12.1 of \cite{DH} (but with $p$ replaced by $2$), and apply Gagliardo-Nirenberg to bound the left hand side of (\ref{largedevest2}), up to a multiple $L^{10d}$, by
\[\sup_{k,t}\langle k\rangle^{20d}\big(\Eb|(\Jc_n)_k(t)|^2+\Eb|\partial_t(\Jc_n)_k(t)|^2\big),\] which is then bounded by $(C^+\sqrt{\delta})^nL^{40d}$ in the same way as \cite{DH}. In fact, the bound for $\Eb|(\Jc_n)_k(t)|^2$ is as in Proposition 2.5 of \cite{DH} (again our choice $N=\lfloor(\log L)^4\rfloor$ here does not affect the proof), while the bound for $\Eb|\partial_t(\Jc_n)_k(t)|^2$ is as in (12.4) of \cite{DH}, which is proved by similar arguments. This settles (\ref{largedevest2}). The proof of (\ref{largedevest3}) is the same, except that $\Jc_n$ is replaced by $\Rc$ and $n$ is replaced by $N$.

\medskip
Finally, to prove (\ref{operatordev}), again by Chebyshev's inequality, it suffices to show that the kernel $(\Ls^n)_{k\ell}^{\zeta}(t,s)$ of the $\Rb$-linear operator $\Ls^n$ (with $\zeta\in\{\pm\}$ indicating the linear and conjugate linear parts) can be decomposed as 
\[(\Ls^n)_{k\ell}^{\zeta}(t,s)=\sum_{n\leq m\leq N^3}(\Ls^n)_{k\ell}^{m,\zeta}(t,s)\] and that for each $n\leq m\leq N^3$, the kernel $(\Ls^n)_{k\ell}^{m,\zeta}(t,s)$ satisfies that
\begin{equation}\label{largedevest4}\Eb\big|\sup_{k,\ell}\sup_{0\leq s<t\leq 1}\langle k-\ell\rangle^{9d}(\Ls^n)_{k\ell}^{m,\zeta}(t,s)\big|^2\leq (C^+\sqrt{\delta})^mL^{100d}.\end{equation} Now the decomposition is provided as in Proposition 11.2 of \cite{DH}, and (\ref{largedevest4}) is proved as in the proof of Proposition 12.2 of \cite{DH} (in particular this proof does require the hypercontractivity property). Note that in that proof, we actually further decompose $(\Ls^n)_{k\ell}^{m,\zeta}(t,s)$ into $(\Ls^n)_{M,k\ell}^{m,\zeta}(t,s)$ for dyadic $M$, and proves (\ref{largedevest4}) for $(\Ls^n)_{M,k\ell}^{m,\zeta}$ with the right hand side involving a negative power of $M$; see (12.10) of \cite{DH}. Both proofs carry over to the current case with our choice $N=\lfloor(\log L)^4\rfloor$ with out any change, which then proves (\ref{operatordev}) and completes the proof of Proposition \ref{mainprop2}.
\section{The non-Gaussian case}\label{sec:nongaussian} In this section we briefly discuss the non-Gaussian case, i.e. Theorem \ref{main2}, which we prove in Sections \ref{subst}--\ref{overgar} (Theorem \ref{main3} basically follows from it and is proved separately in Section \ref{evodens}). Since much of the proof will be identical with Theorem \ref{main}, we will only elaborate on the parts where the proofs are different.

First, in the Gaussian case our proof yields uniform estimates as long as $R\leq \log L$ (or $R\leq 2\log L$); here we will make slightly stronger assumptions $R\leq\log L/(\log\log L)^2$. Again we may consider $2R$ at some places, but it does not affect the result.

Next, using the expansion (\ref{expand}), we can reduce the proof of Theorem \ref{main2} to analyzing the correlations \begin{equation}\label{corrnew}\Eb\bigg(\prod_{j=1}^{2R}(\Jc_{\Tc_j})_{k_j}^{\zeta_j}(t)\bigg).\end{equation} The rest of the proof, including the treatment of the remainder term $b$, can be done in the same way as with these correlations, see Section \ref{remterm}. Specifically, the Gaussian hypercontractivity inequality, which is used in Section \ref{remterm}, can be substituted by similar inequalities for the current density function thanks to the assumed bound on $\mu_r \leq (Cr)!$; an instance of such argument can be found in Lemma 3.1 of \cite{DH0} which treats the particular case of the uniform distribution on the unit circle, but the general case can be treated in the same manner. Therefore, below we will focus on the study of (\ref{corrnew}).
\subsection{A substitute for Isserlis' theorem}\label{subst}
The obvious difference in the study of (\ref{corrnew}) in the non-Gaussian case is that the Isserlis' theorem is not available. Instead we have the following substitute:
\begin{lem}\label{isserlisnew} Recall all the random variables $\eta_k$ are i.i.d. with radial law, and $\Eb|\eta_k|^{2r}=\mu_r$ with $\mu_1=1$ and $\mu_r\leq (Cr)!$ for a positive integer $C$. Then, for any $k_j\in\Zb_L^d\,(1\leq j\leq n)$ and $\zeta_j\in\{\pm\}\,(1\leq j\leq n)$ we have
\begin{equation}\label{isserlis2}\Eb\bigg(\prod_{j=1}^n\eta_{k_j}^{\zeta_j}\bigg)=\sum_{\Oc}\lambda(\Oc)\cdot\sum_{\Os}\prod_{A\in\Os}\prod_{j,j'\in A}\mathbf{1}_{k_j=k_{j'}}.
\end{equation} 

Here in (\ref{isserlis2}), $\Oc$ runs over all partitions of $n$ into \emph{even} positive integers (in particular if $n$ is odd then the right hand side of (\ref{isserlis2}) is zero). For fixed $\Oc$, the $\Os$ runs over all \emph{over-pairings} of the set $\{1,\cdots,n\}$ subordinate to $\Oc$ (we use the notation $\Os\models\Oc$), which are partitions of $\{1,\cdots,n\}$ such that the cardinalities of the subsets exactly form the partition $\Oc$, and for each subset $A$ exactly half of the signs $\zeta_j\,(j\in A)$ are $+$ and half are $-$.

The coefficients $\lambda(\Oc)$ are constants depending only on $\Oc$ (and $n$), and $\lambda(2,\cdots,2)=1$; in general we have $\lambda(2,\cdots,2,2a_1,\cdots,2a_r)=\lambda(2a_1,\cdots,2a_r)$. Moreover, let $q$ be the sum of the elements in $\Oc$ that are \emph{at least $4$}, then we have $|\lambda(\Oc)|\leq C_1^nn^{C_1q}$ for some constant $C_1\gg C$.
\end{lem}
\begin{proof} We may assume $n$ is even and half the signs $\zeta_j\,(1\leq j\leq n)$ are $+$ and half are $-$, since otherwise both sides of (\ref{isserlis2}) are zero. Denote by $|\Oc|$ the number of elements in $\Oc$  (counted with multiplicity). For two partitions $\Oc$ and $\Oc'$, we write $\Oc'\preceq\Oc$, if $\Oc'$ can be formed by further partitioning some elements in $\Oc$ into even integers (also define $\succeq$ and $\prec$ etc. accordingly). Similarly for set partitions $\Os$ and $\Os'$, we write $\Os'\preceq\Os$ if $\Os'$ can be formed by further partitioning some subsets in $\Os$ (still keeping half of the signs $+$ and half $-$ in each subset).

Now, for $\Oc'\preceq \Oc$, we define $\xi_{\Oc,\Oc'}$ as follows: given a partition $\Os\models\Oc$, consider the number of partitions $\Os'\preceq\Os$ such that $\Os'\models\Oc'$. The number of choices for $\Os'$ is independent of the choice of $\Os$, and we define it to be $\xi_{\Oc,\Oc'}$. Obviously $\xi_{\Oc,\Oc}=1$. We define the coefficients $\lambda(\Oc)$ for each $\Oc$, such that they satisfy the following recurrence relation: first $\lambda(2,\cdots,2)=1$, and for each $\Oc$, we have
 \begin{equation}
 \label{recur}
 \prod_{2b\in\Oc}\mu_b=\sum_{\Oc'\preceq\Oc}\xi_{\Oc,\Oc'}\cdot\lambda(\Oc').
 \end{equation} Here the product is taken over all elements $2b$ appearing in $\Oc$, counted with multiplicity. Clearly the values of of $\lambda(\Oc)$ for each $\Oc$ are uniquely determined by (\ref{recur}). To prove $\lambda(2,\cdots,2,2a_1,\cdots,2a_r)=\lambda(2a_1,\cdots,2a_r)$, we simply notice $\mu_1=1$ and that if $\Oc$ contains a certain number of terms $2$, then any $\Oc'\preceq\Oc$ must contain at least the same number of $2$'s. Then we may proceed inductively using (\ref{recur}).
 
 Next we prove (\ref{isserlis2}) with $\lambda(\Oc)$ defined by (\ref{recur}). Assume all different values of these $k_j$ are $m_i\,(1\leq i\leq r)$, where for each $j$ there are $a_i$ copies of $m_i$ with corresponding $\zeta_j=+$, and $b_j$ copies with $\zeta_j=-$. We may assume $a_i=b_i$ (otherwise it is easy to check that both sides of (\ref{isserlis2}) are zero), hence the left hand side of (\ref{isserlis2}) is reduced to
 \[\Eb\bigg(\prod_{i=1}^r|\eta_{m_i}|^{2a_i}\bigg)=\prod_{j=1}^r\mu_{a_i}.\] Note also that $\Oc_*=(2a_1,\cdots,2a_r)$ is a partition of $n$ and the sets $\{j:k_j=m_i\}$ form a partition $\Os_*\models\Oc_*$. Moreover, on the right hand side of (\ref{isserlis2}), in order for the product $\prod_{A\in\Os}\prod_{j,j'\in A}\mathbf{1}_{k_j=k_{j'}}$ to be nonzero, one must have $\Oc\preceq\Oc_*$ and $\Os\preceq\Os_*$, and in this case this product equals $1$. Thus the right hand side of (\ref{isserlis2}) equals
 \[\sum_{\Oc\preceq\Oc_*}\lambda(\Oc)\sum_{\Os\preceq\Os_*,\Os\models\Oc}1=\sum_{\Oc\preceq\Oc_*}\xi_{\Oc_*,\Oc}\cdot\lambda(\Oc)=\prod_{i=1}^r\mu_{a_i}\] using the definition of $\xi_{\Oc_*,\Oc}$ and (\ref{recur}), as desired.
 
Next we prove that 
 \begin{equation}\label{induct}|\lambda(\Oc)|\leq \frac{(n/2)!}{(|\Oc|)!}\prod_{2b\in\Oc}(C_0b)!
 \end{equation} for $C_0=C+40$. The base case $\Oc=(2,\cdots,2)$ is clear. By induction, and using that $\mu_r\leq (Cr)!$, we only need to prove that for any $\Oc$, we have
 \begin{equation}\label{inductstep}\sum_{\Oc'\prec\Oc}\xi_{\Oc,\Oc'}\frac{(|\Oc|)!}{(|\Oc'|)!}\frac{\prod_{2b\in\Oc'}(C_0b)!}{\prod_{2b\in\Oc}(C_0b)!}\leq\frac{1}{2}.
 \end{equation} Now, fix a partition $\Os$ of $\{1,\cdots,n\}$ subordinate to $\Oc$. To construct $\Os'$, we first fix a partition of each element of $\Oc$ into even positive integers, such that the terms in these partitions exactly constitute $\Oc'$. Let the number of choices for these partitions be $\eta_{\Oc,\Oc'}$. Once these partitions are fixed, we have that
 \begin{equation}\label{upperbd}\mathrm{Number\ of\ choices \ for\ }\Os'\mathrm{\ is\ }\leq\bigg(\frac{\prod_{2b\in\Oc}b!}{\prod_{2b\in\Oc'}b!}\bigg)^2\quad \Rightarrow\quad\xi_{\Oc,\Oc'}\leq \bigg(\frac{\prod_{2b\in\Oc}b!}{\prod_{2b\in\Oc'}b!}\bigg)^2\cdot\eta_{\Oc,\Oc'}.\end{equation} In fact, consider any subset $A\in\Os$, say $|A|=2a$, which is partitioned $a=b_1+\cdots +b_q$ as described above. To divide $A$ into subsets of cardinalities $2b_j\,(1\leq j\leq q)$ to form part of $\Os'$ (we may call this part $\Os_A$), we need to divide the set of elements with $+$ sign and the set of elements with $-$ sign separately, leading to at most $(a!)^2/((b_1)!\cdots (b_q)!)^2$ choices, considering also that there may be repetitions due to symmetry. Applying this for each $A$, we get the upper bound (\ref{upperbd}).
 
We write $\Oc=(2a_1,\cdots,2a_r)$, and define $G(a)=(C_0a)!/(a!)^2$. Using (\ref{upperbd}), we can bound the left hand side of (\ref{inductstep}) by
\[\sum_{\Oc'\prec\Oc}\eta_{\Oc,\Oc'}\frac{r!}{(|\Oc'|)!}\frac{\prod_{2b\in\Oc'}G(b)}{G(a_1)\cdots G(a_r)}.\] Using the definition of $\eta_{\Oc,\Oc'}$, we can further reduce this to
\[\sum_{(\Pc_1,\cdots,\Pc_r)}\frac{r!}{(|\Pc_1|+\cdots +|\Pc_r|)!}\prod_{i=1}^r\frac{1}{G(a_i)}\prod_{2b\in \Pc_i}G(b)\] where each $\Pc_i$ is a partition of $2a_i$ into even positive integers, and at least one $\Pc_i$ is nontrivial (i.e. contains at least two elements). Suppose $q$ of these $r$ partitions are nontrivial, then $|\Pc_1|+\cdots +|\Pc_r|\geq r+q$. Thus we get an upper bound
\begin{equation}\label{upperbd2}\sum_{q=1}^r\binom{r}{q}\frac{r!}{(r+q)!}\bigg(\sup_a\sum_{\Pc}\frac{1}{G(a)}\prod_{2b\in \Pc}G(b)\bigg)^q,\end{equation} where $\Pc$ runs over all nontrivial even partitions of $2a$. As $\binom{r}{q}\frac{r!}{(r+q)!}\leq 1/q!$, it suffices to show that
\begin{equation}\label{binom}\sum_{\Pc}\frac{1}{G(a)}\prod_{2b\in \Pc}G(b)\leq\frac{1}{4},\end{equation} which would then imply that $(\ref{upperbd2})\leq 1/2$ and thus complete the induction.

The proof of (\ref{binom}) is easy. Note that $\log G$ is convex, so if $|\Pc|=s$, then
\[\frac{1}{G(a)}\prod_{2b\in \Pc}G(b)\leq\frac{((C_0)!)^{s-1}G(a-s+1)}{G(a)}\leq\min\bigg(\frac{(C_0!)^{s-1}}{(C_0(s-1))!},\frac{(C_0!)^{s-1}}{(C_0(a-s+1))^{C_0(s-1)}}\bigg)a^{2(s-1)}.\] Using also that $(m/3)^m\leq m!\leq m^m$, we can further bound this by 
\[\frac{1}{G(a)}\prod_{2b\in \Pc}G(b)\leq \frac{C_0^{C_0(s-1)}}{(C_0\max((s-1)/3,a-s+1))^{C_0(s-1)}}\leq (a/4)^{-C_0(s-1)}a^{2(s-1)}.\]The number of choices of $\Pc$ is at most $a^{s-1}$, so the left hand side of (\ref{binom}) is bounded by
\[\sum_{s=2}^aa^{s-1}(a/4)^{-C_0(s-1)}a^{2(s-1)}\leq\frac{1}{4}\] provided $C_0$ is big enough (we may assume $a\geq 5$, since the cases $a\leq 4$ are easily verified). This proves (\ref{binom}), and finishes the inductive proof of (\ref{induct}).

Finally, suppose the sum of elements in $\Oc$ that are at least $4$ is $q$, then $|\Oc|\geq(n-q)/2$. Thus by (\ref{induct}) we have
\[|\lambda(\Oc)|\leq\frac{(n/2)!}{((n-q)/2)!}(C_0!)^{(n-q)/2}(C_0q/2)!\leq (C_0!)^{n/2}(n/2)^{q/2}(C_0q/2)^{C_0q/2}\leq C_1^nn^{C_1q},\] which completes the proof.
\end{proof}
\begin{rem} Note that $\xi_{\Oc,(2,\cdots,2)}=\prod_{2b\in\Oc}b!$. In the Gaussian case, where $\mu_b=b!$, (\ref{recur}) implies that $\lambda(\Oc)=0$ for \emph{any} $\Oc\neq (2,\cdots,2)$, hence we reduce to the Isserlis' theorem.
\end{rem}
\subsection{Over-gardens}\label{overgar} Using Lemma \ref{isserlis2} instead of Isserlis' theorem, we can replace (\ref{corrgarden}) by \begin{equation}\label{corrgarden2}\Eb\bigg(\prod_{j=1}^{2R}(\Jc_{\Tc_j})_{k_j}^{\zeta_j}(t)\bigg)=\sum_{\Os}\lambda(\Oc)\Mc_{\Oc\Gc}(t,k_1,\cdots,k_{2R}).
\end{equation} Here $\Os$ is a set of \emph{over-pairings} of the leaves of the trees $\Tc_j\,(1\leq j\leq 2R)$, which is a partition of these leaves into subsets, such that the number of leaves with sign $+$ in each subset is equal to the number of leaves with sign $-$. The total number of these leaves is $2(m+R)$ where $m$ is the sum of scales of $\Tc_j\,(1\leq j\leq 2R)$, and $\Os$ induces an even partition of $2(m+R)$ which we denote by $\Oc$, such that $\Os\models\Oc$ in the sense of Lemma \ref{isserlisnew}. The coefficient $\lambda(\Oc)$ is as in Lemma \ref{isserlisnew}, and $\Oc\Gc$ is the set of these $2R$ trees together with the set of over-pairing $\Os$, which we refer to as an \emph{over-garden}. Note that we may also write $\Oc\Gc\models\Oc$ instead of $\Os\models\Oc$. Finally, like (\ref{defkg}) we have \begin{equation}\label{defkg2}\Mc_{\Oc\Gc}(t,k_1,\cdots,k_{2R})=\bigg(\frac{\delta}{2L^{d-1}}\bigg)^m\zeta^*(\Oc\Gc)\sum_\Is\epsilon_\Is\int_{\Ic}\prod_{\nf\in\Nc^*}e^{\pi i\zeta_\nf\cdot\delta L^2\Omega_\nf t_\nf}\,\mathrm{d}t_\nf\cdot\prod_{\lf\in\Lc^*}^{(+)}n_{\mathrm{in}}(k_\lf),
\end{equation} where all the objects are defined as in (\ref{defkg}), except that for the decoration $\Is$ we require that $k_\nf\,(\nf\in A)$ are all equal for each $A\in\Os$. 

Clearly an over-garden $\Oc\Gc$ can be turned into a garden $\Gc$ by dividing each over-pairing $A\in\Os$ into leaf pairs; below we will write $\Oc\Gc\sim\Gc$ for this. In this case $\Mc_{\Oc\Gc}$ is the same as $\Mc_\Gc$, except for the finitely many additional conditions of form $k_\lf=k_{\lf'}$ associated with the over-pairings structure of $\Oc\Gc$, which are added to the decoration $\Is$ in the summation (\ref{defkg2}). Now by (\ref{corrgarden2}) we have
\begin{equation}\label{corrgarden3}
\Eb\bigg(\prod_{j=1}^{2R}(\Jc_{m_j})_{k_j}^{\zeta_j}(t)\bigg)=\sum_{\Oc}\lambda(\Oc)\sum_{\Oc\Gc\models\Oc}\Mc_{\Oc\Gc}(t,k_1,\cdots,k_{2R}),
\end{equation} where $\Oc$ is an even partition of $2(m+R)$ with $m=m_1+\cdots+m_{2R}$, and $\Oc\Gc=(\Tc_1,\cdots,\Tc_{2R})\models\Oc$ is the over-garden of width $2R$ and signature $(\zeta_1,\cdots,\zeta_{2R})$ such that the scale of $\Tc_j$ is $m_j$ for $1\leq j\leq 2R$. Note that for fixed $\Gc$, the number of $\Oc\Gc$ such that $\Oc\Gc\models\Oc$ and $\Oc\Gc\sim\Gc$ depends only on $\Oc$; similarly, for fixed $\Oc\Gc\models\Oc$, the number of $\Gc$ such that $\Oc\Gc\sim\Gc$ also depends only on $\Oc$. We denote them by $\sigma_1(\Oc)$ and $\sigma_2(\Oc)$ respectively. Thus (\ref{corrgarden3}) can be rewritten as
\begin{equation}
\label{corrgarden4}\Eb\bigg(\prod_{j=1}^{2R}(\Jc_{m_j})_{k_j}^{\zeta_j}(t)\bigg)=\sum_{\Oc}\frac{\lambda(\Oc)}{\sigma_2(\Oc)}\sum_{\Gc}\sum_{\Oc\Gc\sim\Gc,\Oc\Gc\models\Oc}\Mc_{\Oc\Gc}(t,k_1,\cdots,k_{2R}),
\end{equation} where $\Gc$ satisfies the same condition as $\Oc\Gc$ but is a runs over gardens instead of over-gardens.

Now the study of (\ref{corrnew}) reduces to the study of the quantities $\Mc_{\Oc\Gc}$ for over-gardens $\Oc\Gc$. To this end we introduce the notion of regular over-gardens, and one simple linear algebra lemma.
\begin{df}\label{defregover} Define an over-garden $\Oc\Gc$ to be a \emph{regular} over-garden, if there exists $\Gc$ such that $\Oc\Gc\sim\Gc$, and (i) $\Gc$ is a regular multi-couple (Definition \ref{defgarden}), and (ii) for each leaf $\lf$ in each over-pairing $A\in\Os$ with $|A|\geq 4$, the tree $\Tc_j$ containing $\lf$ \emph{must} be a regular tree and $\lf$ must be its lone leaf.
\end{df}
\begin{lem}\label{linalg} Let $A\subset B$ be two sets of affine linear equations posed on $\Rb^n$, in terms of the coordinates $(x_1,\cdots,x_n)$. Let $\Ac\supset\Bc$ be the affine submanifold of $\Rb^n$ determined by equations in $A$ and $B$ respectively, assume $\Bc\neq\varnothing$, and denote $p=\mathrm{codim}_\Ac(\Bc)$. Let $1\leq r\leq n-1$.

For any fixed $x=(x_1,\cdots, x_r)\in\Rb^r$, consider the affine submanifold $\Ac_x\supset\Bc_x$ of $(x_{r+1},\cdots,x_n)\in\Rb^{n-r}$ determined by equations in $A$ and $B$ for $(x,x_{r+1},\cdots,x_n)$ respectively. Let $\Ac^\circ\supset\Bc^\circ$ be the affine manifold for $x$ such that $\Ac_x\neq\varnothing$ and $\Bc_x\neq\varnothing$ respectively. Then, if $\mathrm{codim}_{\Ac^\circ}(\Bc^\circ)=p^\circ$, then for any $x\in\Bc^\circ$ we have $\mathrm{codim}_{\Ac_x}(\Bc_x)=p-p^\circ$.
\end{lem}
\begin{proof} We omit the proof as it is elementary.
\end{proof}
Now we can state the main estimate for the non-Gaussian case.
\begin{prop}\label{nongaussprop} Fix $R$ and $(\zeta_1,\cdots,\zeta_{2R})$ and $(k_1,\cdots,k_{2R})$ and $(m_1,\cdots,m_{2R})$ as in Proposition \ref{mainprop}. Assume $R\leq 2 \log L/(\log\log L)^2$ and $m_j\leq N\,(1\leq j\leq 2R)$, and set $m:=m_1+\cdots +m_{2R}$. Consider the sum $\Is'$ as in (\ref{corrgarden4}), but we restrict to \emph{non-regular over-gardens} $\Oc\Gc$ in the summation. Then we have
\begin{equation}\label{nonregoverest}|\Is'|\leq (C^+\delta^{1/4})^mL^{-\nu}
\end{equation} uniformly in $t$ and $(k_1,\cdots,k_{2R})$.
\end{prop}
\begin{proof} Let $\Oc\Gc$ be an over-garden and $\Oc\Gc\sim\Gc$. Since the expression $\Mc_{\Oc\Gc}$ in (\ref{defkg2}) is just the expression $\Mc_{\Gc}$ in (\ref{defkg}) with finitely many extra requirements of form $k_\lf=k_{\lf'}$ in the summation, it is clear that $\Mc_{\Oc\Gc}$ can at least be estimated in the same way as $\Mc_{\Gc}$ with no power loss. Also the number of $\Oc$ is at most $C^m$, and for fixed $\Gc$ and $\Oc$, the third sum in (\ref{corrgarden4}) contains $\sigma_1(\Oc)$ terms, where $\sigma_1(\Oc)$ has the same upper bound as $\lambda(\Oc)$ in Lemma \ref{isserlis2}.

Therefore, compared to the bounds for $\Mc_{\Gc}$ that we already know, we only have two tasks in proving the desired result: first, gain the extra power $L^{-\nu}$, and second, gain enough extra powers to cancel the factor $|\lambda(\Oc)|$ in case the latter is too large. Note also that if $\Oc$ is given and $\Gc$ and $\Gc'$ are congruent in the sense of Definition \ref{conggen}, then the over-gardens $\Oc\Gc\models\Oc,\,\Oc\Gc\sim\Gc$ are in one-to-one correspondence with the over-gardens $\Oc\Gc'\models\Oc,\,\Oc\Gc'\sim\Gc'$, and cancellations between the terms $\Mc_{\Oc\Gc}$ and $\Mc_{\Oc\Gc'}$ are the same as the cancellations between $\Mc_\Gc$ and $\Mc_{\Gc'}$ in Section \ref{irre}, up to minor modifications. As such, we can exploit the same cancellation for irregular chains in $\Gc$ as in Section \ref{irre} and \cite{DH}.

Now let us go over the process of studying $\Mc_\Gc$, and see what the extra conditions $k_\lf=k_{\lf'}$ may do at each step of this process. Below let $q_0$ be the number of independent extra equations $k_\lf=k_{\lf'}$, then $q_0\sim q$, where $q$ is the sum of elements in $\Oc$ that are at least $4$ as in Lemma \ref{isserlisnew}. In fact, we have
\[q=\sum_{4\leq a\in\Oc}a,\quad q_0=\sum_{4\leq a\in\Oc}\big(\frac{a}{2}-1\big),\] hence $2q_0\leq q\leq 4q_0$. We will keep track of the codimension introduced by these $q_0$ extra equations to the affine manifold of all possible decorations $(k_\lf)$ using Lemma \ref{linalg}.

(1) Assume that $\Gc$ is a regular multi-couple. In this case, we shall estimate the summation (together with the integration) in (\ref{defkg2}) using the method in Section 6 of \cite{DH} (note that here we have to treat all the regular couples in $\Gc$ together---instead of one at a time in Section 6 of \cite{DH}---because of the extra conditions linking different regular couples together, but this will cause minor changes to the proof). In particular, we define the variables $x_\nf$ and $y_\nf$ as in the proof of Proposition 6.7 of \cite{DH}. Note that there are $2R$ linear equations that any decoration of the leaves of $\Gc$ must satisfy (and such decoration of leaves uniquely determines the full decoration of $\Gc$); moreover the set of decorations satisfying these $2R$ equations is in affine bijection with the set of free variables $(x_\nf,y_\nf)$, see the proof of Proposition 6.7 of \cite{DH}.

Now, with the extra conditions, the dimension of the affine manifold of all possible decorations $(k_\lf)$ gets strictly lower, and the codimension $r$ introduced satisfies $r\gtrsim\max(1,q_0-O(R))$. In fact we have $r>0$ because at least one extra condition must take the form $k_\lf=k_{\lf'}$ where $\lf$ is not the lone leaf of a regular tree by Definition \ref{defregover}, and this equation will be independent of the $2R$ original equations stated above (since the only way for this extra condition to be dependent is for the two trees containing $\lf$ and $\lf'$ to be distinct and coupled, which would easily imply that they are two regular trees with lone leaves $\lf$ and $\lf'$). The lower bound $q_0-O(R)$ is because the number of independent extra equations is $q_0$, and we subtract $O(R)$ because some extra equations combined may imply some of the $2R$ original equations.

Using the affine bijection, we know that the $(x_\nf,y_\nf)$ variables must satisfy $r$ independent linear equations. Then, we proceed as in the proof of Proposition 6.1 in \cite{DH}, and sum over the $(x_\nf,y_\nf)$ variables one by one. At each step, suppose we are summing over the pair $(x_j,y_j)$, depending on the extra equations satisfied by these variables, we have one of three possibilities: (a) there is no restriction on $(x_j,y_j)$ and we are summing over all choices of $(x_j,y_j)$; (b) we are summing over $(x_j,y_j)$ that satisfies one linear equation (such as $x_j=\mathrm{const}$ or $y_j=\mathrm{const}$ or $x_j\pm y_j=\mathrm{const}$); (c) we are summing over $(x_j,y_j)$ that satisfies two linear equations, i.e. over only one point $(x_j^*,y_j^*)$. In either case the summation can be performed in the same way as in \cite{DH}, and in cases (b) and (c) we are gaining a power of $L$ in this summation, compared to the factor $L^{2d-2}$ in \cite{DH}. Also, by repeating Lemma \ref{linalg}, we know that case (b) or (c) must occur at least $\gtrsim r$ times during this process.

Therefore, putting altogether, with these extra conditions we can gain power $L^{-cr}$ for some small constant $c$ with $r\gtrsim\max(1,q-O(R))$, in the estimate of $\Mc_{\Oc\Gc}$ compared to $\Mc_{\Gc}$. This already provides the needed $L^{-\nu}$ gain. It also covers any possible loss due to $\lambda(\Oc)$ because $|\lambda(\Oc)|\leq C^mm^{Cq}\leq C^m(\log L)^{Cq_0}$ by Lemma \ref{isserlisnew}, hence
\[|\lambda(\Oc)|\leq C^m(\log L)^{O(R)}\cdot(\log L)^{C\max(1,q_0-O(R))}\leq C^mL^{\nu/2}\cdot (\log L)^{O(r)}\] by our choice $R\leq \log L/(\log\log L)^2$, which can then be covered by the $L^{-\nu}$ gain; also the various loss of $C^m$ is unimportant as they can be absorbed into $(C^+)^m$ in (\ref{nonregoverest}).

(2) Now we assume $\Gc$ is \emph{not} a regular multi-couple. Then the sum of $\Mc_\Gc$ already gains the power $L^{-\nu}$ in view of Proposition \ref{mainprop} and also Proposition 10.4 in \cite{DH}. It then suffices to cover the possible loss due to $\lambda(\Oc)$. We perform the reduction steps as in previous sections, and analyze the extra conditions appearing in each step. As in (1), the total codimension introduced by the $q_0$ extra equations is $r\gtrsim\max(1,q_0-O(R))$; we may assume $q_0\gg R$ because otherwise the loss $|\lambda(\Oc)|\leq C^mm^{Cq_0}\leq C^mL^{\nu/2}$ can already be covered by the guaranteed $L^{-\nu}$ gain. Therefore we now have $r\gtrsim q_0$.

\emph{Step 1}. We first remove the regular couples and regular trees to reduce $\Gc$ to its skeleton $\Gc_{sk}$ as in Proposition \ref{defskeleton}. In this process we are fixing all the remaining $k_\nf$ variables (which are determined  by the variables $k_{\lf_1}$ for leaves $\lf_1$ of $\Gc_{sk}$) and sum over the variables $k_{\lf_2}$, where $\lf_2$ are leaves of these regular couple and regular trees, similar to (1) above. By Lemma \ref{linalg}, there exists $r_1+r_2=r$, such that for any fixed $(k_{\lf_1})$, the codimension of the submanifold formed by the $(k_{\lf_2})$ variables is $r_2$, and the codimension of the submanifold formed by the $(k_{\lf_1})$ variables is $r_1$. By repeating the argument in (1) above, we can gain a power $L^{-cr_2}$ in summing over the $(k_{\lf_2})$ variables. Note that some extra equations satisfied by the $(k_{\lf_2})$ variables may be of form $k_{\lf_2}=\mathrm{const}$ instead of $k_{\lf_2}=k_{\lf_2'}$ as in (1), but this does not affect the proof.

\emph{Step 2}. We further remove the irregular chains from the skeleton $\Gc_{sk}$ and exploit the cancellation as in Section \ref{irre}. Note that if $\Gc$ is not a regular multi-couple, then \emph{any} $\Oc\Gc\sim\Gc$ must be non-regular in the sense of Definition \ref{defregover}, so for fixed $\Gc$, the summation in $\Oc\Gc$ we are studying here is still the same as the one in (\ref{corrgarden4}) even though we have made the restriction that $\Oc\Gc$ is non-regular. Thus, as said above, the cancellation for irregular chains works the same way in the current situation as in Section \ref{irre}. The extra conditions again lead to gain of powers in $L$. Like in \emph{Step 1}, we can write $r_1=r_3+r_4$, such that we can gain a power $L^{-cr_3}$ in the current step, and after removing the irregular chains, the remaining decoration (of the remaining garden $\Gc_{sk}^\#$, see Section \ref{summarysec}) still satisfies $r_4$ extra linear equations.

\emph{Step 3}. Finally we reduce the estimate of the remaining expression to the counting problem associated with the molecule formed from $\Gc_{sk}^\#$.  Here we only need to show that, if, in addition to the equations in the counting problem, the variables in question also satisfy $r_4$ additional independent linear equations, then we can improve the upper bound for the counting problem by a power $L^{-cr_4}$ with a small constant $c$.

To see this, we follow the procedure described in Section \ref{improvecount}, and in particular apply the algorithm introduced in Section 9.4 of \cite{DH}. In this process, where we fix some of the variables in each step, we keep track of the codimension, or the number $p$ of independent equations satisfied by the remaining variables. Initially we have $p\gtrsim r_4$, while in the end we have $p=0$. Therefore, there must be at least $\gtrsim r_4$ steps where $p$ strictly decreases. If this step is a good step in the sense of Section 9.3 of \cite{DH}, then we are gaining a constant power $L^{-\nu}$ here; even if it is a normal step, since $\Delta p<0$, by Lemma \ref{linalg}, in doing the counting estimate for this step, we can take into account an additional independent linear equation satisfied by the variables in consideration. For example, if we perform the step (3R-1) defined in Section 9.3.8 of \cite{DH}, then the corresponding counting problem we solve is (say)
\[\left\{
\begin{aligned}&a-b+c=\mathrm{const},\\
&|a|_\beta^2-|b|_\beta^2+|c|_\beta^2=\mathrm{const}+O(L^{-2}),\\
&a,b,c\in\Zb_L^d,\,|a|,|b|,|c|\lesssim 1,
\end{aligned}
\right.\] which has $O(L^{2d-2})$ solutions. However, if we add to this system another independent linear equation $\alpha a+\beta b+\gamma c=\mathrm{const}$, where $(\alpha,\beta,\gamma)$ is not a multiple of $(1,-1,1)$, then the number of solutions will be at most $L^d$ with $d<2d-2$. This leads to a power gain in each such step, so in total we can gain a power $L^{-cr_4}$ for some constant $c$.

After the above three steps, the total power we gain would be $L^{-c(r_2+r_3+r_4)}=L^{-cr}$, which is enough to cover the loss $C^mm^{Cq}$ from $\lambda(\Oc)$ because $r\gtrsim q_0\gtrsim q$. Therefore in any case we can cover the possible loss with an extra gain of $L^{-\nu}$, hence (\ref{nonregoverest}) holds. This completes the proof.
\end{proof}
With Proposition \ref{nongaussprop} we can now prove Theorem \ref{main2}.
\begin{proof}[Proof of Theorem \ref{main2}] We use (\ref{expand}) to expand $\Eb\big(\prod_{j=1}^{2R}a_{k_j}^{\zeta_j}(t)\big)$. The estimate for the remainder term $b$ can be done using arguments similar to Section \ref{remterm}, which we shall omit. Then, using also (\ref{corrgarden4}), we can write
\begin{equation}
\label{corrgarden5}
\Eb\bigg(\prod_{j=1}^{2R}a_{k_j}^{\zeta_j}(t)\bigg)=\sum_{\Gc,\Oc}\frac{\lambda(\Oc)}{\sigma_2(\Oc)}\sum_{\Oc\Gc\sim\Gc,\Oc\Gc\models\Oc}\Mc_{\Oc\Gc}(t,k_1,\cdots,k_{2R})+O(L^{-\nu}),
\end{equation} where $\Gc$ runs over all gardens of width $2R$ such that the scale of each tree is at most $N$, and $\Oc$ runs over all even partitions of $2(m+R)$ where $m$ is the scale of $\Gc$. By Proposition \ref{nongaussprop} and summing over all possible $m_j$ as in the proof of Theorem \ref{main} above, we see that with $R$ fixed and $L\to\infty$, the contribution of \emph{non-regular} over-gardens $\Oc\Gc$ decays like $L^{-\nu}$ in the limit. Thus, we only need to consider \emph{regular} over-gardens $\Oc\Gc$. Suppose $\Oc\Gc\sim\Gc$, then $\Gc$ is a regular multi-couple. Therefore, unless we can divide $\{1,\cdots,2R\}$ into pairs such that for each pair $\{i,j\}$ we have $k_i=k_j$ and $\zeta_i=-\zeta_j$, the contribution of regular over-gardens must vanish, in particular (\ref{approximation1}) is true. Now we only need to prove (\ref{approximation2new}).

For any regular couple $\Qc=(\Tc_1,\Tc_2)$, the tree $\Tc_1$ is a regular tree if and only if $\Tc_2$ is a regular tree (and hence the two lone leafs are paired). In this case we say $\Qc$ is \emph{tangential} (since the two trees only have one leaf-pair in common), otherwise say $\Qc$ is \emph{non-tangential}. Note that by the proof of Theorem \ref{main} in \cite{DH} we have
\begin{equation}\label{allasymp}\sum_{\Qc}\Mc_\Qc(t,t,k)=n(\delta t,k)+O(L^{-\nu}),\end{equation} where $\Qc$ runs over all regular couples with both trees having scale at most $N$. If we restrict to \emph{tangential} couples only, then the sum should be approximated by $n_0(\delta t,k)$ where $n_0$ is defined in (\ref{wke0}). The easiest way to see this is that, the expression $\Mc_\Qc(t,t,k)$ contains a factor $n_{\mathrm{in}}(k)$ if and only if $\Qc$ is tangential, because for any regular tree $\Tc$ with root $\rf$ and lone leaf $\lf$ we must have $k_\lf=k_\rf=k$. Thus, since the sum (\ref{allasymp}) over all couples exactly matches the Taylor expansion of $n(\delta t,k)$ (as shown in \cite{DH}), we know that the sum over tangential couples will exactly match the terms in the Taylor expansion \emph{that contain the factor $n_{\mathrm{in}}(k)$}. Due to the form of (\ref{wkenon}), it is easy to see that the sum of these terms is exactly $n_0(t,k)$, hence the result. Therefore we have
\begin{equation}\label{partasymp1}\sum_{\Qc\mathrm{\ \!tangential}}\Mc_\Qc(t,t,k)=n_0(\delta t,k)+O(L^{-\nu}),\end{equation}
\begin{equation}\label{partasymp2}\sum_{\Qc\textrm{\ \!non-tangential}}\Mc_\Qc(t,t,k)=n_+(\delta t,k)+O(L^{-\nu}).\end{equation}

We now return to the sum (\ref{corrgarden5}) over regular over-gardens $\Oc\Gc$. For (\ref{approximation2new}), we may rename $(k_1,\cdots,k_{2R})$ such that there are $2a_j$ copies of $k_i^*$ for $1\leq i\leq r$ (with half of them having sign $+$ and half having sign $-$) where the $k_i^*$ are all different and $a_1+\cdots+a_r=R$. For simplicity we will write $k_i$ instead of $k_i^*$ below. Clearly the $2a_i$ trees corresponding to the input variable $k_i$ must form $a_i$ couples in $\Gc$; assume $b_i$ of these $a_i$ couples are tangential and the rest are non-tangential, where $0\leq b_i\leq a_i$. Note also that for any $\Oc\Gc\sim\Gc$ we have $\Mc_{\Oc\Gc}=\Mc_\Gc$ because over-pairings can only happen at lone leaves of regular trees. Therefore, for fixed $(b_1,\cdots,b_r)$, we can calculate the contribution of regular over-gardens as
\begin{multline}\label{corrgarden6}\Is_{b_1,\cdots,b_r}=\prod_{i=1}^r\binom{a_i}{b_i}^2(a_i-b_i)!\bigg(\sum_{\Qc:\mathrm{non-tangential}}\Mc_\Qc(t,t,k_i)\bigg)^{a_i-b_i}\\\times\sum_{\Gc',\Oc'}\frac{\lambda(\Oc')\sigma_3(\Oc')}{\sigma_2(\Oc')}\Mc_{\Gc'}(t,k_1,\cdots,k_1,\cdots,k_r,\cdots,k_r).
\end{multline} Here $\Gc'$ runs over all gardens \emph{of width $2(b_1+\cdots+b_r)$} which are multi-couples formed by \emph{tangential couples} (and otherwise same as $\Gc$), $\Oc'$ runs over all \emph{even partitions of $2(b_1+\cdots+b_r)$} such that $\Oc'\preceq(2b_1,\cdots,2b_r)$, and $\sigma_3(\Oc')$ is the number of regular over-gardens $\Oc\Gc$ such that $\Oc\Gc\sim\Gc$ and $\Oc\Gc\models\Oc$ where $\Oc=(\Oc',2,\cdots,2)$. We can further reduce the inner sum in (\ref{corrgarden6}) to
\begin{multline}\label{corrgarden6.5}\sum_{\Gc',\Oc'}\frac{\lambda(\Oc')\sigma_3(\Oc')}{\sigma_2(\Oc')}\Mc_{\Gc'}(t,k_1,\cdots,k_1,\cdots,k_r,\cdots,k_r)=\prod_{i=1}^r\bigg(\sum_{\Qc\mathrm{\ \!tangential}}\Mc_\Qc(t,t,k_i)\bigg)^{b_i}\\\times\sum_{\Oc'}\frac{\lambda(\Oc')\sigma_3(\Oc')}{\sigma_2(\Oc')}(b_1)!\cdots(b_r)!.
\end{multline} Finally, in the above summation $\Oc'$ must be $\preceq (2b_1,\cdots,2b_r)$ otherwise $\sigma_3(\Oc')=0$, and by definitions and a double counting argument we can show that under this assumption we have
\[\xi_{(2b_1,\cdots,2b_r),\Oc'}=\sigma_3(\Oc')\frac{(b_1)!\cdots (b_r)!}{\sigma_2(\Oc')}.\] In fact, fix $\Os\models(2b_1,\cdots,2b_r)$. To construct $\Os'$ such that $\Os'\preceq\Os$ and $\Os'\models\Oc'$, we first divide each subset in $\Os$ into pairs (which has $(b_1)!\cdots (b_r)!$ choices), thus obtaining a garden $\Gc$, then construct $\Oc\Gc\sim\Gc$ and $\Oc\Gc\models\Oc$ (which has $\sigma_3(\Oc')$ choices) and form $\Os'$ accordingly. Note that each $\Os'$ is counted exactly $\sigma_2(\Oc')$ times (which is just the number of choices of $\Gc$ with fixed $\Oc\Gc$), hence the result.

Now we can reduce the last sum in (\ref{corrgarden6}) to
\[\sum_{\Oc'}\frac{\lambda(\Oc')\sigma_3(\Oc')}{\sigma_2(\Oc')}(b_1)!\cdots(b_r)!=\sum_{\Oc'\preceq (2b_1,\cdots,2b_r)}\lambda(\Oc')\xi_{(2b_1,\cdots,2b_r),\Oc'}=\prod_{i=1}^r\mu_{b_i}\] using (\ref{recur}). Putting altogether, using (\ref{partasymp1}) and (\ref{partasymp2}), and then summing over $b_i\,(1\leq i\leq r)$, we get
\[\Eb\bigg(\prod_{j=1}^{2R}a_{k_j}^{\zeta_j}(t)\bigg)=\sum_{0\leq b_i\leq a_i\,(1\leq i\leq r)}\prod_{i=1}^r\binom{a_i}{b_i}^2(a_i-b_i)!\mu_{b_i}(n_0(\delta t,k))^{b_i}(n_+(\delta t,k))^{a_i-b_i}+O(L^{-\nu})\] which is just (\ref{approximation2new}) given (\ref{defmur}). This completes the proof.
\end{proof}
\subsection{Evolution of density}\label{evodens} Finally we prove Theorem \ref{main3}. Note that if $\mu_r\leq C^r(2r)!$, then by (\ref{defmur}) for any $t$ we also have $\mu_r(t,k)\leq C^r(2r)!$ perhaps for some different $C$. Thus, convergence in law will be a consequence of the following lemma:
\begin{lem}\label{convlaw} Suppose $\{X_n\}$ are $\Rb^d$ valued random variables, such that for any multi-index $\mu$ the limit
\[A_\mu:=\lim_{n\to\infty}\Eb(X_n^\mu)\] exists and $|A_\mu|\leq C^{|\mu|}(|\mu|)!$. Then $\{X_n\}$ converges in law to a random variable $X$ satisfying $\Eb(X^\mu)=A_\mu$ for any multi-index $\mu$. Moreover, the law with these given moments is unique.
\end{lem}
\begin{proof} First the assumption implies that $\Eb|X_n|^2$ is uniformly bounded in $n$, thus the sequence of laws of $X_n$ is tight. For any subsequence $X_{n_k}$ we then have a subsequence $X_{n_{k_\ell}}$ that converges in law to (say) some random variable $X$. For any $\mu$, since $\Eb|X_n|^{2|\mu|}$ are bounded in $n$ it is easy to see that $\Eb(X^\mu)=\lim \Eb(X_n^\mu)=A_\mu$. Therefore, it suffices to prove that the law of the random variables $X$ with $\Eb(X^\mu)=A_\mu$ is unique. This is true because, since $|A_\mu|\leq C^{|\mu|}(|\mu|)!$, we then have $\Eb(e^{\varepsilon|X|})<\infty$ for small enough $\varepsilon$, hence $f(\xi)=\Eb(e^{i\xi\cdot X})$ is well-defined and holomorphic in the region $|\mathrm{Im}\xi|<\varepsilon$. The moments $\Eb(X^\mu)$ uniquely determines the Taylor expansion of $f(\xi)$ at $\xi=0$, hence uniquely determines the value of $f(\xi)$ in a neighborhood of $0$---and consequently in the whole region $|\mathrm{Im}\xi|<\varepsilon$ by analyticity. In particular the moments uniquely determine the value of $f(\xi)$ on $\Rb$, which is the characteristic function of $X$. Thus the law of $X$ is unique, as desired.
\end{proof}
Now, for any $t\in[0,\delta]$ and $k\in\Zb_L^d$, consider the unique radial density $\rho=\rho_k(t,v)$ (where $v\in\Cb$ also viewed as an $\Rb^2$ vector) such that 
\begin{equation}\label{dens}\int_\Cb(v_1^2+v_2^2)^r\rho_k(t,v)\,\mathrm{d}v_1\mathrm{d}v_2=\mu_r(t,k)=\sum_{p=0}^r\binom{r}{p}^2(r-p)!\mu_p\cdot(n_0(t,k))^p(n_+(t,k))^{r-p},\end{equation} then by Theorem \ref{main2} and Lemma \ref{convlaw} we have that convergence in law described by (\ref{evolpdf4}) and (\ref{evolpdf5}) is true. Thus to prove Theorem \ref{main3} it suffices to show that $\rho_k(t,v)$ solves the equation (\ref{evolpdf3}). The initial data $\rho_k(0)$ is verified by definition since the right hand side of (\ref{dens}) equals $\mu_r n_{\mathrm{in}}(k)^r$ when $t=0$. Now, by approximating from below, we may assume $\{\mu_p\}$ is bounded, and consider
\[L_k(t,\xi)=\int_\Cb e^{i\xi(v_1^2+v_2^2)}\rho_k(t,v)\,\mathrm{d}v_1\mathrm{d}v_2=\sum_{0\leq p\leq r}\frac{(i\xi)^r}{p!}\binom{r}{p}\mu_p\cdot(n_0(t,k))^p(n_+(t,k))^{r-p}.\] Taking time derivative and calculating using (\ref{wke}), (\ref{wkenon}), (\ref{wke0}) and (\ref{wkenon0}) yields
\begin{multline}\partial_tL_k=\sum_{0\leq p\leq r}\frac{(i\xi)^r}{p!}\binom{r}{p}\mu_p\cdot\big\{p(n_0(t,k))^{p}(n_+(t,k))^{r-p}\gamma_k(t)\\+(r-p)(n_0(t,k))^p(n_+(t,k))^{r-p-1}(\sigma_k(t)+\gamma_k(t)n_+(t,k))\big\}\end{multline} where $\sigma_k(t)$ and $\gamma_k(t)$ are defined in (\ref{evolpdf1}) and (\ref{evolpdf2}). This simplifies to
\begin{multline}\partial_tL_k=\gamma_k(t)\cdot\sum_{0\leq p\leq r}r\cdot\frac{(i\xi)^r}{p!}\binom{r}{p}\mu_p\cdot(n_0(t,k))^p(n_+(t,k))^{r-p}\\+\sigma_k(t)\cdot\sum_{0\leq p\leq r}(r-p)\cdot\frac{(i\xi)^r}{p!}\binom{r}{p}\mu_p\cdot(n_0(t,k))^p(n_+(t,k))^{r-p-1}.\end{multline} The first sum on the right hand side equals $\xi\partial_\xi L_k$, and the second sum equals
\[\sum_{0\leq p\leq r}(r+1)\cdot\frac{(i\xi)^{r+1}}{p!}\binom{r}{p}\mu_p\cdot(n_0(t,k))^p(n_+(t,k))^{r-p}\] by changing $r$ to $r+1$, which then equals $(i\xi+i\xi^2\partial_\xi)L_k$. Therefore we have
\[\partial_tL_k=\gamma_k(t)\cdot \xi\partial_\xi L_k+\sigma_k(t)\cdot(i\xi+i\xi^2\partial_\xi)L_k.\] Finally, by taking the inverse Fourier transform and switching between Cartesian and polar coordinates, we can verify that the density $\rho_k(t,v)$ solves (\ref{evolpdf3}). This completes the proof of Theorem \ref{main3}.
\section{Proof of Theorem \ref{main4}}\label{sec:wkh} In this last section we prove Theorem \ref{main4}. The proof is of ``soft" nature and has a different flavor from the other parts of this paper. The main idea, already hinted at in Section \ref{sec:wkh0}, is to represent $n_r$ as an average of tensor products for which one can apply Theorem \ref{main}. This is demonstrated in the following lemma, which is a variant of the classical Hewitt-Savage theorem.
\begin{lem}[a variant of Hewitt-Savage]
\label{hewsav} Suppose $(n_r)_{\mathrm{in}}$ is as in the statement of Theorem \ref{main4}, satisfies (\ref{bounded}), and is admissible in the sense of (\ref{admissible}). Recall the set $\Ac$ defined in (\ref{measurezeta}). Then, there exists a unique probability measure $\zeta$ on $\Ac$ such that for any $r$ and any distinct $(k_1,\cdots,k_r)$ we have
\begin{equation}\label{hewsav:eq1}(n_r)_{\mathrm{in}}(k_1,\cdots,k_r)=\int_{\Ac}m(k_1)\cdots m(k_r)\,\mathrm{d}\zeta(m).
\end{equation}
\end{lem}
\begin{proof} For simplicity we will write $(n_r)_{\mathrm{in}}$ just as $n_r$. We may assume $\Xf=1$ (otherwise simply replace $n_r$ by $\Xf^{-r}n_r$ and $C_1$ by $\Xf^{-1}C_1$). Then, each $n_r$ is the density of the joint distribution of some $r$ random variables valued in $\Rb^d$, and $n_{r-1}$ is a marginal of $n_r$. By Kolmogorov extension theorem, we can find an infinite sequence of $\Rb^d$-valued random variables $(X_1,X_2,\cdots)$ such that $n_r$ is the density of the joint distribution of $(X_1,\cdots,X_r)$. By symmetry of $n_r$, these random variables $X_j$ are exchangeable, so by Hewitt-Savage theorem \cite{HS}, there exists a unique probability measure $\zeta$ on the set $\Mc$ of \emph{all probability measures} $m$ on $\Rb^d$ such that (\ref{hewsav:eq1}) is true with $\Ac$ replaced by $\Mc$.

Now it remains to prove that $\zeta$ is supported on $\Ac$. We will apply the beautiful argument of Rosenzweig-Staffilani in Section 5 of \cite{RosStaff}, which goes back to the work of Spohn \cite{Spohn}. Fix a Schwarz function $\varphi$ on $\Rb^d$ and indices $\alpha,\beta$ with $|\alpha|,|\beta|\leq 40d$; for any \emph{even} integer $r$, we may test both sides of (\ref{hewsav:eq1}) by the tensor product $(\partial_k^\beta k^\alpha \varphi)^{\otimes r}$ to get
\[\int_\Mc(m,\partial_k^\beta k^\alpha\varphi)^r\,\mathrm{d}\zeta(m)=\int_{(\Rb^d)^r}n_r(k_1,\cdots,k_r)\prod_{j=1}^rk_j^\alpha\partial_{k_j}^\beta\varphi(k_j)\,\mathrm{d}k_1\cdots\mathrm{d}k_r,\] where $(m,\partial_k^\beta k^\alpha\varphi)$ is the integral of $\partial_k^\beta k^\alpha\varphi$ against the measure $m$, in the usual distributional sense. Recall that $\|n_r\|_{\Sf^{40d;r}}\leq C_1^r$ from (\ref{bounded}), so the right hand side is bounded by
\[\int_{(\Rb^d)^r}n_r(k_1,\cdots,k_r)\prod_{j=1}^r\partial_{k_j}^\beta k_j^\alpha\varphi(k_j)\,\mathrm{d}k_1\cdots\mathrm{d}k_r\leq \|n_r\|_{\Sf^{40d;r}}\cdot\|\varphi\|_{L^2}^r\leq (C_1\|\varphi\|_{L^2})^r.\] For any $C_2>C_1$, by Chebyshev, this implies that
\[\zeta\big(\big\{m:|(m,\partial_k^\beta k^\alpha\varphi)|\geq C_2\|\varphi\|_{L^2}\big\}\big)\leq (C_2\|\varphi\|_{L^2})^{-r}\int_\Mc(m,\partial_k^\beta k^\alpha\varphi)^r\,\mathrm{d}\zeta(m)\leq (C_1/C_2)^{r}.\] Since $r$ is arbitrary, we conclude that
\begin{equation}\label{bdd}|(m,\partial_k^\beta k^\alpha\varphi)|\leq C_2\|\varphi\|_{L^2}\end{equation} holds $\zeta$-almost surely for any fixed $\varphi$, fixed $(\alpha,\beta)$, and fixed $C_2>C_1$. By choosing $\varphi$ in a countable dense subset of Schwartz space, enumerating finitely many possible $(\alpha,\beta)$, and choosing a sequence $C_2=(1+\varepsilon)C_1$, we know that $\zeta$-almost surely in $m$, (\ref{bdd}) actually holds for \emph{all} $(\varphi,\alpha,\beta)$, and with $C_2$ replaced by $C_1$. By Riesz representation theorem, this implies that \[\|k^\alpha\partial_k^\beta m\|_{L^2}\leq C_1,\quad \forall |\alpha|,|\beta|\leq 40d\] holds $\zeta$-almost surely. By definition this means that $m\in\Ac$ because $m$ is a probability measure and $\Xf=1$. Therefore we get $\zeta(\Mc\backslash \Ac)=0$, and hence we can define the measure $\zeta$ on $\Ac$ such that (\ref{hewsav:eq1}) is true.
\end{proof}

We are now ready to prove Theorem \ref{main4}.
\begin{proof}[Proof of Theorem \ref{main4}] First, if $(n_r)_{\mathrm{in}}$ is hybrid then it must be admissible. In fact, by the definition of hybrid initial data we have
\begin{equation}\label{hybridint}\int_{\Ac}m(k_1)\cdots m(k_r)\,\mathrm{d}\zeta(m)=(n_r)_{\mathrm{in}}(k_1,\cdots,k_r)\end{equation} for any distinct $k_j\in\Zb_L^d\,(1\leq j\leq r)$. Since $m$ and $(n_r)_{\mathrm{in}}$ are all continuous functions, by taking suitable limits and letting $L\to\infty$ we know that (\ref{hybridint}) is actually true for all $k_j\in\Rb^d\,(1\leq j\leq r)$. Then we simply integrate (\ref{hybridint}) in $k_r$, using the integral condition in the definition of $\Ac$, to get (\ref{admissible}).

From now on we shall assume $(n_r)_{\mathrm{in}}$ is admissible, then by Lemma \ref{hewsav}, we can find a unique measure $\zeta$ such that (\ref{hybridint}) holds. Consider the hybrid data $u_{\mathrm{in}}$ described in the statement of Theorem \ref{main4}. We can view it as obtained by first randomly selecting $m\in\Ac$ with law given by $\zeta$, then working in the same setting as in Theorems \ref{main0} and \ref{main}--\ref{main3} with the particular choice $n_{\mathrm{in}}=m$. Since $m\in\Ac$, by Theorem \ref{main0} we know that the conditional probability $\Pb(E|m)\geq 1-L^{-A}$ for any $m\in\Ac$, where $E$ is the event that (\ref{nls}) has a smooth solution up to time $T$. This implies that
\[\Pb(E)=\int_{\Ac}\Pb(E|m)\,\mathrm{d}\zeta(m)\geq 1-L^{-A}.\]

Finally we prove (\ref{approximationwkh}). As is clear from the proof of Theorem \ref{main2}, the remainder in (\ref{approximation2new}) is in fact $O(L^{-\nu})$ for some absolute constant $\nu>0$, where the implicit constant in $O(\cdot)$ may depend on $r$, but is uniform in $(t,k_j)$ and $n_{\mathrm{in}}=m$, as long as $m\in\Ac$. Thus, by Theorem \ref{main2}, for the hybrid data $u_{\mathrm{in}}$, for $t\in[0,T]$ we have that
\[\Eb\bigg(\prod_{j=1}^r|\widehat{u}(t,k_j)|^2\bigg)=\int_{\Ac}\Eb\bigg(\prod_{j=1}^r|\widehat{u}(t,k_j)|^2\,\bigg|\,m\bigg)\,\mathrm{d}\zeta(m)=\int_{\Ac}\prod_{j=1}^r\widetilde{m}\bigg(\frac{t}{T_{\mathrm{kin}}},k_j\bigg)\,\mathrm{d}\zeta(m)+O(L^{-\nu})\] where $\widetilde{m}(t,k)$ is the solution to (\ref{wke}) with initial data $\widetilde{m}(0,k)=m(k)$. It remains to show that
\[n_r(t,k_1,\cdots,k_r)=\int_{\Ac}\prod_{j=1}^r\widetilde{m}(t,k_j)\,\mathrm{d}\zeta(m).\] In fact, for any $m\in \Ac$, by definition $\widetilde{m}$ is the solution to (\ref{wke}) with initial data $m$, so by direct calculation we see that $(\widetilde{m})^{\otimes r}$ is a solution to (\ref{WKH}) with initial data $m^{\otimes r}$. Since (\ref{WKH}) is linear, we know that
\begin{equation}\label{acgsolution}\int_{\Ac}(\widetilde{m})^{\otimes r}\,\mathrm{d}\zeta(m)\end{equation} is a solution to (\ref{WKH}) with initial data
\[\int_{\Ac}m^{\otimes r}\,\mathrm{d}\zeta(m)=(n_r)_{\mathrm{in}},\] due to (\ref{hybridint}). Moreover, since $m\in\Ac$, for short time $t\in[0,\delta]$, the solution (\ref{acgsolution}) clearly belongs to the class $C_t^0\mathfrak{L}_{s,\epsilon}^\infty$ defined in \cite{RosStaff}; since short time uniqueness is proved in \cite{RosStaff} for solutions in this class, we know that (\ref{acgsolution}) has to equal $n_r$, which is the unique solution to (\ref{WKH}) in this class with initial data $(n_r)_{\mathrm{in}}$. The admissibility condition (\ref{admissible}) follows from the fact that $n_r$ equals (\ref{acgsolution}) which is an average of factorized solutions for which admissibility is clearly true {(Here we use that the \eqref{wke} conserves the total mass $\int n(t,\xi)\, d\xi$)}. This completes the proof.
\end{proof}

\end{document}